\documentclass[a4paper]{article}
\usepackage{amsmath,amsthm,amsfonts,amssymb,amscd,bm,makeidx,indentfirst,tocloft}

\newcommand{\e}{\ensuremath{\epsilon}}
\newcommand{\ve}{\ensuremath{\varepsilon}}
\newcommand{\rto}{\ensuremath{\rightarrow}}
\newcommand{\lem}{\ensuremath{\lesssim}}

\newtheorem{theorem}{Theorem}[section]
\newtheorem{definition}[theorem]{Definition}
\newtheorem{lemma}[theorem]{Lemma}

\newtheorem{remark}[theorem]{Remark}

\numberwithin{equation}{section}

\title{\Large Initial Layer and Relaxation Limit of Non-Isentropic Compressible Euler Equations with Damping}
\author{\normalsize Fuzhou Wu\thanks{E-mail: michael8723@gmail.com; fuzhou.wu@yahoo.com} \\
\small\it  Mathematical Sciences Center, Tsinghua University\\
\small\it  Beijing 100084, China
}
\date{\normalsize }

\begin{document}
\maketitle
\setlength\parindent{2em}
\setlength\parskip{5pt}

\begin{abstract}
\normalsize{
In this paper, we study the relaxation limit of the relaxing Cauchy problem for non-isentropic compressible Euler equations with damping in multi-dimensions. We prove that the velocity of the relaxing equations converges weakly to that of the relaxed equations, while other variables of the relaxing equations converges strongly to the corresponding variables of the relaxed equations. We show that as relaxation time approaches $0$, there exists an initial layer for the ill-prepared data, the convergence of the velocity is strong outside the layer; while there is no initial layer for the well-prepared data, the convergence of the velocity is strong near $t=0$.
}
\end{abstract}

\tableofcontents

\section{Introduction}
In this paper, we use $p, u, S, \varrho$ to denote the pressure, velocity, entropy and density of ideal gases respectively with the equation of state
\begin{equation*}
\begin{array}{ll}
\varrho = \varrho(p,S)
:= \frac{1}{\sqrt[\gamma]{A}}p^{\frac{1}{\gamma}}\exp\{-\frac{S}{\gamma}\},
\end{array}
\end{equation*}
where $A>0$, $\gamma=\frac{C_p}{C_V}>1$ are constants.
Assume the numbers $\bar{p},\bar{\varrho},\bar{S}$ satisfy $\bar{p}>0,\bar{\varrho}>0$, $\bar{p}=A\bar{\varrho}^{\gamma}e^{\bar{S}}$.
Then we study the relaxation limit of the relaxing Cauchy problem for 3D non-isentropic compressible Euler equations with damping:
\begin{equation}\label{Sect1_Relaxing_Eq}
\left\{\begin{array}{ll}
p_t + u\cdot\nabla p + \gamma p\nabla\cdot u = 0, \\[6pt]
\tau^2 u_t + \tau^2 u\cdot\nabla u + \frac{1}{\varrho}\nabla p + u =0, \\[6pt]
S_t + u\cdot\nabla S =0, \\[6pt]
(p,u,S)(x,0)=(p_0(x,\tau), \frac{u_0(x,\tau)}{\tau}, S_0(x,\tau)),
\end{array}\right.
\end{equation}
where $(p_0(x,\tau),u_0(x,\tau),S_0(x,\tau))$ are small perturbations of $(\bar{p},0,\bar{S})$. In \eqref{Sect1_Relaxing_Eq}, $(p,u,S,\varrho)\rto (\bar{p},0,\bar{S},\bar{\varrho})$ as $|x|\rto +\infty$.
$\tau$ is a small positive parameter representing the relaxation time, let $\tau\in (0,1]$. The density $\varrho$ satisfies the equation
$\varrho_t + u\cdot\nabla\varrho + \varrho\nabla\cdot u =0$ by \eqref{Sect1_Relaxing_Eq}, and
$\varrho_0(x,\tau)=\varrho(p_0(x,\tau),S_0(x,\tau))$ is small perturbation of $\bar{\varrho}=\varrho(\bar{p},\bar{S})$.

The equations $(\ref{Sect1_Relaxing_Eq})$ are derived from
\begin{equation}\label{Sect1_NonIsentropic_EulerEq}
\left\{\begin{array}{lll}
\hat{p}_{t^{\prime}} + \hat{u}\cdot\nabla \hat{p} + \gamma \hat{p}\nabla\cdot \hat{u} = 0, \\[6pt]
\hat{u}_{t^{\prime}} + \hat{u}\cdot\nabla \hat{u} + \frac{1}{\hat{\varrho}}\nabla \hat{p} + \frac{1}{\tau} \hat{u} =0, \\[6pt]
\hat{S}_{t^{\prime}} + \hat{u}\cdot\nabla \hat{S} =0, \\[6pt]
(\hat{p},\hat{u},\hat{S})(x,0)=(p_0(x,\tau), u_0(x,\tau), S_0(x,\tau)),
\end{array}\right.
\end{equation}
with the rescaling  of variables:
\begin{equation}\label{Sect1_Rescaling_Variables}
\begin{array}{lll}
t=\tau t^{\prime},
u(x,t) = \frac{\hat{u}}{\tau}(x,t^{\prime}),
p(x,t) = \hat{p}(x,t^{\prime}),
\varrho(x,t) = \hat{\varrho}(x,t^{\prime}),
S(x,t) = \hat{S}(x,t^{\prime}),
\end{array}
\end{equation}
then $(p,u,S,\varrho)$ satisfy the equations $(\ref{Sect1_Relaxing_Eq})$.

Let $\tau\rto 0$ in the relaxing equations $(\ref{Sect1_Relaxing_Eq})$, we formally obtain the following relaxed equations
\begin{equation}\label{Sect1_Relaxed_Eq}
\left\{\begin{array}{ll}
p_t + u\cdot\nabla p + \gamma p\nabla\cdot u = 0, \\[6pt]
\frac{1}{\varrho}\nabla p + u =0, \\[6pt]
S_t + u\cdot\nabla S =0, \\[6pt]
(p,S)(x,0)=(\lim\limits_{\tau\rto 0}p_0(x,\tau), \lim\limits_{\tau\rto 0}S_0(x,\tau)),
\end{array}\right.
\end{equation}
where $\varrho = \varrho(p,S)$.

Due to its fundamental importance in both application and nonlinear PDE theory, the relaxation limit problems have been attracting much attention. We survey there some results closely related to this paper.
\par
For the relaxing isothermal compressible Euler equations with damping:
\begin{equation}\label{Sect1_Isothermal_EulerEq}
\left\{\begin{array}{ll}
\varrho_t + \nabla\cdot(\varrho u)=0, \\[6pt]
\varrho u_t + \varrho u\cdot\nabla u+ \bar{\sigma}^2\nabla \varrho + \frac{1}{\tau}\varrho u =0,
\end{array}\right.
\end{equation}
where $\bar{\sigma}^2 = \mathcal{R}\theta_{\ast}$ is constant.
In the Sobolev space $H^s(\mathbb(R)^d),s\in \mathbb{Z},s>1+\frac{d}{2}$ (see \cite{Coulombel_Goudon_2007}),
$-\frac{\varrho u}{\tau}-\bar{\sigma}^2\nabla \varrho\rightharpoonup 0\ in\ \mathcal{D}^{\prime}(\mathbb{R}^d\times\mathbb{R}^{+})$,
$\varrho$ converges strongly to the solution of the heat equation in $C([0,T],H^{s^{\prime}}(B_r))$, where $0<s^{\prime}<s$, $B_r$ is a ball with radius $r$.
The results of \cite{Coulombel_Goudon_2007} was extended in \cite{Xu_Fang_2007} to more general Sobolev space of fractional order. In \cite{Junca_Rascle_2002}, $(\ref{Sect1_Isothermal_EulerEq})$ was studied in one dimension with BV large data away from vacuum, and it was proved in \cite{Junca_Rascle_2002} that  $\varrho$ converges strongly to the solution of the heat equation in $L^2(\mathbb{R}\times[0,T])$ (global in space) by using the stream function.

\par
For the relaxing isentropic compressible Euler equations with damping
\begin{equation}\label{Sect1_Isentropic_EulerEq}
\left\{\begin{array}{ll}
\varrho_t + \nabla\cdot(\varrho u)=0, \\[6pt]
\varrho u_t + \varrho u\cdot\nabla u +\nabla p + \frac{1}{\tau} \varrho u = 0,
\end{array}\right.
\end{equation}
where $p(\rho)=A\rho^{\gamma}$.
In the Sobolev space $H^s(\mathbb(R)^d),s\in \mathbb{Z},s>1+\frac{d}{2}$ (see \cite{Lin_Coulombel_2013}),
$-\frac{\varrho u}{\tau}-\nabla p\rightharpoonup 0\ in\ \mathcal{D}^{\prime}(\mathbb{R}^d\times\mathbb{R}^{+})$,
$\varrho$ converges strongly to the solution of the porous media equation in $C([0,T],H^{s^{\prime}}(B_r))$, where $0<s^{\prime}<s$.
In the Besov space $\mathcal{B}_{2,1}^{\sigma}(\mathbb{R}^d),\sigma=1+\frac{d}{2}$ (see \cite{Xu_Wang_2013})
and in the Chemin-Lerner space (see \cite{Xu_Kawashima_2014}), the density of $(\ref{Sect1_Isentropic_EulerEq})$ converges strongly to the solution of the porous media equation.

Relaxation limit problem also appears in Euler-Poisson equations, see \cite{Marcati_Natalini_1995,Jungel_Peng_1999,Lattanzio_Marcati_1999,Hsiao_Zhang_2000} for weak solutions and \cite{Chen_Jerome_Zhang_1997,Xu_2009} for smooth solutions. It has been proved that the current density, which is the product of the electron density and electron velocity, converges weakly to that of the drift-diffusion model. If the initial data are well-prepared, the current density converges strongly to that of the drift-diffusion model (see \cite{Yong_2004}).
If the initial data are ill-prepared, the authors (see \cite{Nishibata_Suzuki_2010}) proved the difference between the current density of 1D hydrodynamic model and that of the drift-diffusion model decays exponentially fast in the large time interval $[0,\frac{1}{\beta}\log(\frac{1}{\tau^{\lambda}})]$ with $\lambda\in (0,1),\beta>0$. The key of the proof in \cite{Nishibata_Suzuki_2010} is that the solutions of the relaxing and relaxed equations converge to the corresponding stationary solutions exponentially fast while both stationary solutions are close to each other.
As to the relaxation limit of weak solutions (see \cite{Gasser_Natalini_1999}) and classical solutions (see \cite{Ali_Bini_Rionero_2000,Xu_Yong_2009,Li_2007}) to non-isentropic Euler-Poisson equations, the current density converges weakly to that of the energy-transportation model or drift-diffusion model.

However, there have been no rigorous analysis of the initial layer and strong convergence of the velocity for the ill-prepared data in the above mentioned papers. A main distinction of results in this paper is that we give results on the initial layer and strong convergence of the velocity. Our main concern is the non-isentropic flow $(\ref{Sect1_Relaxing_Eq})$, but our results are valid for the isentropic flow $(\ref{Sect1_Isentropic_EulerEq})$ and isothermal flow (\ref{Sect1_Isothermal_EulerEq}) (assuming no vacuum). We show that for the ill-prepared initial data, there exists an initial layer $[0,t^{\ast}]$ for the velocity, where $0<t^{\ast}<C\tau^{2-\delta},\delta>0$. Outside the initial layer, the velocity of the relaxing equations converge strongly to that of the relaxed equations. Only for the well-prepared initial data, there is no initial layer, the convergence of the velocity is strong near $t=0$. The key of our analysis in this paper is uniform a priori estimates with respect to $\tau$ and pointwise decay of the quantity $u+\frac{1}{\varrho}\nabla p$.
Also, the methods in this paper can be applied to the relaxation limit problems for Euler-Poisson equations and Euler-Maxwell equations.

The first result in this paper is the following convergence result:
\begin{theorem}\label{Sect1_Thm1}
Suppose that the initial data for the relaxing Cauchy problem $(\ref{Sect1_Relaxing_Eq})$ satisfy $(p_0(x,\tau),u_0(x,\tau),S_0(x,\tau))\in H^4(\mathbb{R}^3)$,  $\inf\limits_{x\in\mathbb{R}^3}\lim\limits_{\tau\rto 0}p_0(x,\tau)>0$, $\inf\limits_{x\in\mathbb{R}^3} p_0(x,\tau)>0$,
$\|(p_0(x,\tau)-\bar{p},u_0(x,\tau),S_0(x,\tau)-\bar{S})\|_{H^4(\mathbb{R}^3)}$ is sufficiently small for some constants $\bar p>0$ and $\bar S$. Then for any finite $T>0$,  the problem $(\ref{Sect1_Relaxing_Eq})$ admits a unique solution $(p,u,S,\varrho)$ in $[0 , T]$ satisfying
\begin{equation}\label{Sect1_Uniform_Regularity}
\begin{array}{ll}
\partial_t^{\ell}(p-\bar{p}),\ \tau \partial_t^{\ell}u,\ \partial_t^{\ell}(S-\bar{S}),\ 
\partial_t^{\ell}(\varrho-\bar{\varrho})
\in L^{\infty}([0,T],H^{4-\ell}(\mathbb{R}^3)),0\leq\ell\leq 2,\\[6pt]
p-\bar{p},\ u,\ S-\bar{S},\ \varrho-\bar{\varrho} \in \underset{0\leq \ell\leq 2}{\cap}H^{\ell}([0,T],H^{4-\ell}(\mathbb{R}^3)),
\end{array}
\end{equation}
such that as $\tau\rto 0$,
\begin{equation}\label{Sect1_Convergence_Result}
\begin{array}{ll}
p \rto \tilde{p}\quad in\ C([0,T],C^{2+\mu_1}(\mathcal{K})\cap W^{3,\mu_2}(\mathcal{K})),
\mu_1\in (0,\frac{1}{2}),2\leq \mu_2<6, \\[6pt]
S \rto \tilde{S}\quad in\ C([0,T],C^{2+\mu_1}(\mathcal{K})\cap W^{3,\mu_2}(\mathcal{K})),
\mu_1\in (0,\frac{1}{2}),2\leq \mu_2<6, \\[6pt]
\varrho \rto \tilde{\varrho}\quad in\ C([0,T],C^{2+\mu_1}(\mathcal{K})\cap W^{3,\mu_2}(\mathcal{K})),
\mu_1\in (0,\frac{1}{2}),2\leq \mu_2<6, \\[6pt]
u \rightharpoonup \tilde{u}\quad in\ \underset{0\leq \ell\leq 2}{\cap}H^\ell([0,T],H^{4-\ell}(\mathbb{R}^3)), 
\end{array}
\end{equation}
for some function $(\tilde{p}, \tilde{S}, \tilde{\varrho}, \tilde{u})$ which is a weak solution to the relaxed equations $(\ref{Sect1_Relaxed_Eq})$ with the data $(\lim\limits_{\tau\rto 0}p_0(x,\tau),\lim\limits_{\tau\rto 0}S_0(x,\tau))$,
where $\mathcal{K}$ denotes any compact subset of $\mathbb{R}^3$. 
$(\tilde{p}, \tilde{S}, \tilde{\varrho}, \tilde{u})$ is the classical solution to $(\ref{Sect1_Relaxed_Eq})$,
if $(\lim\limits_{\tau\rto 0}p_0(x,\tau)-\bar{p},
\lim\limits_{\tau\rto 0}S_0(x,\tau)-\bar{S})\in H^5(\mathbb{R}^3)\times H^4(\mathbb{R}^3)$ is satisfied.
\end{theorem}
The main results concerned with the initial layer and strong convergence of the velocity are stated in the following Theorem:
\begin{theorem}\label{Sect1_Thm2}
Let $(p,u,S,\varrho)$ and $(\tilde{p},\tilde{u},\tilde{S},\tilde{\varrho})$ be the solutions obtained in  Theorem $\ref{Sect1_Thm1}$ in $[0, T]$ and $\mathcal{K}$ denotes any compact subset of $\mathbb{R}^3$. Then it holds that,  for the ill-prepared data, i.e., $\lim\limits_{\tau\rto 0}\left|\frac{1}{\tau}u_0(x,\tau) + \frac{1}{\varrho_0(x,\tau)}\nabla p_0(x,\tau)\right|_{\infty} \neq 0$,
there exists an initial layer $[0, t^{\ast}]$ with $0<t^{\ast}< C\tau^{2-\delta}$ for the velocity $u$, where $C>0$, $\delta>0$, such that as $\tau\rto 0$,
\begin{equation}\label{Sect1_Strong_Convergence_Result1}
\begin{array}{ll}
u \rto \tilde{u}\quad in\ C((0, T],\ C^{0+\mu_1}(\mathcal{K})\cap W^{1,\mu_2}(\mathcal{K})),
\mu_1\in (0,\frac{1}{2}),2\leq \mu_2<6;
\end{array}
\end{equation}
for the well-prepared data, i.e., $\lim\limits_{\tau\rto 0}\left\|\frac{1}{\tau}u_0(x,\tau) + \frac{1}{\varrho_0(x,\tau)}\nabla p_0(x,\tau)\right\|_{H^2(\mathbb{R}^3)} =0$, as $\tau\rto 0$,
\begin{equation}\label{Sect1_Strong_Convergence_Result2}
\begin{array}{ll}
u \rto \tilde{u}\quad in\ C([0,T],C^{0+\mu_1}(\mathcal{K})\cap W^{1,\mu_2}(\mathcal{K})),
\mu_1\in (0,\frac{1}{2}),2\leq \mu_2<6.
\end{array}
\end{equation}
\end{theorem}

\begin{remark}
(i) Compare $(\ref{Sect1_Strong_Convergence_Result1})$ with $(\ref{Sect1_Strong_Convergence_Result2})$. In $(\ref{Sect1_Strong_Convergence_Result1})$, $(0,T]$ can not be replaced by $[0,T]$. That is, the convergence is not uniform near $t=0$, since the initial layer develops for the ill-prepared data.

(ii) For any fixed $t_{\ast}\in (0,T)$,
\begin{equation}\label{Sect1_Strong_Convergence_Result3}
\begin{array}{ll}
u \rto \tilde{u}\quad in\ C([t_{\ast}, T],\ C^{0+\mu_1}(\mathcal{K})\cap W^{1,\mu_2}(\mathcal{K})),
\mu_1\in (0,\frac{1}{2}),2\leq \mu_2<6,
\end{array}
\end{equation}
while $(\ref{Sect1_Strong_Convergence_Result3})$ is equivalent to $(\ref{Sect1_Strong_Convergence_Result1})$ due to the arbitrariness of $t_{\ast}$.

(iii) If one replaces the initial data $(p_0(x,\tau),u_0(x,\tau),S_0(x,\tau),\varrho_0(x,\tau))$ with $(p_0(x),u_0(x),S_0(x),\varrho_0(x))$ which are independent of $\tau$, then for the well-prepared data,
$p_0(x)\equiv const$ and $u_0(x)\equiv 0$ (equilibrium states); while for the ill-prepared data, $p_0(x)\neq const$ or $u_0(x)\neq 0$ (non-equilibrium states).
\end{remark}

In the following, we give more comments on Theorem \ref{Sect1_Thm1} and Theorem \ref{Sect1_Thm2}. If $\tau>0$ is fixed, the global existence of classical solutions to the equations $(\ref{Sect1_NonIsentropic_EulerEq})$
is proved in \cite{Wu_2014,Wu_Tan_Huang_2013,Zhang_Wu_2014}. However, for the relaxation limit problem in this paper,
$\tau>0$ is variant and approaches $0$, so we need the uniform existence of the solutions and uniform regularities $(\ref{Sect1_Uniform_Regularity})$ which are different from \cite{Wu_2014,Wu_Tan_Huang_2013,Zhang_Wu_2014}.
The uniform a priori estimates with respect to $\tau$ produce the convergence results. To make sure that the solutions of the relaxed equations remain classical and the estimates related to the initial layer can be proceeded, $(p_0(x,\tau),u_0(x,\tau),S_0(x,\tau)$ are required to be in $H^4(\mathbb{R}^3)$.

The uniform a priori estimates for the equations $(\ref{Sect1_Relaxing_Eq})$ imply the uniform regularities $(\ref{Sect1_Uniform_Regularity})$. Passing to the limit, we have the convergence results $(\ref{Sect1_Convergence_Result})$.

Let us give some comments and remarks on the initial layer as follows. The asymptotic expansions of the solutions to $(\ref{Sect1_Relaxing_Eq})$ give us some indication of the initial layer.  We illustrate this as follows: if $u_0(x,\tau)=O(\tau)$, we assume that the initial data have asymptotic expansion $$(p_0(x,\tau),\frac{u_0(x,\tau)}{\tau},S_0(x,\tau),\varrho_0(x,\tau))
=\sum\limits_{m\geq 0} \tau^{2m}(p_0^m, u_0^m, S_0^m, \varrho_0^m),$$
and solutions of the equations $(\ref{Sect1_Relaxing_Eq})$ have the asymptotic expansion
$$(p(x,t),u(x,t),S(x,t),\varrho(x,t))
=\sum\limits_{m\geq 0} \tau^{2m}(p^m(x,t), u^m(x,t), S^m(x,t), \varrho^m(x,t)),$$
 then the leading order profiles satisfy the equations
\begin{equation}\label{Sect2_Profile1}
\left\{\begin{array}{ll}
\partial_t p^0 + u^0 \cdot\nabla p^0 + \gamma p^0\nabla\cdot u^0 =0, \\[6pt]
u^0 + \frac{1}{\varrho_0}\nabla p^0 =0, \\[6pt]
\partial_t S^0 + u^0\cdot\nabla S^0 =0, \\[6pt]
(p^0, S^0)(x,0) = (p_0^0,S_0^0),
\end{array}\right.
\end{equation}
if the initial velocity is well-prepared, i.e.,  $u_0^0 = -\frac{1}{\varrho_0^0}\nabla p_0^0,$
where $\varrho^0 = \varrho(p^0,S^0)$. Note that in this case,  $u^0 + \frac{1}{\varrho^0}\nabla p^0 =0$ matches
$u_0^0 + \frac{1}{\varrho_0^0}\nabla p_0^0 =0$.

However, if the initial velocity is ill-prepared, i. e.,  $u_0^0 \neq -\frac{1}{\varrho_0^0}\nabla p_0^0$, we may assume that
the velocity has the asymptotic expansion
$$u(x,t)=\sum\limits_{m\geq 0}
( u^m(x,t)+\hat{u}^m(x,z)),\ z=\frac{t}{\tau^2},$$
then the leading order profile of the initial layer correction $\hat{u}$ satisfies the equation
\begin{equation}\label{Sect2_Profile2}
\left\{\begin{array}{ll}
\partial_z\hat{u}^0 + \hat{u}^0 =0, \\[6pt]
\hat{u}^0(x,0) = u_0^0 +\frac{1}{\varrho_0^0}\nabla p_0^0.
\end{array}\right.
\end{equation}
Then $\hat{u}(x,z)= (u_0^0 +\frac{1}{\varrho_0^0}\nabla p_0^0) e^{-z}.$
Thus, the difference between $u^0$ and $-\frac{1}{\varrho^0}\nabla p^0$ decays exponentially within the initial layer.

The above arguments of $(\ref{Sect2_Profile1})$ and $(\ref{Sect2_Profile2})$ indicate the relationship between the existence of initial layer and a class of initial data on a  formal level. We are not concerned with the  asymptotic expansions in this paper (as to asymptotic expansion analysis in relaxation limit problem for Euler-Poisson equations, see \cite{Hajjej_Peng_2013,Wang_Xu_2014,Li_2009}; for Euler-Maxwell equations, see \cite{Peng_Wang_Gu_2011,Hajjej_Peng_2012}) and only focus on rigorous analysis of the initial layer and relaxation limit of the relaxing equations $(\ref{Sect1_Relaxing_Eq})$.

The relaxation limit is a singular limit, since $(p, -\frac{1}{\varrho}\nabla p,S,\varrho)$ converge to $(\tilde{p},\tilde{u},\tilde{S},\tilde{\varrho})$, instead of $u\rto \tilde{u}$. In order to measure the difference between $u$ and $-\frac{1}{\varrho}\nabla p$, we introduce a quantity:
\begin{equation}\label{Sect1_Relaxation_Wave}
\begin{array}{ll}
\eta = \frac{1}{k_1}(u+\frac{1}{\varrho}\nabla p),
\end{array}
\end{equation}
where $k_1>0$ is a constant defined in Section 2.

The pointwise decay of $\eta$ outside the initial layer is the key to the strong convergence of the velocity. $\eta$ satisfies the following transportation equations with damping and forcing
\begin{equation}\label{Sect1_Wave_Eq}
\eta_t + u\cdot\nabla\eta + \frac{1}{\tau^2}\eta = f{\!}orcing\ terms,
\end{equation}
where $\tau\cdot[f\!orcing\ terms]$ is bounded uniform with respect to $\tau$, the damping effect becomes stronger as $\tau$ decreases.

Also, $\eta$ satisfies another equation:
\begin{equation}\label{Sect1_Eta_Velocity}
\begin{array}{ll}
\eta = \frac{\tau^2}{k_1}(u_t + u\cdot\nabla u).
\end{array}
\end{equation}

Then the equations $(\ref{Sect1_Wave_Eq})$ and $(\ref{Sect1_Eta_Velocity})$ produce the following estimates respectively:
\begin{equation}\label{Sect1_Wave_Estimate}
\left\{\begin{array}{ll}
|\eta|_{\infty}^2\leq C\|\eta\|_{H^2(\mathbb{R}^3)}^2\leq C\|\eta|_{t=0}\|_{H^2(\mathbb{R}^3)}^2\exp\{-\frac{t}{\tau^2}\}+C\tau^2, \\[7pt]

\int\limits_0^T \|\eta_t\|_{H^1(\mathbb{R}^3)}^2 \,\mathrm{d}x \leq C\tau^2.
\end{array}\right.
\end{equation}

Therefore, for the well-prepared data, $\lim\limits_{\tau\rto 0}\eta(t)=0$ for any $t\in [0,T]$, there is no discrepancy between $\eta(t)|_{t>0}$ and $\eta|_{t=0}$, and thus there is no initial layer,  $u\rto \tilde{u}$ strongly in
$C([0,T],C^{0+\mu_1}(\mathcal{K})\cap W^{1,\mu_2}(\mathcal{K})),
\mu_1\in (0,\frac{1}{2}),2\leq \mu_2<6.$

While for the ill-prepared data, $\lim\limits_{\tau\rto 0} \eta|_{t=0}\neq 0$ while
$\lim\limits_{\tau\rto 0} \eta(t)=0$ for any $t\in [t^{\ast},T]$ where $t^{\ast}= C\tau^{2-\delta}$.
The discrepancy between $\eta|_{t=0}$ and $\eta(t)|_{t=t^{\ast}}$ make the initial layer exist. Within the layer $[0,t^{\ast}]$, $\eta$ decreases rapidly, $\frac{1}{\varrho}\nabla p$ is uniform bounded, then $u$ changes dramatically. Outside the layer, $u$ converges strongly to $\tilde{u}$.
Namely, as $\tau\rto 0$, $u\rto \tilde{u}$ strongly in
$C((0,T],C^{0+\mu_1}(\mathcal{K})\cap W^{1,\mu_2}(\mathcal{K})),
\mu_1\in (0,\frac{1}{2}),2\leq \mu_2<6.$

Finally, the main results of this paper can be extended to periodic domains. Due to the convenience of periodic boundary conditions, a priori estimates for $\mathbb{T}^3$ are similar to those for $\mathbb{R}^3$, the results for $\mathbb{T}^3$ are stated as follows:

Suppose that the initial data for $(\ref{Sect1_Relaxing_Eq})$ satisfy $(p_0(x,\tau),u_0(x,\tau),S_0(x,\tau))\in H^4(\mathbb{T}^3)$ and other similar conditions given in Theorem $\ref{Sect1_Thm1}$. Then for any finite $T>0$,  the problem $(\ref{Sect1_Relaxing_Eq})$ admits a unique solution $(p,u,S,\varrho)$ in $[0 , T]$
such that as $\tau\rto 0$,
\begin{equation*}
\begin{array}{ll}
(p,S,\varrho) \rto (\tilde{p},\tilde{S},\tilde{\varrho})\ in\ C([0,T],C^{2+\mu_1}(\mathbb{T}^3)\cap W^{3,\mu_2}(\mathbb{T}^3)),
\mu_1\in (0,\frac{1}{2}),2\leq \mu_2<6, \\[6pt]
u \rightharpoonup \tilde{u}\ in\ \underset{0\leq \ell\leq 2}{\cap}H^\ell([0,T],H^{4-\ell}(\mathbb{T}^3)),
\end{array}
\end{equation*}
for some function $(\tilde{p}, \tilde{S}, \tilde{\varrho}, \tilde{u})$ which is a weak solution to $(\ref{Sect1_Relaxed_Eq})$ with the data $(\lim\limits_{\tau\rto 0}p_0(x,\tau)-\bar{p},
\lim\limits_{\tau\rto 0}S_0(x,\tau)-\bar{S})$. $(\tilde{p}, \tilde{S}, \tilde{\varrho}, \tilde{u})$ is the classical solution to $(\ref{Sect1_Relaxed_Eq})$, if the data is in $H^4(\mathbb{T}^3)\times H^3(\mathbb{T}^3)$.

For the ill-prepared data, there exists an initial layer for the velocity $u$, such that as $\tau\rto 0$,
\begin{equation*}
\begin{array}{ll}
u \rto \tilde{u}\quad in\ C((0, T],\ C^{0+\mu_1}(\mathbb{T}^3)\cap W^{1,\mu_2}(\mathbb{T}^3)),
\mu_1\in (0,\frac{1}{2}),2\leq \mu_2<6;
\end{array}
\end{equation*}
for the well-prepared data, as $\tau\rto 0$,
\begin{equation*}
\begin{array}{ll}
u \rto \tilde{u}\quad in\ C([0,T],C^{0+\mu_1}(\mathbb{T}^3)\cap W^{1,\mu_2}(\mathbb{T}^3)),
\mu_1\in (0,\frac{1}{2}),2\leq \mu_2<6.
\end{array}
\end{equation*}

\vspace{0.4cm}
The rest of this paper is organized as follows: In Section 2, we reformulate the equations into appropriate forms and derive the equations of $\eta$. In Section 3, we prove a priori estimates in $[0,T]$ for the relaxed equations. In Section 4, we prove the existence in $[0,T]$ of classical solutions to the relaxed equations. In Section 5, we prove uniform a priori estimates for the relaxing equations. In Section 6, we discuss the relaxation limit and initial layer of the relaxing equations. In Section 7, we extend the main results of this paper to periodic domains.

\vspace{0.4cm}
\section{Preliminaries}
In this section, we will reformulate the equations $(\ref{Sect1_Relaxing_Eq})$ into appropriate forms, define the energy quantities and derive the equations of the quantity $\eta$.

For the relaxing equations $(\ref{Sect1_Relaxing_Eq})$ together with their initial data $(p_0,u_0,S_0,\varrho_0)$ and constants $\bar{p},\bar{S},\bar{\varrho}$, we introduce the constants:
\begin{equation*}
\begin{array}{ll}
k_1=\sqrt{\frac{1}{\gamma\bar{\varrho}\bar{p}}},\quad k_2 = \sqrt{\frac{\gamma\bar{p}}{\bar{\varrho}}},
\end{array}
\end{equation*}
define the variables:
\begin{equation*}
\begin{array}{ll}
\xi=p-\bar{p},\ \phi=S-\bar{S},\ \zeta=\varrho-\bar{\varrho},\ v=\frac{1}{k_1}u.
\end{array}
\end{equation*}

In order to have the solutions of the relaxing equations $(\ref{Sect1_Relaxing_Eq})$ in any finite time interval $[0,T]$ and obtain uniform a priori estimates, we reformulate the equations $(\ref{Sect1_Relaxing_Eq})$ into the following form:
\begin{equation}\label{Sect2_Relaxing_Eq}
\left\{\begin{array}{ll}
\xi_t + k_2\nabla\cdot v = -\gamma k_1\xi\nabla\cdot v - k_1 v\cdot\nabla \xi, \\[6pt]
\tau^2 v_t + k_2 \nabla\xi + v = - k_1 \tau^2 v\cdot\nabla v
+ \frac{1}{k_1} (\frac{1}{\bar{\varrho}}- \frac{1}{\varrho})\nabla\xi, \\[6pt]
\phi_t = - k_1 v\cdot\nabla\phi, \\[6pt]
(\xi, v, \phi)(x,0)=(p_0(x,\tau)-\bar{p}, \frac{1}{k_1\tau} u_0(x,\tau), S_0(x,\tau)-\bar{S}), \\[6pt]
\end{array}\right.
\end{equation}
where $\varrho = \zeta+\bar{\varrho} =\varrho(\xi+\bar{p},\phi+\bar{S})$ and
$\zeta$ satisfies the equation $\zeta_t + k_1v\cdot\nabla\zeta + k_1\varrho\nabla\cdot v=0$ by $(\ref{Sect2_Relaxing_Eq})$. $(\xi,v,\phi,\zeta)\rto (0,0,0,0)$ as $|x|\rto +\infty$.

Let $\tau=0$ in the relaxing equations $(\ref{Sect2_Relaxing_Eq})$, we formally obtained the following relaxed equations, which are equivalent to the relaxed equations $(\ref{Sect1_Relaxed_Eq})$.
\begin{equation}\label{Sect2_Relaxed_Eq}
\left\{\begin{array}{ll}
\xi_t + k_2\nabla\cdot v = -\gamma k_1\xi\nabla\cdot v - k_1 v\cdot\nabla \xi, \\[6pt]
k_2 \nabla\xi + v = \frac{1}{k_1} (\frac{1}{\bar{\varrho}}- \frac{1}{\varrho})\nabla\xi, \\[6pt]
\phi_t = - k_1 v\cdot\nabla\phi, \\[6pt]
(\xi,\phi)(x,0)=(\lim\limits_{\tau\rto 0}p_0(x,\tau)-\bar{p}, \lim\limits_{\tau\rto 0}S_0(x,\tau)-\bar{S}),
\end{array}\right.
\end{equation}
where $\varrho =\zeta +\bar{\varrho} =\varrho(\xi+\bar{p},\phi+\bar{S})$.

The global existence of classical solutions to IBVP $(\ref{Sect2_Relaxed_Eq})$ with boundary condition $v\cdot n|_{\partial\Omega}=0$ has been proved in \cite{Wu_2014}. However, in this paper, we prove the existence in any finite time interval $[0,T]$ of classical solutions to Cauchy problem $(\ref{Sect2_Relaxed_Eq})$ with small data
in $H^5(\mathbb{R}^3)\times H^4(\mathbb{R}^3)$. Note that the regularity index for $\mathbb{R}^3$ is one order higher than that for periodic or bounded domains.

In order to prove the existence in any finite time interval $[0,T]$ of classical solutions to $(\ref{Sect2_Relaxed_Eq})$ via energy method, we define the following energy quantities:
\begin{definition}\label{Sect2_Def_Energy_RelaxedEq} Define
\begin{equation}\label{Sect2_Energy_Define_RelaxedEq}
\begin{array}{ll}
\mathcal{F}[\xi](t):=
\sum\limits_{0\leq|\alpha|\leq 4} \|\mathcal{D}^{\alpha}\xi(t)\|_{L^2(\mathbb{R}^3)}^2
+ \sum\limits_{0\leq|\alpha|\leq 2} \|\mathcal{D}^{\alpha}\partial_t\xi(t)\|_{L^2(\mathbb{R}^3)}^2,
\\[12pt]

\tilde{\mathcal{F}}[\xi](t):=
\sum\limits_{0\leq|\alpha|\leq 5} \|\mathcal{D}^{\alpha}\xi(t)\|_{L^2(\mathbb{R}^3)}^2
+ \sum\limits_{0\leq|\alpha|\leq 3} \|\mathcal{D}^{\alpha}\partial_t\xi(t)\|_{L^2(\mathbb{R}^3)}^2,
\\[12pt]

\mathcal{F}_X[\xi](t):=
\sum\limits_{1\leq|\alpha|\leq 4} \|\mathcal{D}^{\alpha}\xi(t)\|_{L^2(\mathbb{R}^3)}^2
+ \sum\limits_{0\leq|\alpha|\leq 2} \|\mathcal{D}^{\alpha}\partial_t\xi(t)\|_{L^2(\mathbb{R}^3)}^2,
\\[15pt]

\mathcal{F}_1[\xi](t) = \sum\limits_{0\leq|\alpha|\leq 4} \|\mathcal{D}^{\alpha}\xi(t)\|_{L^2(\mathbb{R}^3)}^2
-\sum\limits_{|\alpha|=4}\
\int\limits_{\mathbb{R}^3}\frac{\xi}{p} (\mathcal{D}^{\alpha}\xi)^2\,\mathrm{d}x
\\[12pt]\hspace{1.8cm}
+ \sum\limits_{0\leq|\alpha|\leq 2} \|\mathcal{D}^{\alpha}\partial_t\xi(t)\|_{L^2(\mathbb{R}^3)}^2 
-\sum\limits_{|\alpha|=2}\
\int\limits_{\mathbb{R}^3}\frac{\xi}{p} (\mathcal{D}^{\alpha}\partial_t\xi)^2\,\mathrm{d}x, \\[15pt]

\mathcal{F}[v](t) = \sum\limits_{0\leq|\alpha|\leq 4} \|\mathcal{D}^{\alpha}v(t)\|_{L^2(\mathbb{R}^3)}^2
+ \sum\limits_{0\leq|\alpha|\leq 2} \|\mathcal{D}^{\alpha}\partial_t v(t)\|_{L^2(\mathbb{R}^3)}^2,
\\[13pt]

\mathcal{F}[\phi](t):=
\sum\limits_{0\leq|\alpha|\leq 4} \|\mathcal{D}^{\alpha}\phi(t)\|_{L^2(\mathbb{R}^3)}^2
+ \sum\limits_{0\leq|\alpha|\leq 2} \|\mathcal{D}^{\alpha}\partial_t\phi(t)\|_{L^2(\mathbb{R}^3)}^2,\\[12pt]
\mathcal{F}[\zeta](t):=
\sum\limits_{0\leq|\alpha|\leq 4} \|\mathcal{D}^{\alpha}\zeta(t)\|_{L^2(\mathbb{R}^3)}^2
+ \sum\limits_{0\leq|\alpha|\leq 2} \|\mathcal{D}^{\alpha}\partial_t\zeta(t)\|_{L^2(\mathbb{R}^3)}^2,\\[12pt]

\mathcal{F}[\xi,v,\phi,\zeta](t):=\mathcal{F}[\xi](t) +\mathcal{F}[v](t)
 + \mathcal{F}[\phi](t) + \mathcal{F}[\zeta](t).
\end{array}
\end{equation}
\end{definition}

Moreover, we can derive the evolution equations of $v$ from $(\ref{Sect2_Relaxed_Eq})$, which is useful for proving a priori estimate for the $L^{\infty}$ bound of $\mathcal{F}[v](t)$.
Apply $\partial_{i}$ to $(\ref{Sect2_Relaxed_Eq})_1$ and then substitute $\partial_i \xi$ for $- k_1\varrho v_i$, we have
\begin{equation}\label{Sect2_Velocity_Solve}
\begin{array}{ll}
v_t = k_1(1-\gamma)v(\nabla\cdot v) - k_1 v\cdot\nabla v
- \frac{k_1}{2} \nabla(|v|^2) + \frac{\gamma p}{\varrho}\nabla(\nabla\cdot v).
\end{array}
\end{equation}

To prove the uniform regularities $(\ref{Sect1_Uniform_Regularity})$ is equivalent to prove the following uniform regularities:
\begin{equation}\label{Sect2_Solution_Regularity}
\begin{array}{ll}
\partial_t^{\ell}\xi,\ \tau \partial_t^{\ell}v,\ \partial_t^{\ell}\phi,\ \partial_t^{\ell}\zeta
\in L^{\infty}([0,T],H^{4-\ell}
(\mathbb{R}^3)), 0\leq\ell\leq 2, \\[6pt]

\xi,\ v,\ \phi,\ \zeta \in \underset{0\leq \ell\leq 2}{\cap}H^\ell([0,T],H^{4-\ell}(\mathbb{R}^3)).
\end{array}
\end{equation}

In order to use the energy method to derive uniform a priori estimates with respect to $\tau$, we define the following energy quantities:
\begin{definition}\label{Sect2_Def_Energy} Define
\begin{equation}\label{Sect2_Energy_Define}
\begin{array}{ll}
\mathcal{E}[\xi](t):=
\sum\limits_{0\leq\ell\leq 2,0\leq \ell+|\alpha|\leq 4} \|\partial_t^{\ell} \mathcal{D}^{\alpha}\xi(t)\|_{L^2(\mathbb{R}^3)}^2,\\[12pt]
\mathcal{E}_X[\xi](t)=\sum\limits_{0\leq\ell\leq 2,0<\ell+|\alpha|\leq 4}
\|\partial_t^{\ell}\mathcal{D}^{\alpha}\xi\|_{L^2(\mathbb{R}^3)}^2,\\[12pt]
\mathcal{E}_1[\xi](t) = \mathcal{E}[\xi](t)
-\sum\limits_{0\leq\ell\leq 2,\ell+|\alpha|=4}\
\int\limits_{\mathbb{R}^3}\frac{\xi}{p} (\partial_t^{\ell}\mathcal{D}^{\alpha}\xi)^2\,\mathrm{d}x, \\[15pt]

\mathcal{E}[v](t):=
\sum\limits_{0\leq\ell\leq 2,0\leq \ell+|\alpha|\leq 4} \|\partial_t^{\ell} \mathcal{D}^{\alpha} v(t)\|_{L^2(\mathbb{R}^3)}^2, \\[13pt]
\mathcal{E}_1[v](t) = \mathcal{E}[v](t)
+\sum\limits_{0\leq\ell\leq 2,\ell+|\alpha|=4}\
\int\limits_{\mathbb{R}^3}(\frac{\varrho}{\bar{\varrho}}- 1)|\partial_t^{\ell}\mathcal{D}^{\alpha}v|^2 \,\mathrm{d}x,
\\[15pt]

\mathcal{E}[\phi](t):=
\sum\limits_{0\leq\ell\leq 2,0\leq \ell+|\alpha|\leq 4} \|\partial_t^{\ell} \mathcal{D}^{\alpha}\phi(t)\|_{L^2(\mathbb{R}^3)}^2,\\[13pt]
\mathcal{E}[\zeta](t):=
\sum\limits_{0\leq\ell\leq 2,0\leq \ell+|\alpha|\leq 4} \|\partial_t^{\ell} \mathcal{D}^{\alpha}\zeta(t)\|_{L^2(\mathbb{R}^3)}^2,\\[12pt]
\mathcal{E}[\xi,v](t):=\mathcal{E}[\xi](t) + \mathcal{E}[v](t),\\[6pt]
\mathcal{E}[\xi,v,\phi,\zeta](t):=\mathcal{E}[\xi,v](t) + \mathcal{E}[\phi](t) + \mathcal{E}[\zeta](t).
\end{array}
\end{equation}
\end{definition}

To prove the convergence results $(\ref{Sect1_Convergence_Result})$ is equivalent to prove the following convergence results: as $\tau\rto 0$,
\begin{equation}\label{Sect2_Convergence_Result}
\begin{array}{ll}
\xi \rto \tilde{\xi}\quad in\ C([0,T],C^{2+\mu_1}(\mathcal{K})\cap W^{3,\mu_2}(\mathcal{K})),
\mu_1\in (0,\frac{1}{2}),2\leq \mu_2<6, \\[6pt]
\phi \rto \tilde{\phi}\quad in\ C([0,T],C^{2+\mu_1}(\mathcal{K})\cap W^{3,\mu_2}(\mathcal{K})),
\mu_1\in (0,\frac{1}{2}),2\leq \mu_2<6, \\[6pt]
\zeta \rto \tilde{\zeta} \quad in\ C([0,T],C^{2+\mu_1}(\mathcal{K})\cap W^{3,\mu_2}(\mathcal{K})),
\mu_1\in (0,\frac{1}{2}),2\leq \mu_2<6, \\[6pt]
v \rightharpoonup \tilde{v} \quad in\ \underset{0\leq \ell\leq 2}{\cap}H^\ell([0,T],H^{4-\ell}(\mathbb{R}^3)),
\end{array}
\end{equation}
where $\mathcal{K}$ denotes any compact subset of $\mathbb{R}^3$.

For the ill-prepared data, $\lim\limits_{\tau\rto 0}\eta|_{t=0}\neq 0$, while $\lim\limits_{\tau\rto 0}\eta(t)|_{t\geq t_{\ast}}= 0$ for any small $t_{\ast}>0$. This discrepancy between $\eta|_{t=0}$ and $\eta(t)|_{t=t_{\ast}}$ makes the initial layer for the velocity generate.

In order to prove the strong convergence outside the initial layer, we need the pointwise decay of the quantity:
\begin{equation}\label{Sect2_Relaxation_Wave_Define}
\begin{array}{ll}
\eta(x,t) = v(x,t) + \frac{1}{k_1\varrho(x,t)}\nabla\xi(x,t).
\end{array}
\end{equation}

Differentiate $(\ref{Sect2_Relaxation_Wave_Define})$ wit respect to $t$, we have
\begin{equation*}
\begin{array}{ll}
\eta_t = v_t + \frac{1}{k_1\varrho}\nabla\xi_t - \frac{\zeta_t}{k_1\varrho^2}\nabla\xi,
\end{array}
\end{equation*}
then $\eta$ satisfies the following transportation equation with damping and  forcing:
\begin{equation}\label{Sect2_Wave_Eq}
\begin{array}{ll}
\eta_t + k_1 v\cdot\nabla\eta + \frac{1}{\tau^2}\eta
= v\cdot\nabla(\frac{1}{\varrho})\nabla\xi
- \frac{1}{\varrho}(\nabla v) \nabla\xi - \frac{\gamma}{\varrho}\nabla \xi \nabla\cdot v \\[9pt]\hspace{3.4cm}
- \frac{\gamma}{\varrho}p\nabla(\nabla\cdot v)
+\frac{1}{\varrho^2}(v\cdot\nabla\zeta + \varrho\nabla\cdot v)\nabla\xi,
\end{array}
\end{equation}
where $\tau\cdot[f\!orcing\ terms]$ is bounded uniform with respect to $\tau$, $(\nabla v)$ is a matrix, the damping effect for $\eta$ becomes stronger as $\tau$ decreases.

The equation $(\ref{Sect2_Wave_Eq})$ produces the following estimate:
\begin{equation}\label{Sect2_Wave_Estimate1}
\begin{array}{ll}
|\eta|_{\infty}^2\leq C\|\eta\|_{H^2(\mathbb{R}^3)}^2\leq C\|\eta|_{t=0}\|_{H^2(\mathbb{R}^3)}^2\exp\{-\frac{t}{\tau^2}\}+C\tau^2.
\end{array}
\end{equation}

Also, $\eta$ satisfies another equation:
\begin{equation}\label{Sect2_Eta_Velocity}
\begin{array}{ll}
\eta = \tau^2(v_t + k_1 v\cdot\nabla v),
\end{array}
\end{equation}
which produces the following estimate:
\begin{equation}\label{Sect2_Wave_Estimate2}
\begin{array}{ll}
\int\limits_0^T \|\eta_t\|_{H^1(\mathbb{R}^3)}^2 \,\mathrm{d}x \leq C\tau^2.
\end{array}
\end{equation}

Based on the estimates $(\ref{Sect2_Wave_Estimate1})$ and $(\ref{Sect2_Wave_Estimate2})$, we have the strong convergence of $v$ outside the initial layer.

In $\mathbb{R}^3$, we have Gagliado-Nirenberg type inequalities, which are useful in estimating the nonlinear terms.
\begin{equation}\label{Sect2_Sobolev_Ineq}
\left\{\begin{array}{ll}
\|\cdot\|_{L^6(\mathbb{R}^3)}\lem \|\cdot\|_{H^1(\mathbb{R}^3)}, \\[6pt]
\|\partial\nabla U\|_{L^4(\mathbb{R}^3)}^2 \lem |\partial U|_{\infty}\|\partial\nabla^2 U\|_{L^2(\mathbb{R}^3)}, \
(see\ \cite{Sideris_Thomases_Wang_2003}),
\end{array}\right.
\end{equation}
where $U$ is a vector or scalar function, $\partial=(\partial_t,\nabla)$.

In the sequent sections, we will use the following notations: $X\lem Y$ denotes the estimate $X\leq CY$ for some implied constant $C>0$ which may different line by line. $(\cdot)_{k}$ denotes a vector in $\mathbb{R}^3$, for instance, $\omega_{k}=\delta_{ijk}\partial_{i}v_j$, where $\delta_{ijk}$ is totally anti-symmetric tensor such that $\delta_{123}=\delta_{231}=\delta_{312}=1, \delta_{213}=\delta_{321}=\delta_{132}=-1$, others are $0$. 'R.H.S.' is the abbreviation for 'right hand side'.

\section{A Priori Estimates for the Relaxed Equations}
In this section, we derive a priori estimates in $[0,T]$ for the relaxed equations $(\ref{Sect2_Relaxed_Eq})$, where $T>0$ is finite and independent of $\tau$.

The following lemma concerns the estimates related to the density, the proof is omitted for its simplicity.
\begin{lemma}\label{Sect3_Varrho_Lemma}
For any given $T\in (0,+\infty)$, if
\begin{equation*}
\sup\limits_{0\leq t\leq T} \mathcal{F}[\xi,v,\phi,\zeta](t) \leq\ve,
\end{equation*}
where $0<\ve\ll 1$, then
\begin{equation}\label{Sect3_Varrho_Eq}
\begin{array}{ll}
\sup\limits_{0\leq t\leq T}|\zeta_t|_{\infty}\lem \sqrt{\ve},\
\sup\limits_{0\leq t\leq T}|\zeta|_{\infty}\lem \sqrt{\ve},\
\sup\limits_{0\leq t\leq T}|\nabla\zeta|_{\infty}\lem \sqrt{\ve}.
\end{array}
\end{equation}
\end{lemma}

The following lemma states that $\mathcal{F}[\xi](t)$ and $\mathcal{F}_1[\xi](t)$ are equivalent.
\begin{lemma}\label{Sect3_Epsilon0_Lemma}
For any given $T\in (0,+\infty)$, there exists $\ve_1>0$ which is independent of $(\xi_0, v_0,\phi_0)$, such that if $\sup\limits_{0\leq t\leq T} \mathcal{F}[\xi,v,\phi,\zeta](t) \leq\ve_1$, then
$|\xi|_{\infty}\leq \frac{\bar{p}}{3},|\zeta|_{\infty}\leq \frac{\bar{\varrho}}{2}$ and
there exist $c_1>0,c_2>0$ such that
\begin{equation}\label{Sect3_Energy_Equivalence}
\begin{array}{ll}
c_1 \mathcal{F}[\xi](t) \leq \mathcal{F}_1[\xi](t) \leq c_2 \mathcal{F}[\xi](t).
\end{array}
\end{equation}
\end{lemma}

To make calculations simpler, we calculate $\frac{\mathrm{d}}{\mathrm{d}t}\mathcal{F}_1[\xi](t)$ and $\frac{\mathrm{d}}{\mathrm{d}t}\mathcal{F}[v](t)$ separately.

The following lemma concerns the estimate for $L^{\infty}$ bound of $\mathcal{F}_1[\xi](t)$.
\begin{lemma}\label{Sect3_Energy_Estimate_Lemma2}
For any given $T\in (0,+\infty)$, if
\begin{equation*}
\sup\limits_{0\leq t\leq T} \mathcal{F}[\xi,v,\phi,\zeta](t) \leq\ve,
\end{equation*}
where $0<\ve\ll 1$, then for $\forall t\in [0,T]$,
\begin{equation}\label{Sect3_Estimate1_toProve}
\frac{\mathrm{d}}{\mathrm{d}t}\mathcal{F}_1[\xi](t)
+ 2 \mathcal{F}[v](t) \leq C\sqrt{\ve}(\mathcal{F}_X[\xi](t) + \mathcal{F}[v](t)).
\end{equation}
\end{lemma}

\begin{proof}
Let $(\ref{Sect2_Relaxed_Eq})\cdot(\xi, v)$, we get
\begin{equation}\label{Sect3_ZeroOrder_1}
\begin{array}{ll}
(|\xi|^2)_t + 2k_2 \xi\nabla\cdot v + 2k_2 v\cdot\nabla\xi + 2 |v|^2 \\[6pt]
= -2\gamma k_1 \xi^2\nabla\cdot v - 2k_1 \xi v\cdot\nabla \xi
+ \frac{2}{k_1} (\frac{1}{\bar{\varrho}}- \frac{1}{\varrho})\nabla\xi\cdot v.
\end{array}
\end{equation}

Integrate $(\ref{Sect3_ZeroOrder_1})$ in $\mathbb{R}^3$ and note that $\int\limits_{\mathbb{R}^3}\nabla\cdot(\xi v) \,\mathrm{d}x = 0$, we get
\begin{equation}\label{Sect3_ZeroOrder_3}
\begin{array}{ll}
\frac{\mathrm{d}}{\mathrm{d} t}
\int\limits_{\mathbb{R}^3}|\xi|^2 \,\mathrm{d}x + 2\int\limits_{\mathbb{R}^3}|v|^2 \,\mathrm{d}x \\[9pt]
= \int\limits_{\mathbb{R}^3} 2\gamma k_1 v\cdot\nabla(\xi^2) - 2k_1 \xi v\cdot\nabla \xi
+ \frac{2}{k_1} (\frac{1}{\bar{\varrho}}- \frac{1}{\varrho})\nabla\xi\cdot v \,\mathrm{d}x \\[9pt]
\lem \sqrt{\ve}\|\nabla\xi\|_{L^2(\mathbb{R}^3)}\|v\|_{L^2(\mathbb{R}^3)} \\[9pt]
\lem \sqrt{\ve}(\mathcal{F}_X[\xi](t) + \mathcal{F}[v](t)).
\end{array}
\end{equation}

Apply $\mathcal{D}^{\alpha}$ to $(\ref{Sect2_Relaxing_Eq})$, where $1\leq|\alpha|\leq 4$, we get
\begin{equation}\label{Sect3_Estimate2_1}
\left\{\begin{array}{ll}
(\mathcal{D}^{\alpha}\xi)_t + k_2\nabla\cdot(\mathcal{D}^{\alpha} v)
= -\gamma k_1\mathcal{D}^{\alpha}(\xi\nabla\cdot v)
- k_1 \mathcal{D}^{\alpha}(v\cdot\nabla \xi), \\[6pt]
k_2 \nabla(\mathcal{D}^{\alpha}\xi)
+ \mathcal{D}^{\alpha} v
= \frac{1}{k_1} \mathcal{D}^{\alpha}[(\frac{1}{\bar{\varrho}}- \frac{1}{\varrho})\nabla\xi].
\end{array}\right.
\end{equation}

Let $(\ref{Sect3_Estimate2_1})\cdot(\mathcal{D}^{\alpha}\xi, \mathcal{D}^{\alpha} v)$, we get
\begin{equation}\label{Sect3_Estimate2_2}
\begin{array}{ll}
(|\mathcal{D}^{\alpha}\xi|^2)_t
+ 2k_2 \mathcal{D}^{\alpha}\xi\nabla\cdot(\mathcal{D}^{\alpha} v)
+ 2k_2 \mathcal{D}^{\alpha} v\cdot\nabla(\mathcal{D}^{\alpha}\xi)
+ 2 |\mathcal{D}^{\alpha} v|^2 \\[6pt]
= -2\gamma k_1 (\mathcal{D}^{\alpha}\xi)\mathcal{D}^{\alpha}(\xi\nabla\cdot v)
- 2k_1 (\mathcal{D}^{\alpha}\xi)\mathcal{D}^{\alpha}(v\cdot\nabla \xi) \\[6pt]\quad
+ \frac{2}{k_1} (\mathcal{D}^{\alpha} v)\cdot
\mathcal{D}^{\alpha}[(\frac{1}{\bar{\varrho}}- \frac{1}{\varrho})\nabla\xi].
\end{array}
\end{equation}

Integrate $(\ref{Sect3_Estimate2_2})$ in $\mathbb{R}^3$ and note that $\int\limits_{\mathbb{R}^3}\nabla\cdot(\mathcal{D}^{\alpha}\xi\mathcal{D}^{\alpha} v) \,\mathrm{d}x = 0$, we get
\begin{equation}\label{Sect3_Estimate2_5}
\begin{array}{ll}
\frac{\mathrm{d}}{\mathrm{d} t}
\int\limits_{\mathbb{R}^3}|\mathcal{D}^{\alpha}\xi|^2 \,\mathrm{d}x
+ 2\int\limits_{\mathbb{R}^3}|\mathcal{D}^{\alpha} v|^2 \,\mathrm{d}x \\[10pt]
= \int\limits_{\mathbb{R}^3} -2\gamma k_1 (\mathcal{D}^{\alpha}\xi)\mathcal{D}^{\alpha}(\xi\nabla\cdot v)
- 2k_1 (\mathcal{D}^{\alpha}\xi)\mathcal{D}^{\alpha}(v\cdot\nabla \xi) \\[10pt]\quad
+ \frac{2}{k_1} (\mathcal{D}^{\alpha} v)\cdot
\mathcal{D}^{\alpha}[(\frac{1}{\bar{\varrho}}- \frac{1}{\varrho})\nabla\xi] \,\mathrm{d}x
:= I_1.
\end{array}
\end{equation}

\vspace{0.3cm}
When $1\leq|\alpha|\leq 3$, it is easy to check that
$I_1\lem \sqrt{\ve}(\mathcal{F}_X[\xi](t)+\mathcal{F}[v](t))$. \\[6pt]
\indent
When $|\alpha|=4$, we estimate the quantity
$I_1 - \frac{\mathrm{d}}{\mathrm{d}t}\int\limits_{\mathbb{R}^3}\frac{\xi}{p}(\mathcal{D}^{\alpha}\xi)^2 \,\mathrm{d}x$, then
\begin{equation}\label{Sect3_Estimate2_6}
\begin{array}{ll}
I_1 - \frac{\mathrm{d}}{\mathrm{d}t}\int\limits_{\mathbb{R}^3}\frac{\xi}{p}(\mathcal{D}^{\alpha}\xi)^2 \,\mathrm{d}x \\[8pt]
= -2\gamma k_1 \int\limits_{\mathbb{R}^3}(\mathcal{D}^{\alpha}\xi)\xi\nabla\cdot (\mathcal{D}^{\alpha} v)\,\mathrm{d}x
- 2k_1 \int\limits_{\mathbb{R}^3}(\mathcal{D}^{\alpha}\xi) v\cdot\nabla (\mathcal{D}^{\alpha}\xi)\,\mathrm{d}x \\[8pt]\quad
+ \frac{2}{k_1} \int\limits_{\mathbb{R}^3}(\frac{1}{\bar{\varrho}}
- \frac{1}{\varrho})(\mathcal{D}^{\alpha} v)\cdot\nabla(\mathcal{D}^{\alpha}\xi)\,\mathrm{d}x
- 2\int\limits_{\mathbb{R}^3}\frac{\xi}{p}(\mathcal{D}^{\alpha}\xi)
(\mathcal{D}^{\alpha}\xi_t) \,\mathrm{d}x \\[8pt]\quad
- \int\limits_{\mathbb{R}^3}\partial_t(\frac{\xi}{p})(\mathcal{D}^{\alpha}\xi)^2 \,\mathrm{d}x
\\[8pt]

\lem \sqrt{\ve}(\mathcal{F}_X[\xi](t)+\mathcal{F}[v](t))
+ k_1 \int\limits_{\mathbb{R}^3}|\mathcal{D}^{\alpha}\xi|^2
\nabla\cdot v\,\mathrm{d}x \\[8pt]\quad
-2\gamma k_1 \int\limits_{\mathbb{R}^3}\xi(\mathcal{D}^{\alpha}\xi)\nabla\cdot (\mathcal{D}^{\alpha} v)\,\mathrm{d}x

+ \frac{2}{k_1} \int\limits_{\mathbb{R}^3}(\frac{1}{\bar{\varrho}}- \frac{1}{\varrho})(\mathcal{D}^{\alpha} v)\cdot\nabla(\mathcal{D}^{\alpha}\xi)\,\mathrm{d}x
\\[8pt]\quad
- 2\int\limits_{\mathbb{R}^3}\frac{\xi}{p}
(\mathcal{D}^{\alpha}\xi)(\mathcal{D}^{\alpha}\xi_t) \,\mathrm{d}x
\\[10pt]

\lem -2\int\limits_{\mathbb{R}^3}\frac{\xi}{p}(\mathcal{D}^{\alpha}\xi)
[\mathcal{D}^{\alpha}\xi_t +k_1\gamma p\nabla\cdot (\mathcal{D}^{\alpha} v)]\,\mathrm{d}x
\\[8pt]\quad
+ \frac{2}{k_1} \int\limits_{\mathbb{R}^3}(\frac{1}{\bar{\varrho}}- \frac{1}{\varrho})(\mathcal{D}^{\alpha} v)\cdot\nabla(\mathcal{D}^{\alpha}\xi)
\,\mathrm{d}x + \sqrt{\ve}(\mathcal{F}_X[\xi](t)+\mathcal{F}[v](t)).
\end{array}
\end{equation}

Apply $\mathcal{D}^{\alpha}$ to $(\ref{Sect2_Relaxing_Eq})_1$, where $|\alpha|=4$, we get
\begin{equation}\label{Sect3_Estimate2_7}
\begin{array}{ll}
\mathcal{D}^{\alpha}\xi_t + k_1\gamma p\nabla\cdot (\mathcal{D}^{\alpha} v) = -k_1\mathcal{D}^{\alpha}(v\cdot\nabla\xi)  
- k_1\gamma\sum\limits_{|\alpha_1|>0}\mathcal{D}^{\alpha_1}\xi \nabla\cdot (\mathcal{D}^{\alpha_2} v).
\end{array}
\end{equation}

Plug $(\ref{Sect3_Estimate2_7})$ into the following integral, we get
\begin{equation}\label{Sect3_Estimate2_8}
\begin{array}{ll}
\int\limits_{\mathbb{R}^3}\frac{\xi}{p}(\mathcal{D}^{\alpha}\xi)
[\mathcal{D}^{\alpha}\xi_t+k_1\gamma p\nabla\cdot (\mathcal{D}^{\alpha} v)]\,\mathrm{d}x \\[9pt]
= \int\limits_{\mathbb{R}^3}\frac{\xi}{p}(\mathcal{D}^{\alpha}\xi)[R.H.S.\ of\ (\ref{Sect3_Estimate2_7})]\,\mathrm{d}x \\[9pt]

\lem \sqrt{\ve}(\mathcal{F}_X[\xi](t) + \mathcal{F}[v](t))
+ \frac{k_1}{2}\int\limits_{\mathbb{R}^3}|\mathcal{D}^{\alpha}\xi|^2 \nabla\cdot(\frac{\xi}{p}v) \,\mathrm{d}x
\\[6pt]
\lem \sqrt{\ve}(\mathcal{F}_X[\xi](t) + \mathcal{F}[v](t)).
\end{array}
\end{equation}

Apply $\mathcal{D}^{\alpha}$ to
$ k_1\varrho v + \nabla\xi =0$, where $|\alpha|=4$, we get
\begin{equation}\label{Sect3_Estimate2_9}
\begin{array}{ll}
\nabla(\mathcal{D}^{\alpha}\xi)
= -k_1 \varrho\mathcal{D}^{\alpha} v
-\sum\limits_{\ell_1+|\alpha_1|>0} k_1\partial_t^{\ell_1}\mathcal{D}^{\alpha_1}\zeta \partial_t^{\ell_2}\mathcal{D}^{\alpha_2}v.
\end{array}
\end{equation}

Plug $(\ref{Sect3_Estimate2_9})$ into the following integral, we get
\begin{equation}\label{Sect3_Estimate2_10}
\begin{array}{ll}
\int\limits_{\mathbb{R}^3}(\frac{1}{\bar{\varrho}}- \frac{1}{\varrho})(\mathcal{D}^{\alpha} v)\cdot\nabla(\mathcal{D}^{\alpha}\xi)\,\mathrm{d}x \\[6pt]
= \int\limits_{\mathbb{R}^3}(\frac{1}{\bar{\varrho}}- \frac{1}{\varrho})(\mathcal{D}^{\alpha} v)\cdot [R.H.S.\ of\ (\ref{Sect3_Estimate2_9})]\,\mathrm{d}x

\lem \sqrt{\ve}\mathcal{F}[v](t).
\end{array}
\end{equation}

Plug $(\ref{Sect3_Estimate2_8})$ and $(\ref{Sect3_Estimate2_10})$ into $(\ref{Sect3_Estimate2_6})$, we get
\begin{equation}\label{Sect3_Estimate2_11}
\begin{array}{ll}
I_1 - \frac{\mathrm{d}}{\mathrm{d}t}\int\limits_{\mathbb{R}^3}\frac{\xi}{p}(\mathcal{D}^{\alpha}\xi)^2 \,\mathrm{d}x
\lem \sqrt{\ve}(\mathcal{F}_X[\xi](t) + \mathcal{F}[v](t)).
\end{array}
\end{equation}

After summing $\alpha$, where $0\leq|\alpha|\leq 4$, we have
\begin{equation}\label{Sect3_Estimate2_12}
\begin{array}{ll}
\frac{\mathrm{d}}{\mathrm{d} t}\left(\sum\limits_{0<|\alpha|\leq 4}\
\int\limits_{\mathbb{R}^3}|\mathcal{D}^{\alpha}\xi|^2 \,\mathrm{d}x
-\sum\limits_{|\alpha|=4}\
\int\limits_{\mathbb{R}^3}\frac{\xi}{p} (\mathcal{D}^{\alpha}\xi)^2\,\mathrm{d}x \right) 
+ 2 \sum\limits_{0<|\alpha|\leq 4}\
\int\limits_{\mathbb{R}^3}|\mathcal{D}^{\alpha}v|^2 \,\mathrm{d}x \\[15pt]
\lem \sqrt{\ve}(\mathcal{F}_X[\xi](t) + \mathcal{F}[v](t)).
\end{array}
\end{equation}

\vspace{0.3cm}
Next, we estimate the part of $\mathcal{F}_1[\xi](t)$ which contains the time derivatives.
Apply $\mathcal{D}^{\alpha}\partial_t$ to $(\ref{Sect2_Relaxing_Eq})$, where $0\leq|\alpha|\leq 2$, we get
\begin{equation}\label{Sect3_Estimate3_1}
\left\{\begin{array}{ll}
(\mathcal{D}^{\alpha}\xi_t)_t + k_2\nabla\cdot(\mathcal{D}^{\alpha}v_t)
= -\gamma k_1\mathcal{D}^{\alpha}\partial_t(\xi\nabla\cdot v)
- k_1 \mathcal{D}^{\alpha}\partial_t(v\cdot\nabla \xi), \\[6pt]
k_2 \nabla(\mathcal{D}^{\alpha}\xi_t)
+ \mathcal{D}^{\alpha}v_t
= \frac{1}{k_1} \mathcal{D}^{\alpha}\partial_t[(\frac{1}{\bar{\varrho}}- \frac{1}{\varrho})\nabla\xi].
\end{array}\right.
\end{equation}

Let $(\ref{Sect3_Estimate3_1})\cdot(\mathcal{D}^{\alpha}\xi_t, \mathcal{D}^{\alpha}v_t)$, we get
\begin{equation}\label{Sect3_Estimate3_2}
\begin{array}{ll}
(|\mathcal{D}^{\alpha}\xi_t|^2)_t
+ 2k_2 \mathcal{D}^{\alpha}\xi_t\nabla\cdot(\mathcal{D}^{\alpha}v_t)
+ 2k_2 \mathcal{D}^{\alpha}v_t\cdot\nabla(\mathcal{D}^{\alpha}\xi_t)
+ 2 |\mathcal{D}^{\alpha}v_t|^2 \\[6pt]
= -2\gamma k_1 (\mathcal{D}^{\alpha}\xi_t)\mathcal{D}^{\alpha}\partial_t(\xi\nabla\cdot v)
- 2k_1 (\mathcal{D}^{\alpha}\xi_t)\mathcal{D}^{\alpha}\partial_t(v\cdot\nabla \xi) \\[6pt]\quad
+ \frac{2}{k_1} (\mathcal{D}^{\alpha}v_t)\cdot
\mathcal{D}^{\alpha}\partial_t[(\frac{1}{\bar{\varrho}}- \frac{1}{\varrho})\nabla\xi].
\end{array}
\end{equation}

Integrate $(\ref{Sect3_Estimate3_2})$ in $\mathbb{R}^3$ and note that $\int\limits_{\mathbb{R}^3}\nabla\cdot(\mathcal{D}^{\alpha}\xi_t\mathcal{D}^{\alpha}v_t) \,\mathrm{d}x = 0$, we get
\begin{equation}\label{Sect3_Estimate3_5}
\begin{array}{ll}
\frac{\mathrm{d}}{\mathrm{d} t}
\int\limits_{\mathbb{R}^3}|\mathcal{D}^{\alpha}\xi_t|^2 \,\mathrm{d}x
+ 2\int\limits_{\mathbb{R}^3}|\mathcal{D}^{\alpha}v_t|^2 \,\mathrm{d}x \\[10pt]
= \int\limits_{\mathbb{R}^3} -2\gamma k_1 (\mathcal{D}^{\alpha}\xi_t)\mathcal{D}^{\alpha}\partial_t(\xi\nabla\cdot v)
- 2k_1 (\mathcal{D}^{\alpha}\xi_t)\mathcal{D}^{\alpha}\partial_t(v\cdot\nabla \xi) \\[10pt]\quad
+ \frac{2}{k_1} (\mathcal{D}^{\alpha}v_t)\cdot
\mathcal{D}^{\alpha}\partial_t[(\frac{1}{\bar{\varrho}}- \frac{1}{\varrho})\nabla\xi] \,\mathrm{d}x
:= I_2.
\end{array}
\end{equation}

\vspace{0.3cm}
When $0\leq|\alpha|\leq 1$, it is easy to check that
$I_2\lem \sqrt{\ve}(\mathcal{F}_X[\xi](t)+\mathcal{F}[v](t))$. \\[6pt]
\indent
When $|\alpha|=2$, we estimate the quantity
$I_2 - \frac{\mathrm{d}}{\mathrm{d}t}\int\limits_{\mathbb{R}^3}\frac{\xi}{p}(\mathcal{D}^{\alpha}\xi_t)^2 \,\mathrm{d}x$, then
\begin{equation}\label{Sect3_Estimate3_6}
\begin{array}{ll}
I_2 - \frac{\mathrm{d}}{\mathrm{d}t}\int\limits_{\mathbb{R}^3}\frac{\xi}{p}(\mathcal{D}^{\alpha}\xi_t)^2 \,\mathrm{d}x \\[8pt]
= -2\gamma k_1 \int\limits_{\mathbb{R}^3}(\mathcal{D}^{\alpha}\xi_t)\xi\nabla\cdot (\mathcal{D}^{\alpha}v_t)\,\mathrm{d}x
- 2k_1 \int\limits_{\mathbb{R}^3}(\mathcal{D}^{\alpha}\xi_t) v\cdot\nabla (\mathcal{D}^{\alpha}\xi_t)\,\mathrm{d}x \\[8pt]\quad
+ \frac{2}{k_1} \int\limits_{\mathbb{R}^3}(\frac{1}{\bar{\varrho}}
- \frac{1}{\varrho})(\mathcal{D}^{\alpha}v_t)\cdot\nabla(\mathcal{D}^{\alpha}\xi_t)\,\mathrm{d}x
- 2\int\limits_{\mathbb{R}^3}\frac{\xi}{p}(\mathcal{D}^{\alpha}\xi_t)
(\mathcal{D}^{\alpha}\xi_{tt}) \,\mathrm{d}x \\[8pt]\quad
- \int\limits_{\mathbb{R}^3}\partial_t(\frac{\xi}{p})(\mathcal{D}^{\alpha}\xi_t)^2 \,\mathrm{d}x
\\[8pt]

\lem \sqrt{\ve}(\mathcal{F}_X[\xi](t)+\mathcal{F}[v](t))
+ k_1 \int\limits_{\mathbb{R}^3}|\mathcal{D}^{\alpha}\xi_t|^2
\nabla\cdot v\,\mathrm{d}x \\[8pt]\quad
-2\gamma k_1 \int\limits_{\mathbb{R}^3}\xi(\mathcal{D}^{\alpha}\xi_t)\nabla\cdot (\mathcal{D}^{\alpha}v_t)\,\mathrm{d}x

+ \frac{2}{k_1} \int\limits_{\mathbb{R}^3}(\frac{1}{\bar{\varrho}}- \frac{1}{\varrho})(\mathcal{D}^{\alpha}v_t)\cdot\nabla(\mathcal{D}^{\alpha}\xi_t)\,\mathrm{d}x
\\[8pt]\quad
- 2\int\limits_{\mathbb{R}^3}\frac{\xi}{p}
(\mathcal{D}^{\alpha}\xi_t)(\mathcal{D}^{\alpha}\xi_{tt}) \,\mathrm{d}x
\\[10pt]

\lem -2\int\limits_{\mathbb{R}^3}\frac{\xi}{p}(\mathcal{D}^{\alpha}\xi_t)
[\mathcal{D}^{\alpha}\xi_{tt} +k_1\gamma p\nabla\cdot (\mathcal{D}^{\alpha}v_t)]\,\mathrm{d}x
\\[8pt]\quad
+ \frac{2}{k_1} \int\limits_{\mathbb{R}^3}(\frac{1}{\bar{\varrho}}- \frac{1}{\varrho})(\mathcal{D}^{\alpha}v_t)\cdot\nabla(\mathcal{D}^{\alpha}\xi_t)
\,\mathrm{d}x + \sqrt{\ve}(\mathcal{F}_X[\xi](t)+\mathcal{F}[v](t)).
\end{array}
\end{equation}

Apply $\mathcal{D}^{\alpha}\partial_t$ to $(\ref{Sect2_Relaxing_Eq})_1$, where $|\alpha|=2$, we get
\begin{equation}\label{Sect3_Estimate3_7}
\begin{array}{ll}
\mathcal{D}^{\alpha}\xi_{tt} + k_1\gamma p\nabla\cdot (\mathcal{D}^{\alpha}v_t) \\[6pt]
= -k_1\mathcal{D}^{\alpha}(v_t\cdot\nabla\xi) -k_1\mathcal{D}^{\alpha}(v\cdot\nabla\xi_t)\\[8pt]\quad
- k_1\gamma\sum\limits_{|\alpha_1|\geq 0}\mathcal{D}^{\alpha_1}\xi_t \nabla\cdot (\mathcal{D}^{\alpha_2} v)
- k_1\gamma\sum\limits_{|\alpha_1|>0}\mathcal{D}^{\alpha_1}\xi \nabla\cdot (\mathcal{D}^{\alpha_2}v_t).
\end{array}
\end{equation}

Plug $(\ref{Sect3_Estimate3_7})$ into the following integral, we get
\begin{equation}\label{Sect3_Estimate3_8}
\begin{array}{ll}
\int\limits_{\mathbb{R}^3}\frac{\xi}{p}(\mathcal{D}^{\alpha}\xi_t)
[\mathcal{D}^{\alpha}\xi_{tt}+k_1\gamma p\nabla\cdot (\mathcal{D}^{\alpha}v_t)]\,\mathrm{d}x \\[9pt]
= \int\limits_{\mathbb{R}^3}\frac{\xi}{p}(\mathcal{D}^{\alpha}\xi_t)[R.H.S.\ of\ (\ref{Sect3_Estimate3_7})]\,\mathrm{d}x \\[9pt]

\lem \sqrt{\ve}(\mathcal{F}_X[\xi](t) + \mathcal{F}[v](t))
+ \frac{k_1}{2}\int\limits_{\mathbb{R}^3}|\mathcal{D}^{\alpha}\xi_t|^2 \nabla\cdot(\frac{\xi}{p}v) \,\mathrm{d}x
\\[6pt]
\lem \sqrt{\ve}(\mathcal{F}_X[\xi](t) + \mathcal{F}[v](t)).
\end{array}
\end{equation}

Apply $\mathcal{D}^{\alpha}\partial_t$ to
$ k_1\varrho v + \nabla\xi =0$, where $|\alpha|=2$, we get
\begin{equation}\label{Sect3_Estimate3_9}
\begin{array}{ll}
\nabla(\mathcal{D}^{\alpha}\xi_t)
= -k_1 \varrho\mathcal{D}^{\alpha}v_t
-\sum\limits_{|\alpha_1|>0} k_1\mathcal{D}^{\alpha_1}\zeta \mathcal{D}^{\alpha_2}v_t
-\sum\limits_{|\alpha_1|\geq 0} k_1\mathcal{D}^{\alpha_1}\zeta_t \mathcal{D}^{\alpha_2}v.
\end{array}
\end{equation}

Plug $(\ref{Sect3_Estimate3_9})$ into the following integral, we get
\begin{equation}\label{Sect3_Estimate3_10}
\begin{array}{ll}
\int\limits_{\mathbb{R}^3}(\frac{1}{\bar{\varrho}}- \frac{1}{\varrho})(\mathcal{D}^{\alpha}v_t)\cdot\nabla(\mathcal{D}^{\alpha}\xi_t)\,\mathrm{d}x \\[6pt]
= \int\limits_{\mathbb{R}^3}(\frac{1}{\bar{\varrho}}- \frac{1}{\varrho})(\mathcal{D}^{\alpha}v_t)\cdot [R.H.S.\ of\ (\ref{Sect3_Estimate3_9})]\,\mathrm{d}x

\lem \sqrt{\ve}\mathcal{F}[v](t).
\end{array}
\end{equation}

Plug $(\ref{Sect3_Estimate3_8})$ and $(\ref{Sect3_Estimate3_10})$ into $(\ref{Sect3_Estimate3_6})$, we get
\begin{equation}\label{Sect3_Estimate3_11}
\begin{array}{ll}
I_2 - \frac{\mathrm{d}}{\mathrm{d}t}\int\limits_{\mathbb{R}^3}\frac{\xi}{p}(\mathcal{D}^{\alpha}\xi_t)^2 \,\mathrm{d}x
\lem \sqrt{\ve}(\mathcal{F}_X[\xi](t) + \mathcal{F}[v](t)).
\end{array}
\end{equation}

After summing $\alpha$, where $0\leq|\alpha|\leq 2$, we have
\begin{equation}\label{Sect3_Estimate3_12}
\begin{array}{ll}
\frac{\mathrm{d}}{\mathrm{d} t}\left(\sum\limits_{0\leq|\alpha|\leq 2}\
\int\limits_{\mathbb{R}^3}|\mathcal{D}^{\alpha}\xi_t|^2 \,\mathrm{d}x
-\sum\limits_{|\alpha|=2}\
\int\limits_{\mathbb{R}^3}\frac{\xi}{p} (\mathcal{D}^{\alpha}\xi_t)^2\,\mathrm{d}x \right) 
+ 2 \sum\limits_{0\leq|\alpha|\leq 2}\
\int\limits_{\mathbb{R}^3}|\mathcal{D}^{\alpha}v_t|^2 \,\mathrm{d}x \\[15pt]\quad
\lem \sqrt{\ve}(\mathcal{F}_X[\xi](t) + \mathcal{F}[v](t)).
\end{array}
\end{equation}

By $(\ref{Sect3_Estimate2_12})+(\ref{Sect3_Estimate3_12})$, we obtain
\begin{equation}\label{Sect3_Estimate2_14}
\frac{\mathrm{d}}{\mathrm{d}t}\mathcal{F}_1[\xi](t)
+ 2 \mathcal{F}[v](t) \lem \sqrt{\ve}(\mathcal{F}_X[\xi](t) + \mathcal{F}[v](t)).
\end{equation}
Thus, Lemma $\ref{Sect3_Energy_Estimate_Lemma2}$ is proved.
\end{proof}

The following lemma concerns the estimate for $L^{\infty}$ bound of $\mathcal{F}[v](t)$.
\begin{lemma}\label{Sect3_Velocity_Lemma}
For any given $T\in (0,+\infty)$, if
\begin{equation*}
\sup\limits_{0\leq t\leq T} \mathcal{F}[\xi,v,\phi,\zeta](t) \leq\ve,
\end{equation*}
where $0<\ve\ll 1$, then for $\forall t\in [0,T]$,
\begin{equation}\label{Sect3_Velocity_toProve}
\begin{array}{ll}
\frac{\mathrm{d}}{\mathrm{d}t} \mathcal{F}[v](t)
+ \frac{\gamma}{a}\int\limits_{\mathbb{R}^3} \frac{p}{\varrho}\sum\limits_{0\leq|\alpha|\leq 4}|\nabla\cdot \mathcal{D}^{\alpha}v|^2\,\mathrm{d}x
+ \frac{\gamma}{a}\int\limits_{\mathbb{R}^3} \frac{p}{\varrho}\sum\limits_{0\leq|\alpha|\leq 2}|\nabla\cdot \mathcal{D}^{\alpha}v_t|^2\,\mathrm{d}x \\[9pt]
\leq C\sqrt{\ve}\mathcal{F}[v](t).
\end{array}
\end{equation}
\end{lemma}

\begin{proof}
Firstly, we estimate the $L^{\infty}$ bound of $\|v\|_{H^4(\mathbb{R}^3)}$:
Let $\mathcal{D}^{\alpha} v\cdot\mathcal{D}^{\alpha}(\ref{Sect2_Velocity_Solve})$,
where $0\leq|\alpha|\leq 4$, we get
\begin{equation}\label{Sect3_Velocity_2}
\begin{array}{ll}
\partial_t|\mathcal{D}^{\alpha} v|^2 = 2\mathcal{D}^{\alpha} v\cdot\mathcal{D}^{\alpha}
[k_1(1-\gamma)v(\nabla\cdot v)- k_1 v\cdot\nabla v - \frac{k_1}{2} \nabla(|v|^2)] \\[7pt]\hspace{1.8cm}

+ \frac{2\gamma}{a}\sum\limits_{|\alpha_1|>0}
[\mathcal{D}^{\alpha_1}(\frac{p}{\varrho})\mathcal{D}^{\alpha_2}\nabla(\nabla\cdot v)]\cdot\mathcal{D}^{\alpha} v \\[10pt]\hspace{1.8cm}
+ \frac{2\gamma p}{a\varrho}\mathcal{D}^{\alpha}\nabla(\nabla\cdot v)\cdot\mathcal{D}^{\alpha} v.
\end{array}
\end{equation}

Integrate $(\ref{Sect3_Velocity_2})$ in $\mathbb{R}^3$, we get
\begin{equation}\label{Sect3_Velocity_3}
\begin{array}{ll}
\frac{\mathrm{d}}{\mathrm{d}t}\int\limits_{\mathbb{R}^3}|\mathcal{D}^{\alpha} v|^2 \,\mathrm{d}x
\\[10pt]
= 2\int\limits_{\mathbb{R}^3}\mathcal{D}^{\alpha} v\cdot\mathcal{D}^{\alpha}
[k_1(1-\gamma)v(\nabla\cdot v)- k_1 v\cdot\nabla v
- \frac{k_1}{2} \nabla(|v|^2)]\,\mathrm{d}x \\[10pt]\quad

+ \frac{2\gamma}{a}\int\limits_{\mathbb{R}^3}\sum\limits_{|\alpha_1|>0}
[\mathcal{D}^{\alpha_1}(\frac{p}{\varrho})\mathcal{D}^{\alpha_2}\nabla(\nabla\cdot v)]\cdot
\mathcal{D}^{\alpha} v \,\mathrm{d}x
\\[10pt]\quad

+ \frac{2\gamma}{a}\int\limits_{\mathbb{R}^3}\frac{p}{\varrho}\mathcal{D}^{\alpha}\nabla(\nabla\cdot v)
\cdot\mathcal{D}^{\alpha} v \,\mathrm{d}x := I_3.
\end{array}
\end{equation}

When $0\leq|\alpha|\leq 3$,
\begin{equation}\label{Sect3_Velocity_8}
\begin{array}{ll}
\frac{2\gamma}{a}\int\limits_{\mathbb{R}^3} \frac{p}{\varrho}\mathcal{D}^{\alpha}v\cdot\nabla(\nabla\cdot \mathcal{D}^{\alpha}v)\,\mathrm{d}x \\[10pt]
= - \frac{2\gamma}{a}\int\limits_{\mathbb{R}^3} \frac{p}{\varrho}|\nabla\cdot \mathcal{D}^{\alpha}v|^2\,\mathrm{d}x
- \frac{2\gamma}{a}\int\limits_{\mathbb{R}^3} (\nabla\cdot \mathcal{D}^{\alpha}v)\mathcal{D}^{\alpha}v\cdot\nabla(\frac{p}{\varrho}) \,\mathrm{d}x \\[10pt]
\leq C\sqrt{\ve}\mathcal{F}[v](t) - \frac{2\gamma}{a}\int\limits_{\mathbb{R}^3} \frac{p}{\varrho}|\nabla\cdot \mathcal{D}^{\alpha}v|^2\,\mathrm{d}x.
\end{array}
\end{equation}

Then
\begin{equation}\label{Sect3_Velocity_9}
\begin{array}{ll}
\frac{\mathrm{d}}{\mathrm{d}t} \int\limits_{\mathbb{R}^3}|\mathcal{D}^{\alpha}v|^2 \,\mathrm{d}x
+ \frac{2\gamma}{a}\int\limits_{\mathbb{R}^3} \frac{p}{\varrho}|\nabla\cdot \mathcal{D}^{\alpha}v|^2\,\mathrm{d}x
\lem \sqrt{\ve}\mathcal{F}[v](t).
\end{array}
\end{equation}

When $|\alpha|=4$, we estimate each term of $I_3$ separately. Assume the positive constants $\Lambda_{0,\alpha,1},\Lambda_{0,\alpha,2},\Lambda_{0,\alpha,3},\Lambda_{0,\alpha,4}$ are so small that 
$\sum\limits_{j=1}^{4} \Lambda_{0,\alpha,j}\leq \frac{\gamma}{a}$.

The first term of $I_3$:
\begin{equation}\label{Sect3_Velocity_10}
\begin{array}{ll}
2k_1(1-\gamma)\int\limits_{\mathbb{R}^3}\mathcal{D}^{\alpha} v\cdot\mathcal{D}^{\alpha}
[v(\nabla\cdot v)]\,\mathrm{d}x \\[6pt]

\leq C\sqrt{\ve}\mathcal{F}[v](t)
+ \Lambda_{0,\alpha,1}\int\limits_{\mathbb{R}^3}\frac{p}{\varrho}|\nabla\cdot \mathcal{D}^{\alpha}v|^2\,\mathrm{d}x
+ C\int\limits_{\mathbb{R}^3}\frac{\varrho}{p}(v\cdot \mathcal{D}^{\alpha}v)^2\,\mathrm{d}x

\\[6pt]
\leq C\sqrt{\ve}\mathcal{F}[v](t)
+ \Lambda_{0,\alpha,1}\int\limits_{\mathbb{R}^3}\frac{p}{\varrho}|\nabla\cdot \mathcal{D}^{\alpha}v|^2\,\mathrm{d}x.
\end{array}
\end{equation}

The second term of $I_3$:
\begin{equation}\label{Sect3_Velocity_11}
\begin{array}{ll}
-2 k_1\int\limits_{\mathbb{R}^3}\mathcal{D}^{\alpha} v\cdot\mathcal{D}^{\alpha}
(v\cdot\nabla v)\,\mathrm{d}x \\[6pt]

\leq C\sqrt{\ve}\mathcal{F}[v](t) + k_1\int\limits_{\mathbb{R}^3}\nabla\cdot v |\mathcal{D}^{\alpha}v|^2\,\mathrm{d}x
\leq C\sqrt{\ve}\mathcal{F}[v](t).
\end{array}
\end{equation}

The third term of $I_3$:
\begin{equation}\label{Sect3_Velocity_12}
\begin{array}{ll}
-k_1 \int\limits_{\mathbb{R}^3}\mathcal{D}^{\alpha} v\cdot\mathcal{D}^{\alpha}\nabla(|v|^2)\,\mathrm{d}x 

= k_1\int\limits_{\mathbb{R}^3}\nabla\cdot \mathcal{D}^{\alpha}v \mathcal{D}^{\alpha}(|v|^2)\,\mathrm{d}x \\[9pt]

\leq \Lambda_{0,\alpha,2}\int\limits_{\mathbb{R}^3}\frac{p}{\varrho}|\nabla\cdot \mathcal{D}^{\alpha}v|^2\,\mathrm{d}x
+ C\int\limits_{\mathbb{R}^3}\frac{\varrho}{p}|\mathcal{D}^{\alpha}(|v|^2)|^2\,\mathrm{d}x \\[9pt]

\leq C\sqrt{\ve}\mathcal{F}[v](t)
+ \Lambda_{0,\alpha,2}\int\limits_{\mathbb{R}^3}\frac{p}{\varrho}|\nabla\cdot \mathcal{D}^{\alpha}v|^2\,\mathrm{d}x.
\end{array}
\end{equation}

The fourth term of $I_3$:
\begin{equation}\label{Sect3_Velocity_13}
\begin{array}{ll}
\frac{2\gamma}{a}\int\limits_{\mathbb{R}^3}\sum\limits_{|\alpha_1|>0}
[\mathcal{D}^{\alpha_1}(\frac{p}{\varrho})\mathcal{D}^{\alpha_2}\nabla(\nabla\cdot v)]\cdot
\mathcal{D}^{\alpha} v \,\mathrm{d}x \\[12pt]

\leq \frac{2\gamma}{a}\sum\limits_{|\alpha_1|=1}
|\mathcal{D}^{\alpha_1}(\frac{p}{\varrho})|_{\infty}
\|\mathcal{D}^{\alpha_2}\nabla(\nabla\cdot v)\|_{L^2(\mathbb{R}^3)}
\|\mathcal{D}^{\alpha} v\|_{L^2(\mathbb{R}^3)} \\[10pt]\quad
+\frac{2\gamma}{a}\sum\limits_{|\alpha_1|=2}
|\mathcal{D}^{\alpha_1}(\frac{p}{\varrho})|_{\infty}
\|\mathcal{D}^{\alpha_2}\nabla(\nabla\cdot v)\|_{L^2(\mathbb{R}^3)}
\|\mathcal{D}^{\alpha} v\|_{L^2(\mathbb{R}^3)} \\[10pt]\quad
+\frac{2\gamma}{a}\sum\limits_{|\alpha_1|=3}
\|\mathcal{D}^{\alpha_1}(\frac{p}{\varrho})\|_{L^4(\mathbb{R}^3)}
\|\mathcal{D}^{\alpha_2}\nabla(\nabla\cdot v)\|_{L^4(\mathbb{R}^3)}
\|\mathcal{D}^{\alpha} v\|_{L^2(\mathbb{R}^3)} \\[10pt]\quad
+\frac{2\gamma}{a}|\nabla(\nabla\cdot v)|_{\infty}
\|\mathcal{D}^{\alpha}(\frac{p}{\varrho})\|_{L^2(\mathbb{R}^3)}
\|\mathcal{D}^{\alpha} v\|_{L^2(\mathbb{R}^3)} \\[10pt]

\leq C\sqrt{\ve}\mathcal{F}[v](t)
+ \Lambda_{0,\alpha,3}\int\limits_{\mathbb{R}^3}\frac{p}{\varrho}|\nabla\cdot \mathcal{D}^{\alpha}v|^2\,\mathrm{d}x.
\end{array}
\end{equation}

The fifth term of $I_3$:
\begin{equation}\label{Sect3_Velocity_14}
\begin{array}{ll}
\frac{2\gamma}{a}\int\limits_{\mathbb{R}^3} \frac{p}{\varrho}\mathcal{D}^{\alpha}v\cdot\nabla(\nabla\cdot \mathcal{D}^{\alpha}v)\,\mathrm{d}x \\[10pt]
= - \frac{2\gamma}{a}\int\limits_{\mathbb{R}^3} \frac{p}{\varrho}|\nabla\cdot \mathcal{D}^{\alpha}v|^2\,\mathrm{d}x
- \frac{2\gamma}{a}\int\limits_{\mathbb{R}^3} (\nabla\cdot \mathcal{D}^{\alpha}v)\mathcal{D}^{\alpha}v\cdot\nabla(\frac{p}{\varrho}) \,\mathrm{d}x \\[10pt]

\leq C\sqrt{\ve}\mathcal{F}[v] +(-\frac{2\gamma}{a}+\Lambda_{0,\alpha,4})\int\limits_{\mathbb{R}^3} \frac{p}{\varrho}|\nabla\cdot \mathcal{D}^{\alpha}v|^2\,\mathrm{d}x 
+ C\int\limits_{\mathbb{R}^3} [\mathcal{D}^{\alpha}v\cdot\nabla(\frac{p}{\varrho})]^2 \,\mathrm{d}x
 \\[10pt]
\leq C\sqrt{\ve}\mathcal{F}[v] +(-\frac{2\gamma}{a}+\Lambda_{0,\alpha,4})\int\limits_{\mathbb{R}^3} \frac{p}{\varrho}|\nabla\cdot \mathcal{D}^{\alpha}v|^2\,\mathrm{d}x.
\end{array}
\end{equation}

By $(\ref{Sect3_Velocity_10})+(\ref{Sect3_Velocity_11})+(\ref{Sect3_Velocity_12})+(\ref{Sect3_Velocity_13})+
(\ref{Sect3_Velocity_14})$, we have that when $|\alpha|=4$,
\begin{equation}\label{Sect3_Velocity_15}
\begin{array}{ll}
\frac{\mathrm{d}}{\mathrm{d}t} \int\limits_{\mathbb{R}^3}|\mathcal{D}^{\alpha}v|^2 \,\mathrm{d}x
+ \frac{\gamma}{a}\int\limits_{\mathbb{R}^3} \frac{p}{\varrho}|\mathcal{D}^{\alpha}\nabla\cdot v|^2\,\mathrm{d}x
\lem \sqrt{\ve}\mathcal{F}[v](t).
\end{array}
\end{equation}

By $(\ref{Sect3_Velocity_9})$ and $(\ref{Sect3_Velocity_15})$, we have that when $0\leq |\alpha|\leq 4$,
\begin{equation}\label{Sect3_Velocity_16}
\begin{array}{ll}
\frac{\mathrm{d}}{\mathrm{d}t} \sum\limits_{0\leq|\alpha|\leq 4}
\int\limits_{\mathbb{R}^3}|\mathcal{D}^{\alpha}v|^2 \,\mathrm{d}x
+ \frac{\gamma}{a}\sum\limits_{0\leq|\alpha|\leq 4}
\int\limits_{\mathbb{R}^3} \frac{p}{\varrho}|\mathcal{D}^{\alpha}\nabla\cdot v|^2\,\mathrm{d}x
\lem \sqrt{\ve}\mathcal{F}[v](t).
\end{array}
\end{equation}

Next, we estimate the $L^{\infty}$ bound of $\|v_t\|_{H^2(\mathbb{R}^3)}$: 

Let $\mathcal{D}^{\alpha} v_t\cdot\mathcal{D}^{\alpha}\partial_t(\ref{Sect2_Velocity_Solve})$,
where $0\leq|\alpha|\leq 2$, we get
\begin{equation}\label{Sect3_Velocity_T_2}
\begin{array}{ll}
\partial_t|\mathcal{D}^{\alpha} v_t|^2 = 2\mathcal{D}^{\alpha} v_t\cdot\mathcal{D}^{\alpha}\partial_t
[k_1(1-\gamma)v(\nabla\cdot v)- k_1 v\cdot\nabla v - \frac{k_1}{2} \nabla(|v|^2)] \\[7pt]\hspace{1.9cm}

+ \frac{2\gamma}{a}\sum\limits_{|\alpha_1|>0}
[\mathcal{D}^{\alpha_1}(\frac{p}{\varrho})\mathcal{D}^{\alpha_2}\nabla(\nabla\cdot v_t)]\cdot\mathcal{D}^{\alpha} v_t \\[10pt]\hspace{1.9cm}
+ \frac{2\gamma}{a}\sum\limits_{|\alpha_1|\geq 0}
[\mathcal{D}^{\alpha_1}\partial_t(\frac{p}{\varrho})\mathcal{D}^{\alpha_2}\nabla(\nabla\cdot v)]\cdot\mathcal{D}^{\alpha} v_t \\[10pt]\hspace{1.9cm}
+ \frac{2\gamma p}{a\varrho}\mathcal{D}^{\alpha}\nabla(\nabla\cdot v_t)\cdot\mathcal{D}^{\alpha} v_t.
\end{array}
\end{equation}

Integrate $(\ref{Sect3_Velocity_T_2})$ in $\mathbb{R}^3$, we get
\begin{equation}\label{Sect3_Velocity_T_3}
\begin{array}{ll}
\frac{\mathrm{d}}{\mathrm{d}t}\int\limits_{\mathbb{R}^3}|\mathcal{D}^{\alpha} v_t|^2 \,\mathrm{d}x
\\[10pt]
= 2\int\limits_{\mathbb{R}^3}\mathcal{D}^{\alpha} v_t\cdot\mathcal{D}^{\alpha}\partial_t
[k_1(1-\gamma)v(\nabla\cdot v)- k_1 v\cdot\nabla v
- \frac{k_1}{2} \nabla(|v|^2)]\,\mathrm{d}x \\[10pt]\quad

+ \frac{2\gamma}{a}\int\limits_{\mathbb{R}^3}\sum\limits_{|\alpha_1|>0}
[\mathcal{D}^{\alpha_1}(\frac{p}{\varrho})\mathcal{D}^{\alpha_2}\nabla(\nabla\cdot v_t)]\cdot
\mathcal{D}^{\alpha} v_t \,\mathrm{d}x
\\[10pt]\quad
+ \frac{2\gamma}{a}\int\limits_{\mathbb{R}^3}\sum\limits_{|\alpha_1|\geq 0}
[\mathcal{D}^{\alpha_1}\partial_t(\frac{p}{\varrho})\mathcal{D}^{\alpha_2}\nabla(\nabla\cdot v)]\cdot
\mathcal{D}^{\alpha} v_t \,\mathrm{d}x
\\[10pt]\quad

+ \frac{2\gamma}{a}\int\limits_{\mathbb{R}^3}\frac{p}{\varrho}\mathcal{D}^{\alpha}\nabla(\nabla\cdot v_t)
\cdot\mathcal{D}^{\alpha} v_t \,\mathrm{d}x := I_4.
\end{array}
\end{equation}

When $0\leq|\alpha|\leq 1$,
\begin{equation}\label{Sect3_Velocity_T_8}
\begin{array}{ll}
\frac{2\gamma}{a}\int\limits_{\mathbb{R}^3} \frac{p}{\varrho}\mathcal{D}^{\alpha}v_t\cdot\nabla(\nabla\cdot \mathcal{D}^{\alpha}v_t)\,\mathrm{d}x \\[10pt]
= - \frac{2\gamma}{a}\int\limits_{\mathbb{R}^3} \frac{p}{\varrho}|\nabla\cdot \mathcal{D}^{\alpha}v_t|^2\,\mathrm{d}x
- \frac{2\gamma}{a}\int\limits_{\mathbb{R}^3} (\nabla\cdot \mathcal{D}^{\alpha}v_t)\mathcal{D}^{\alpha}v_t\cdot\nabla(\frac{p}{\varrho}) \,\mathrm{d}x \\[10pt]
\leq C\sqrt{\ve}\mathcal{F}[v](t) - \frac{2\gamma}{a}\int\limits_{\mathbb{R}^3} \frac{p}{\varrho}|\nabla\cdot \mathcal{D}^{\alpha}v_t|^2\,\mathrm{d}x.
\end{array}
\end{equation}

Then
\begin{equation}\label{Sect3_Velocity_T_9}
\begin{array}{ll}
\frac{\mathrm{d}}{\mathrm{d}t} \int\limits_{\mathbb{R}^3}|\mathcal{D}^{\alpha}v_t|^2 \,\mathrm{d}x
+ \frac{2\gamma}{a}\int\limits_{\mathbb{R}^3} \frac{p}{\varrho}|\nabla\cdot \mathcal{D}^{\alpha}v_t|^2\,\mathrm{d}x
\lem \sqrt{\ve}\mathcal{F}[v](t).
\end{array}
\end{equation}

When $|\alpha|=2$, we estimate each term of $I_4$ separately. Assume the positive constants $\Lambda_{1,\alpha,1},\Lambda_{1,\alpha,2},\Lambda_{1,\alpha,3},\Lambda_{1,\alpha,4}$ are so small that $\sum\limits_{j=1}^{4} \Lambda_{1,\alpha,j}\leq \frac{\gamma}{a}$.

The first term of $I_4$:
\begin{equation}\label{Sect3_Velocity_T_10}
\begin{array}{ll}
2k_1(1-\gamma)\int\limits_{\mathbb{R}^3}\mathcal{D}^{\alpha} v_t\cdot\mathcal{D}^{\alpha}\partial_t
[v(\nabla\cdot v)]\,\mathrm{d}x \\[9pt]
=2k_1(1-\gamma)\left(\int\limits_{\mathbb{R}^3}\mathcal{D}^{\alpha} v_t\cdot\mathcal{D}^{\alpha}
[v_t(\nabla\cdot v)]\,\mathrm{d}x 
+\int\limits_{\mathbb{R}^3}\mathcal{D}^{\alpha} v_t\cdot\mathcal{D}^{\alpha}
[v(\nabla\cdot v_t)]\,\mathrm{d}x\right) \\[9pt]

\leq C\sqrt{\ve}\mathcal{F}[v](t)
+ \Lambda_{1,\alpha,1}\int\limits_{\mathbb{R}^3}\frac{p}{\varrho}|\nabla\cdot \mathcal{D}^{\alpha}v_t|^2\,\mathrm{d}x
+ C\int\limits_{\mathbb{R}^3}\frac{\varrho}{p}(v\cdot \mathcal{D}^{\alpha}v_t)^2\,\mathrm{d}x

\\[6pt]
\leq C\sqrt{\ve}\mathcal{F}[v](t)
+ \Lambda_{1,\alpha,1}\int\limits_{\mathbb{R}^3}\frac{p}{\varrho}|\nabla\cdot \mathcal{D}^{\alpha}v_t|^2\,\mathrm{d}x.
\end{array}
\end{equation}

The second term of $I_4$:
\begin{equation}\label{Sect3_Velocity_T_11}
\begin{array}{ll}
-2k_1 \int\limits_{\mathbb{R}^3}\mathcal{D}^{\alpha} v_t\cdot\mathcal{D}^{\alpha}\partial_t
(v\cdot\nabla v)\,\mathrm{d}x \\[9pt]

-2k_1 \left(\int\limits_{\mathbb{R}^3}\mathcal{D}^{\alpha} v_t\cdot\mathcal{D}^{\alpha}
(v_t\cdot\nabla v)\,\mathrm{d}x
+\int\limits_{\mathbb{R}^3}\mathcal{D}^{\alpha} v_t\cdot\mathcal{D}^{\alpha}
(v\cdot\nabla v_t)\,\mathrm{d}x \right) \\[12pt]

\leq C\sqrt{\ve}\mathcal{F}[v](t) + k_1\int\limits_{\mathbb{R}^3}\nabla\cdot v |\mathcal{D}^{\alpha}v_t|^2\,\mathrm{d}x \\[9pt]
\leq C\sqrt{\ve}\mathcal{F}[v](t).
\end{array}
\end{equation}

The third term of $I_4$:
\begin{equation}\label{Sect3_Velocity_T_12}
\begin{array}{ll}
-k_1\int\limits_{\mathbb{R}^3}\mathcal{D}^{\alpha} v_t\cdot\mathcal{D}^{\alpha}\partial_t \nabla(|v|^2)\,\mathrm{d}x 
= 2k_1\int\limits_{\mathbb{R}^3}\nabla\cdot \mathcal{D}^{\alpha}v_t \mathcal{D}^{\alpha}(v\cdot v_t)\,\mathrm{d}x \\[9pt]

\leq \Lambda_{1,\alpha,2}\int\limits_{\mathbb{R}^3}\frac{p}{\varrho}|\nabla\cdot \mathcal{D}^{\alpha}v_t|^2\,\mathrm{d}x
+ C\int\limits_{\mathbb{R}^3}\frac{\varrho}{p}|\mathcal{D}^{\alpha}(v\cdot v_t)|^2\,\mathrm{d}x \\[9pt]

\leq C\sqrt{\ve}\mathcal{F}[v](t)
+ \Lambda_{1,\alpha,2}\int\limits_{\mathbb{R}^3}\frac{p}{\varrho}|\nabla\cdot \mathcal{D}^{\alpha}v|^2\,\mathrm{d}x.
\end{array}
\end{equation}

The fourth term of $I_4$:
\begin{equation}\label{Sect3_Velocity_T_13}
\begin{array}{ll}
\frac{2\gamma}{a}\int\limits_{\mathbb{R}^3}\sum\limits_{|\alpha_1|>0}
[\mathcal{D}^{\alpha_1}(\frac{p}{\varrho})\mathcal{D}^{\alpha_2}\nabla(\nabla\cdot v_t)]\cdot
\mathcal{D}^{\alpha} v_t \,\mathrm{d}x
\\[14pt]

\leq \frac{2\gamma}{a}\sum\limits_{|\alpha_1|=1}
|\mathcal{D}^{\alpha_1}(\frac{p}{\varrho})|_{\infty}
\|\mathcal{D}^{\alpha_2}\nabla(\nabla\cdot v_t)\|_{L^2(\mathbb{R}^3)}
\|\mathcal{D}^{\alpha} v\|_{L^2(\mathbb{R}^3)} \\[10pt]\quad
+\frac{2\gamma}{a}
|\mathcal{D}^{\alpha}(\frac{p}{\varrho})|_{\infty}
\|\nabla(\nabla\cdot v_t)\|_{L^2(\mathbb{R}^3)}
\|\mathcal{D}^{\alpha} v\|_{L^2(\mathbb{R}^3)} \\[10pt]

\leq C\sqrt{\ve}\mathcal{F}[v](t)
+ \Lambda_{1,\alpha,3}\int\limits_{\mathbb{R}^3}\frac{p}{\varrho}|\nabla\cdot \mathcal{D}^{\alpha}v_t|^2\,\mathrm{d}x.
\end{array}
\end{equation}

The fifth term of $I_4$:
\begin{equation}\label{Sect3_Velocity_T_14}
\begin{array}{ll}
\frac{2\gamma}{a}\int\limits_{\mathbb{R}^3}\sum\limits_{|\alpha_1|\geq 0}
[\mathcal{D}^{\alpha_1}\partial_t(\frac{p}{\varrho})\mathcal{D}^{\alpha_2}\nabla(\nabla\cdot v)]\cdot
\mathcal{D}^{\alpha} v_t \,\mathrm{d}x \\[14pt]

\leq \frac{2\gamma}{a}
|\partial_t(\frac{p}{\varrho})|_{\infty}
\|\mathcal{D}^{\alpha}\nabla(\nabla\cdot v)\|_{L^2(\mathbb{R}^3)}
\|\mathcal{D}^{\alpha} v\|_{L^2(\mathbb{R}^3)} \\[10pt]\quad
+ \frac{2\gamma}{a}\sum\limits_{|\alpha_1|=1}
\|\mathcal{D}^{\alpha_1}\partial_t(\frac{p}{\varrho})\|_{L^4(\mathbb{R}^3)}
\|\mathcal{D}^{\alpha_2}\nabla(\nabla\cdot v)\|_{L^4(\mathbb{R}^3)}
\|\mathcal{D}^{\alpha} v\|_{L^2(\mathbb{R}^3)} \\[13pt]\quad
+\frac{2\gamma}{a}|\nabla(\nabla\cdot v)|_{\infty}
\|\mathcal{D}^{\alpha}\partial_t(\frac{p}{\varrho})\|_{L^2(\mathbb{R}^3)}
\|\mathcal{D}^{\alpha} v\|_{L^2(\mathbb{R}^3)} 
\\[9pt]
\leq C\sqrt{\ve}\mathcal{F}[v](t).
\end{array}
\end{equation}

The sixth term of $I_4$:
\begin{equation}\label{Sect3_Velocity_T_15}
\begin{array}{ll}
\frac{2\gamma}{a}\int\limits_{\mathbb{R}^3}\frac{p}{\varrho}\mathcal{D}^{\alpha}\nabla(\nabla\cdot v_t)
\cdot\mathcal{D}^{\alpha} v_t \,\mathrm{d}x \\[10pt]

= - \frac{2\gamma}{a}\int\limits_{\mathbb{R}^3} \frac{p}{\varrho}|\nabla\cdot \mathcal{D}^{\alpha}v_t|^2\,\mathrm{d}x
- \frac{2\gamma}{a}\int\limits_{\mathbb{R}^3} (\nabla\cdot \mathcal{D}^{\alpha}v_t)\mathcal{D}^{\alpha}v_t\cdot\nabla(\frac{p}{\varrho}) \,\mathrm{d}x \\[10pt]

\leq C\sqrt{\ve}\mathcal{F}[v] +(-\frac{2\gamma}{a}+\Lambda_{1,\alpha,4})\int\limits_{\mathbb{R}^3} \frac{p}{\varrho}|\nabla\cdot \mathcal{D}^{\alpha}v_t|^2\,\mathrm{d}x
+ C\int\limits_{\mathbb{R}^3} [\mathcal{D}^{\alpha}v_t\cdot\nabla(\frac{p}{\varrho})]^2 \,\mathrm{d}x
 \\
\leq C\sqrt{\ve}\mathcal{F}[v] +(-\frac{2\gamma}{a}+\Lambda_{1,\alpha,4})\int\limits_{\mathbb{R}^3} \frac{p}{\varrho}|\nabla\cdot \mathcal{D}^{\alpha}v_t|^2\,\mathrm{d}x.
\end{array}
\end{equation}

By $(\ref{Sect3_Velocity_T_10})+(\ref{Sect3_Velocity_T_11})+(\ref{Sect3_Velocity_T_12})+(\ref{Sect3_Velocity_T_13})+
(\ref{Sect3_Velocity_T_14})+(\ref{Sect3_Velocity_T_15})$, we have that when $|\alpha|=2$,
\begin{equation}\label{Sect3_Velocity_T_16}
\begin{array}{ll}
\frac{\mathrm{d}}{\mathrm{d}t} \int\limits_{\mathbb{R}^3}|\mathcal{D}^{\alpha}v_t|^2 \,\mathrm{d}x
+ \frac{\gamma}{a}\int\limits_{\mathbb{R}^3} \frac{p}{\varrho}|\mathcal{D}^{\alpha}\nabla\cdot v_t|^2\,\mathrm{d}x
\lem \sqrt{\ve}\mathcal{F}[v](t).
\end{array}
\end{equation}

By $(\ref{Sect3_Velocity_T_9})$ and $(\ref{Sect3_Velocity_T_16})$, we have that when $0\leq |\alpha|\leq 2$,
\begin{equation}\label{Sect3_Velocity_T_17}
\begin{array}{ll}
\frac{\mathrm{d}}{\mathrm{d}t} \sum\limits_{0\leq|\alpha|\leq 2}\
\int\limits_{\mathbb{R}^3}|\mathcal{D}^{\alpha}v_t|^2 \,\mathrm{d}x
+ \frac{\gamma}{a}\sum\limits_{0\leq|\alpha|\leq 2}\
\int\limits_{\mathbb{R}^3} \frac{p}{\varrho}|\mathcal{D}^{\alpha}\nabla\cdot v_t|^2\,\mathrm{d}x
\lem \sqrt{\ve}\mathcal{F}[v](t).
\end{array}
\end{equation}

Sum $(\ref{Sect3_Velocity_16})+(\ref{Sect3_Velocity_T_17})$, we obtain $(\ref{Sect3_Velocity_toProve})$.
Thus, Lemma $\ref{Sect3_Velocity_Lemma}$ is proved.
\end{proof}

The structure of the equations $(\ref{Sect2_Relaxing_Eq})$ implies
$\mathcal{F}_X[\xi](t)$ can be estimates by $\mathcal{F}[v](t)$, as the following lemma stated:
\begin{lemma}\label{Sect3_Energy_Ineq}
For any given $T\in (0,+\infty)$, if
\begin{equation*}
\sup\limits_{0\leq t\leq T} \mathcal{F}[\xi,v,\phi,\zeta](t) \leq\ve,
\end{equation*}
there exists $\ve_2>0$, which is independent of $(\xi_0,v_0,\phi_0)$, such that if $0<\ve\ll \min\{1,\ve_2\}$, then for $\forall t\in [0,T]$,
\begin{equation}\label{Sect3_Energy_Ineq_toProve}
\mathcal{F}_X[\xi](t) \leq c_3\mathcal{F}[v](t),
\end{equation}
for some $c_3>0$.
\end{lemma}

\begin{proof}
Apply $\mathcal{D}^{\alpha}$ to $\nabla\xi = - k_1\varrho v$, where $0\leq |\alpha|\leq 3$, we get
\begin{equation}\label{Sect3_Energy_Ineq_1}
\begin{array}{ll}
\mathcal{D}^{\alpha}\nabla\xi = - k_1 \mathcal{D}^{\alpha}(\varrho v), \\[6pt]

\|\mathcal{D}^{\alpha}\nabla\xi\|_{L^2(\mathbb{R}^3)}^2 \lem
\|\mathcal{D}^{\alpha}(\varrho v)\|_{L^2(\mathbb{R}^3)}^2 
\lem \mathcal{F}[v](t) + \mathcal{F}[\zeta]\mathcal{F}[v](t).
\end{array}
\end{equation}

Apply $\mathcal{D}^{\alpha}$ to $\xi_t = - k_1v\cdot\nabla\xi - k_1\gamma p\nabla\cdot v$,
where $0\leq|\alpha|\leq 2$, we get
\begin{equation}\label{Sect3_Energy_Ineq_2}
\begin{array}{ll}
\mathcal{D}^{\alpha}\xi_t = - k_1\mathcal{D}^{\alpha}(v\cdot\nabla\xi)
- k_1\gamma \mathcal{D}^{\alpha}(p\nabla\cdot v), \\[6pt]

\|\mathcal{D}^{\alpha}\xi_t\|_{L^2(\mathbb{R}^3)}^2 \lem
\|\mathcal{D}^{\alpha}(v\cdot\nabla\xi)\|_{L^2(\mathbb{R}^3)}^2
+ \|\mathcal{D}^{\alpha}(p \nabla\cdot v)\|_{L^2(\mathbb{R}^3)}^2 \\[6pt]\hspace{2.05cm}

\lem \mathcal{F}[v](t)+ \mathcal{F}_X[\xi]\mathcal{F}[v](t).
\end{array}
\end{equation}

By $(\ref{Sect3_Energy_Ineq_1})+(\ref{Sect3_Energy_Ineq_2})$, we have
\begin{equation}\label{Sect3_Energy_Ineq_3}
\begin{array}{ll}
\mathcal{F}_X[\xi](t) \leq C_1\mathcal{F}[v](t) + C_1\mathcal{F}[\zeta]\mathcal{F}[v](t)
+ C_1\mathcal{F}_X[\xi]\mathcal{F}[v](t) \\[5pt]\hspace{1.3cm}
\leq C_1\mathcal{F}_X[\xi]\mathcal{F}[v](t)
+ C_1(1+ \ve)\mathcal{F}[v](t).
\end{array}
\end{equation}
for some $C_1>0$.

Let $\ve_2\leq \frac{1}{2C_1}$, $c_3=2C_1(1+\ve_2)\leq 2C_1+1$,
then $\mathcal{F}_X[\xi](t) \leq c_3\mathcal{F}[v](t)$.
Thus, Lemma $\ref{Sect3_Energy_Ineq}$ is proved.
\end{proof}

\begin{remark}\label{Sect3_Higher_Order_Xi_Remark}
The proof of Lemma $\ref{Sect3_Energy_Ineq}$ indicates that under the small data assumption 
\begin{equation*}
\sup\limits_{0\leq t\leq T} \mathcal{F}[\xi,v,\phi,\zeta](t) \leq\e,
\end{equation*}
for some sufficiently small $\ve$, we have
\begin{equation*}
\begin{array}{ll}
\|\nabla\xi\|_{H^4(\mathbb{R}^3)}^2
+\|\xi_t\|_{H^3(\mathbb{R}^3)}^2 \lem \mathcal{F}[v](t)+\mathcal{F}[\zeta](t), \\[6pt]
\mathcal{F}[v](t)\lem \tilde{\mathcal{F}}[\xi](t)+\mathcal{F}[\zeta](t),
\end{array}
\end{equation*}
for $\forall t\in [0,T]$.
\end{remark}

The following lemma gives not only the $L^{\infty}$ bounds of $\mathcal{F}_1[\xi](t),\mathcal{F}[v](t)$, but also the bound of $\int\limits_{0}^{T} \mathcal{F}[\nabla\cdot v](s)\,\mathrm{d}s$.
\begin{lemma}\label{Sect3_Decay_Lemma}
For any given $T\in (0,+\infty)$, there exists $\ve_3>0$ which is independent of $(\xi_0,v_0,\phi_0)$, such that if
\begin{equation*}
\sup\limits_{0\leq t\leq T} \mathcal{F}[\xi,v,\phi,\zeta](t) \leq\ve,
\end{equation*}
where $\ve\ll \min\{1,\ve_1,\ve_2,\ve_3\}$, then for $\forall t\in [0,T]$,
\begin{equation}\label{Sect3_Decay_Estimates_toProve}
\begin{array}{ll}
\mathcal{F}[\xi](t)+\mathcal{F}[v](t)\leq \beta_1(\|\xi_0\|_{H^5(\mathbb{R}^3)}+\|\phi_0\|_{H^4(\mathbb{R}^3)}), \\[6pt]
\int\limits_{0}^{T} \mathcal{F}[v](s)\,\mathrm{d}s 
+\int\limits_{0}^{T} \mathcal{F}[\nabla\cdot v](s)\,\mathrm{d}s 
\leq \beta_2(\|\xi_0\|_{H^5(\mathbb{R}^3)}+\|\phi_0\|_{H^4(\mathbb{R}^3)}),
\end{array}
\end{equation}
for some $\beta_1>0,\beta_2>0$.
\end{lemma}

\begin{proof}
In view of Lemmas $\ref{Sect3_Energy_Estimate_Lemma2}$, $\ref{Sect3_Velocity_Lemma}$ and $\ref{Sect3_Energy_Ineq}$, we have a priori estimates as follows:
\begin{equation}\label{Sect3_Obtained_A_Priori_Estimates}
\left\{\begin{array}{ll}
\frac{\mathrm{d}}{\mathrm{d}t}\mathcal{F}_1[\xi](t)
+ 2 \mathcal{F}_1[v](t) \leq C_2\sqrt{\e}(\mathcal{F}_X[\xi](t) + \mathcal{F}[v](t)), \\[13pt]

\frac{\mathrm{d}}{\mathrm{d}t} \mathcal{F}[v](t)
+ \frac{\gamma}{a}\int\limits_{\mathbb{R}^3} \frac{p}{\varrho}\sum\limits_{0\leq|\alpha|\leq 4}|\nabla\cdot \mathcal{D}^{\alpha}v|^2\,\mathrm{d}x \\[14pt]\hspace{1.5cm}
+ \frac{\gamma}{a}\int\limits_{\mathbb{R}^3} \frac{p}{\varrho}\sum\limits_{0\leq|\alpha|\leq 2}|\nabla\cdot \mathcal{D}^{\alpha}v_t|^2\,\mathrm{d}x 
\leq C_2\sqrt{\ve}\mathcal{F}[v](t), \\[15pt]

\mathcal{F}_X[\xi](t) \leq c_3 \mathcal{F}[v](t).
\end{array}\right.
\end{equation}

By $(\ref{Sect3_Obtained_A_Priori_Estimates})_1+(\ref{Sect3_Obtained_A_Priori_Estimates})_2$, we get
\begin{equation}\label{Sect3_Decay_1}
\begin{array}{ll}
\frac{\mathrm{d}}{\mathrm{d}t}(\mathcal{F}_1[\xi](t) + \mathcal{F}[v](t))
+ 2 \mathcal{F}[v](t) \\[8pt]\quad
+ \frac{\gamma}{a}\int\limits_{\mathbb{R}^3} \frac{p}{\varrho}\sum\limits_{0\leq|\alpha|\leq 4}|\nabla\cdot \mathcal{D}^{\alpha}v|^2\,\mathrm{d}x
+ \frac{\gamma}{a}\int\limits_{\mathbb{R}^3} \frac{p}{\varrho}\sum\limits_{0\leq|\alpha|\leq 2}|\nabla\cdot \mathcal{D}^{\alpha}v_t|^2\,\mathrm{d}x
\\[12pt]
\leq C_2\sqrt{\e}(\mathcal{F}_X[\xi](t) + \mathcal{F}[v](t)).
\end{array}
\end{equation}

plug $(\ref{Sect3_Obtained_A_Priori_Estimates})_3$ into $(\ref{Sect3_Decay_1})$, we get
\begin{equation}\label{Sect3_Decay_2}
\begin{array}{ll}
[R.H.S.\ of\ (\ref{Sect3_Decay_1})] \leq C_2(1+c_3)\sqrt{\e}\mathcal{F}[v](t).
\end{array}
\end{equation}

Take $\ve_3 =\frac{1}{C_2^2(1+c_3)^2}$, when $\ve\leq \ve_3$, we have
\begin{equation}\label{Sect3_Decay_3}
\begin{array}{ll}
\frac{\mathrm{d}}{\mathrm{d}t}(\mathcal{F}_1[\xi](t) + \mathcal{F}[v](t))
+ \mathcal{F}[v](t) \\[8pt]\quad
+ \frac{\gamma}{a}\int\limits_{\mathbb{R}^3} \frac{p}{\varrho}\sum\limits_{0\leq|\alpha|\leq 4}|\nabla\cdot \mathcal{D}^{\alpha}v|^2\,\mathrm{d}x
+ \frac{\gamma}{a}\int\limits_{\mathbb{R}^3} \frac{p}{\varrho}\sum\limits_{0\leq|\alpha|\leq 2}|\nabla\cdot \mathcal{D}^{\alpha}v_t|^2\,\mathrm{d}x \leq 0.
\end{array}
\end{equation}

Integrate $(\ref{Sect3_Decay_3})$ in $\mathbb{R}^3$, when $t\in [0,T]$, 
\begin{equation}\label{Sect3_Decay_4}
\begin{array}{ll}
\mathcal{F}_1[\xi](t) + \mathcal{F}[v](t) + \int\limits_0^{t}\mathcal{F}[v](s)\,\mathrm{d}s
\\[8pt]\quad

+ \frac{\gamma}{a}\int\limits_{0}^{t}\int\limits_{\mathbb{R}^3} 
\frac{p}{\varrho}\sum\limits_{0\leq|\alpha|\leq 4}|\nabla\cdot \mathcal{D}^{\alpha}v|^2
\,\mathrm{d}x \mathrm{d}s
+ \frac{\gamma}{a}\int\limits_{0}^{t}\int\limits_{\mathbb{R}^3} 
\frac{p}{\varrho}\sum\limits_{0\leq|\alpha|\leq 2}|\nabla\cdot \mathcal{D}^{\alpha}v_t|^2
\,\mathrm{d}x \mathrm{d}s \\[16pt]

\leq \mathcal{F}_1[\xi](0) + \mathcal{F}[v](0) 
\leq C_3(\|\xi_0\|_{H^5(\mathbb{R}^3)}^2 + \|\phi_0\|_{H^4(\mathbb{R}^3)}^2),
\end{array}
\end{equation}
for some $C_3>0$.

Take $\beta_1=\frac{C_3}{\min\{c_1,1\}}, \beta_2 =\cfrac{C_3}{\min\{1,\frac{4\gamma\bar{p}}{3a\bar{\varrho}}\}}$.
Thus, Lemma $\ref{Sect3_Decay_Lemma}$ is proved.
\end{proof}

The following lemma concerns the $L^{\infty}$ bound of $\mathcal{F}[\phi](t)$ in $[0,T]$.
\begin{lemma}\label{Sect3_Entropy_Lemma}
For any given $T\in (0,+\infty)$, if
\begin{equation*}
\sup\limits_{0\leq t\leq T} \mathcal{F}[\xi,v,\phi,\zeta](t) \leq\ve,
\end{equation*}
where $0<\ve\ll \min\{1,\ve_1,\ve_2,\ve_3\}$, then for $\forall t\in [0,T]$,
\begin{equation}\label{Sect3_Entropy_toProve_1}
\frac{\mathrm{d}}{\mathrm{d}t}\mathcal{F}[\phi](t)\leq \beta_3\mathcal{F}[v](t)^{\frac{1}{2}}\mathcal{F}[\phi](t),
\end{equation}
and $\mathcal{F}[\phi](t)$ has $L^{\infty}$ bound in $[0,T]$:
\begin{equation}\label{Sect3_Entropy_toProve_2}
\mathcal{F}[\phi](t) \leq C_6\|\phi_0\|_{H^4(\mathbb{R}^3)}^2\exp\{\beta_2\beta_3T(\|\xi_0\|_{H^5(\mathbb{R}^3)}^2
+ \|\phi_0\|_{H^3(\mathbb{R}^4)}^2)\},
\end{equation}
for some $\beta_3>0,C_6>0$.
\end{lemma}

\begin{proof}
Let $\mathcal{D}^{\alpha}\phi\cdot\mathcal{D}^{\alpha}(\ref{Sect2_Relaxed_Eq})_3$, where $0\leq |\alpha|\leq 4$, we get
\begin{equation}\label{Sect3_Entropy_1}
\begin{array}{ll}
(|\mathcal{D}^{\alpha}\phi|^2)_t = -2\sum\limits_{|\alpha_1|>0}\mathcal{D}^{\alpha}\phi
\mathcal{D}^{\alpha_1}v\cdot\nabla(\mathcal{D}^{\alpha_2}\phi) -v\cdot\nabla|\mathcal{D}^{\alpha}\phi|^2.\\[6pt]
\end{array}
\end{equation}

Integrate $(\ref{Sect3_Entropy_1})$ in $\mathbb{R}^3$, we have
\begin{equation}\label{Sect3_Entropy_2}
\begin{array}{ll}
\frac{\mathrm{d}}{\mathrm{d}t}\int\limits_{\mathbb{R}^3}
|\mathcal{D}^{\alpha}\phi|^2 \,\mathrm{d}x \\[10pt]
= -2\sum\limits_{|\alpha_1|>0}\int\limits_{\mathbb{R}^3}\mathcal{D}^{\alpha}\phi
\mathcal{D}^{\alpha_1}v\cdot\nabla(\mathcal{D}^{\alpha_2}\phi) \,\mathrm{d}x

+ \int\limits_{\mathbb{R}^3}|\mathcal{D}^{\alpha}\phi|^2 \nabla\cdot v \,\mathrm{d}x
:= I_5.
\end{array}
\end{equation}

When $0\leq|\alpha|\leq 4$, it is easy to check $I_5 \lem \mathcal{F}[v](t)^{\frac{1}{2}}\mathcal{F}[\phi](t)$.
Sum $\alpha$, we have
\begin{equation}\label{Sect3_Entropy_Prove_3}
\begin{array}{ll}
\frac{\mathrm{d}}{\mathrm{d}t}\int\limits_{\mathbb{R}^3}
\sum\limits_{0\leq |\alpha|\leq 4}|\mathcal{D}^{\alpha}\phi|^2 \,\mathrm{d}x
\leq C_4\mathcal{F}[v](t)^{\frac{1}{2}}\mathcal{F}[\phi](t).
\end{array}
\end{equation}

Let $\mathcal{D}^{\alpha}\phi_t\cdot\mathcal{D}^{\alpha}\partial_t(\ref{Sect2_Relaxed_Eq})_3$, where $0\leq |\alpha|\leq 2$, we get
\begin{equation}\label{Sect3_Entropy_T_1}
\begin{array}{ll}
(|\mathcal{D}^{\alpha}\phi_t|^2)_t 
= -2\sum\limits_{|\alpha_1|\geq 0}\mathcal{D}^{\alpha}\phi_t
\mathcal{D}^{\alpha_1}v_t\cdot\nabla(\mathcal{D}^{\alpha_2}\phi) \\[10pt]\hspace{2cm}
-2\sum\limits_{|\alpha_1|>0}\mathcal{D}^{\alpha}\phi_t
\mathcal{D}^{\alpha_1}v\cdot\nabla(\mathcal{D}^{\alpha_2}\phi_t)
-v\cdot\nabla|\mathcal{D}^{\alpha}\phi_t|^2.\\[6pt]
\end{array}
\end{equation}

Integrate $(\ref{Sect3_Entropy_T_1})$ in $\mathbb{R}^3$, we have
\begin{equation}\label{Sect3_Entropy_T_2}
\begin{array}{ll}
\frac{\mathrm{d}}{\mathrm{d}t}\int\limits_{\mathbb{R}^3}
|\partial_t^{\ell}\mathcal{D}^{\alpha}\phi|^2 \,\mathrm{d}x \\[10pt]
= -2\sum\limits_{|\alpha_1|\geq 0}\int\limits_{\mathbb{R}^3}\mathcal{D}^{\alpha}\phi_t
\mathcal{D}^{\alpha_1}v_t\cdot\nabla(\mathcal{D}^{\alpha_2}\phi) \,\mathrm{d}x

-2\sum\limits_{|\alpha_1|>0}\int\limits_{\mathbb{R}^3}\mathcal{D}^{\alpha}\phi_t
\mathcal{D}^{\alpha_1}v\cdot\nabla(\mathcal{D}^{\alpha_2}\phi_t) \,\mathrm{d}x
\\[10pt]\quad
+ \int\limits_{\mathbb{R}^3}|\mathcal{D}^{\alpha}\phi_t|^2 \nabla\cdot v \,\mathrm{d}x
:= I_6.
\end{array}
\end{equation}

When $0\leq|\alpha|\leq 2$, it is easy to check $I_6 \lem \mathcal{F}[v](t)^{\frac{1}{2}}\mathcal{F}[\phi](t)$.
Sum $\alpha$, we have
\begin{equation}\label{Sect3_Entropy_T_Prove_3}
\begin{array}{ll}
\frac{\mathrm{d}}{\mathrm{d}t}\int\limits_{\mathbb{R}^3}\sum\limits_{0\leq|\alpha|\leq 2}
|\partial_t^{\ell}\mathcal{D}^{\alpha}\phi|^2 \,\mathrm{d}x
\leq C_5\mathcal{F}[v](t)^{\frac{1}{2}}\mathcal{F}[\phi](t).
\end{array}
\end{equation}

Let $\beta_3=C_4+C_5$, sum $(\ref{Sect3_Entropy_Prove_3})$ and $(\ref{Sect3_Entropy_T_Prove_3})$, we have
\begin{equation}\label{Sect3_Entropy_Prove_4}
\begin{array}{ll}
\frac{\mathrm{d}}{\mathrm{d}t}\mathcal{F}[\phi](t)
\leq \beta_3\mathcal{F}[v](t)^{\frac{1}{2}}\mathcal{F}[\phi](t).
\end{array}
\end{equation}

Since $T\in (0,+\infty)$ is finite, integrate $(\ref{Sect3_Entropy_T_Prove_3})$ from $0$ to $t$, when $t\in [0,T]$, we    obtain the a priori estimate for $\mathcal{F}[\phi](t)$.
\begin{equation}\label{Sect3_Entropy_T_Prove_4}
\begin{array}{ll}
\mathcal{F}[\phi](t)\leq C_6\|\phi_0\|_{H^4(\mathbb{R}^3)}^2\exp\{\int\limits_0^{t}\beta_3\mathcal{F}[v](s)^{\frac{1}{2}}\,\mathrm{d}s\} \\[9pt]\hspace{1.1cm}
\leq C_6\|\phi_0\|_{H^4(\mathbb{R}^3)}^2\exp\{\beta_3T
\int\limits_0^{T}\mathcal{F}[v](s)\,\mathrm{d}s\} \\[9pt]\hspace{1.1cm}

\leq C_6\|\phi_0\|_{H^4(\mathbb{R}^3)}^2\exp\{\beta_2\beta_3T(\|\xi_0\|_{H^5(\mathbb{R}^3)}^2
+ \|\phi_0\|_{H^4(\mathbb{R}^3)}^2)\}.
\end{array}
\end{equation} 

Thus, Lemma $\ref{Sect3_Entropy_Lemma}$ is proved.
\end{proof}

Due to $\zeta=\varrho(\xi+\bar{p},\phi+\bar{S})-\bar{\varrho}$, we can estimate $\mathcal{F}[\zeta](t)$ in the following lemma:
\begin{lemma}\label{Sect3_Density_Bound_Lemma}
For any given $T\in (0,+\infty)$, if
\begin{equation*}
\sup\limits_{0\leq t\leq T} \mathcal{F}[\xi,v,\phi,\zeta](t) \leq\ve,
\end{equation*}
where $0<\ve\ll \min\{1,\ve_1,\ve_2,\ve_3\}$, then for $\forall t\in [0,T]$,
\begin{equation}\label{Sect3_Density_toProve_1}
\begin{array}{ll}
\mathcal{F}[\zeta](t)\lem C_7C_6\|\phi_0\|_{H^4(\mathbb{R}^3)}^2\exp\{\beta_2\beta_3T(\|\xi_0\|_{H^5(\mathbb{R}^3)}^2
+ \|\phi_0\|_{H^4(\mathbb{R}^3)}^2)\} \\[6pt]\hspace{1.5cm}
+ C_7\beta_1(\|\xi_0\|_{H^5(\mathbb{R}^3)}^2+\|\phi_0\|_{H^4(\mathbb{R}^3)}^2),
\end{array}
\end{equation}
for some $C_7>0$.
\end{lemma}

\begin{proof}
Since $\zeta = \frac{1}{\sqrt[\gamma]{A}}(\bar{p}+\xi)^{\frac{1}{\gamma}}
\exp\{-\frac{\bar{S}+\phi}{\gamma}\}
-\frac{1}{\sqrt[\gamma]{A}}\bar{p}^{\frac{1}{\gamma}}\exp\{-\frac{\bar{S}}{\gamma}\}$, we have
\begin{equation}\label{Sect3_Density_Prove}
\begin{array}{ll}
\mathcal{F}[\zeta](t) \lem \mathcal{F}[\xi](t)+ \mathcal{F}[\phi](t).
\end{array}
\end{equation}
\end{proof}

The following lemma concerns $L^{\infty}$ bound of $\tilde{\mathcal{F}}[\xi](t)$.
\begin{lemma}\label{Sect3_Pressure_Bound_Lemma}
For any given $T\in (0,+\infty)$, if
\begin{equation*}
\sup\limits_{0\leq t\leq T} \mathcal{F}[\xi,v,\phi,\zeta](t) \leq\ve,
\end{equation*}
where $0<\ve\ll \min\{1,\ve_1,\ve_2,\ve_3\}$, then for $\forall t\in [0,T]$,
\begin{equation}\label{Sect3_L_Infty_toProve}
\begin{array}{ll}
\tilde{\mathcal{F}}[\xi](t)\leq
C_8C_6\|\phi_0\|_{H^4(\mathbb{R}^3)}^2\exp\{\beta_2\beta_3T(\|\xi_0\|_{H^5(\mathbb{R}^3)}^2
+ \|\phi_0\|_{H^4(\mathbb{R}^3)}^2)\} \\[6pt]\hspace{1.5cm}
+ (2C_8+1)\beta_1(\|\xi_0\|_{H^5(\mathbb{R}^3)}^2+\|\phi_0\|_{H^4(\mathbb{R}^3)}^2).
\end{array}
\end{equation}
\end{lemma}

\begin{proof}
We have the $L^{\infty}$ bound of $\mathcal{F}_1[\xi](t)$, then for $\forall t\in [0,T]$,
\begin{equation*}
\begin{array}{ll}
\|\xi\|_{L^2(\mathbb{R}^3)}^2 \leq \beta_1(\|\xi_0\|_{H^5(\mathbb{R}^3)}^2+\|\phi_0\|_{H^4(\mathbb{R}^3)}^2).
\end{array}
\end{equation*}

As Remark $\ref{Sect3_Higher_Order_Xi_Remark}$ sated, $\|\nabla\xi\|_{H^4(\mathbb{R}^3)}^2
+\|\xi_t\|_{H^3(\mathbb{R}^3)}^2 \lem \mathcal{F}[v](t)+\mathcal{F}[\zeta](t)$, while the $L^{\infty}$ bounds of $\mathcal{F}[v](t)$ and $\mathcal{F}[\zeta](t)$ in $[0,T]$ have been proved, then 
\begin{equation*}
\begin{array}{ll}
\|\nabla\xi\|_{H^4(\mathbb{R}^3)}^2
+\|\xi_t\|_{H^3(\mathbb{R}^3)}^2 \\[6pt]
\leq C_8C_6\|\phi_0\|_{H^4(\mathbb{R}^3)}^2\exp\{\beta_2\beta_3T(\|\xi_0\|_{H^5(\mathbb{R}^3)}^2
+ \|\phi_0\|_{H^4(\mathbb{R}^3)}^2)\} \\[6pt]\quad
+ 2C_8\beta_1(\|\xi_0\|_{H^5(\mathbb{R}^3)}^2+\|\phi_0\|_{H^4(\mathbb{R}^3)}^2),
\end{array}
\end{equation*}
for some $C_8>0$. Thus, Lemma $\ref{Sect3_Pressure_Bound_Lemma}$ is proved.
\end{proof}

\section{Existence in $[0,T]$ of Classical Solutions to the Relaxed Equations}
In this section, we prove the existence in $[0,T]$ of classical solutions to the relaxed equations $(\ref{Sect2_Relaxed_Eq})$ under small data assumption.

After eliminating $v$ from $(\ref{Sect2_Relaxed_Eq})$, we have the following parabolic-hyperbolic equations:
\begin{equation}\label{Sect4_Parabolic_Hyperbolic}
\left\{\begin{array}{lll}
\xi_t = \frac{\gamma p}{a\varrho}\triangle\xi + \frac{p}{a\varrho}\nabla\xi\cdot\nabla\phi, \\[6pt]
\phi_t = \frac{1}{a\varrho}\nabla\xi\cdot\nabla\phi, \\[6pt]
(\xi,\phi)(x,0)=(\lim\limits_{\tau\rto 0}p_0(x,\tau)-\bar{p}, \lim\limits_{\tau\rto 0}S_0(x,\tau)-\bar{S}),
\end{array}\right.
\end{equation}
where $\varrho =\zeta +\bar{\varrho} =\varrho(\xi+\bar{p},\phi+\bar{S})$.

The proof of the local existence of classical solutions to $(\ref{Sect4_Parabolic_Hyperbolic})$ is standard (using the linearization-iteration-convergence scheme), so we give a lemma on the local existence without proof here.
\begin{lemma}\label{Sect4_LocalExistence}
$(Local\ Existence)$\\[6pt]
If $(\lim\limits_{\tau\rto 0}\xi_0(x,\tau),\lim\limits_{\tau\rto 0}\phi_0(x,\tau))\in H^5(\mathbb{R}^3)\times H^4(\mathbb{R}^3)$, $\inf\limits_{x\in\mathbb{R}^3}\lim\limits_{\tau\rto 0}p_0(x,\tau)>0$, then there exists a finite time $T_{\ast}>0$, such that Cauchy problem $(\ref{Sect4_Parabolic_Hyperbolic})$ admits a unique local classical solution $(\xi,\phi)$ satisfying
\begin{equation}\label{Sect4_Local_Regularity}
\left\{\begin{array}{ll}
(\xi,\phi)\in \underset{0\leq \ell\leq 1}{\cap}C^{\ell}([0,T_{\ast}),H^{5-\ell}(\mathbb{R}^3)\times H^{4-\ell}(\mathbb{R}^3)), \\[6pt]
\triangle\xi \in C(\mathbb{R}^3\times[0,T_{\ast})).
\end{array}\right.
\end{equation}
\end{lemma}

The above lemma implies the local existence of classical solutions to Cauchy problem $(\ref{Sect2_Relaxed_Eq})$ for $(\xi,\phi)\in \underset{0\leq \ell\leq 1}{\cap}C^{\ell}([0,T^{\ast}), C^{2-\ell}(\mathbb{R}^3)\times C^{1-\ell}(\mathbb{R}^3))$. Based on the a priori estimates for $(\xi,v,\phi,\zeta)$, the solution of $(\ref{Sect2_Relaxed_Eq})$ can be extended from $[0,T_{\ast})$ to any finite time interval $[0,T]$.
\begin{theorem}\label{Sect4_GlobalExistence_Thm}
$(Existence\ in\ [0,T])$\\[6pt]
Assume $(\lim\limits_{\tau\rto 0}\xi_0(x,\tau),\lim\limits_{\tau\rto 0}\phi_0(x,\tau))\in H^5(\mathbb{R}^3)\times H^4(\mathbb{R}^3)$, $\inf\limits_{x\in\mathbb{R}^3}\lim\limits_{\tau\rto 0}p_0(x,\tau)>0$.
There exists a sufficiently small number $\delta_1>0$, such that if $\|\lim\limits_{\tau\rto 0}\xi_0(x,\tau)\|_{H^5(\mathbb{R}^3)}^2+\|\lim\limits_{\tau\rto 0}\phi_0(x,\tau)\|_{H^4(\mathbb{R}^3)}^2\leq \delta_1$, then Cauchy problem for the relaxed equations $(\ref{Sect2_Relaxed_Eq})$ admits a unique classical solution $(\xi,v,\phi,\zeta)$ satisfying
\begin{equation}\label{Sect4_Global_Regularity}
\begin{array}{ll}
\xi\in \underset{0\leq \ell\leq 1}{\cap}C^{\ell}([0,T],H^{5-\ell}(\mathbb{R}^3)),\
\\[10pt]
(v,\phi,\zeta)\in \underset{0\leq \ell\leq 1}{\cap}C^{\ell}([0,T],H^{4-\ell}(\mathbb{R}^3)^3),\\[10pt]
\nabla\cdot v,\ \triangle\xi \in C(\mathbb{R}^3\times[0,T]).
\end{array}
\end{equation}
\end{theorem}

\begin{proof}
In view of Lemmas $\ref{Sect3_Decay_Lemma},\ref{Sect3_Entropy_Lemma},\ref{Sect3_Density_Bound_Lemma},\ref{Sect3_Pressure_Bound_Lemma}$, we have the following a priori estimates: for any given $T\in (0,+\infty)$, if
\begin{equation}\label{Sect4_Assumption_1}
\sup\limits_{0\leq t\leq T} \mathcal{F}[\xi,v,\phi,\zeta](t) \leq\ve,
\end{equation}
where $0<\ve\ll \min\{1,\ve_1,\ve_2,\ve_3\}$, then
\begin{equation}\label{Sect4_Decay}
\begin{array}{ll}
\mathcal{F}[\xi](t)+\mathcal{F}[v](t)\leq \beta_1(\|\xi_0\|_{H^5(\mathbb{R}^3)}+\|\phi_0\|_{H^4(\mathbb{R}^3)}), \\[10pt]

\mathcal{E}[\phi](t) \leq C_6\|\phi_0\|_{H^4(\mathbb{R}^3)}^2\exp\{\beta_2\beta_3T(\|\xi_0\|_{H^5(\mathbb{R}^3)}^2
+ \|\phi_0\|_{H^3(\mathbb{R}^4)}^2)\}, \\[10pt]

\mathcal{E}[\zeta](t) \leq C_7C_6\|\phi_0\|_{H^4(\mathbb{R}^3)}^2\exp\{\beta_2\beta_3T(\|\xi_0\|_{H^5(\mathbb{R}^3)}^2
+ \|\phi_0\|_{H^4(\mathbb{R}^3)}^2)\} \\[7pt]\hspace{1.5cm}
+ C_7\beta_1(\|\xi_0\|_{H^5(\mathbb{R}^3)}^2+\|\phi_0\|_{H^4(\mathbb{R}^3)}^2), \\[10pt]

\tilde{\mathcal{F}}[\xi](t)\leq
C_8C_6\|\phi_0\|_{H^4(\mathbb{R}^3)}^2\exp\{\beta_2\beta_3T(\|\xi_0\|_{H^5(\mathbb{R}^3)}^2
+ \|\phi_0\|_{H^4(\mathbb{R}^3)}^2)\} \\[6pt]\hspace{1.5cm}
+ (2C_8+1)\beta_1(\|\xi_0\|_{H^5(\mathbb{R}^3)}^2+\|\phi_0\|_{H^4(\mathbb{R}^3)}^2).
\end{array}
\end{equation}

Since $\ve$ is independent of $\xi_0(x,\tau)$ and $\phi_0(x,\tau)$, there exists $\delta_1>0$ such that if $\|\lim\limits_{\tau\rto 0}\xi_0(x,\tau)\|_{H^5(\mathbb{R}^3)}^2
+\|\lim\limits_{\tau\rto 0}\phi_0(x,\tau)\|_{H^4(\mathbb{R}^3)}^2\leq \delta_1$, then
\begin{equation}\label{Sect4_Data_Condition}
\left\{\begin{array}{ll}
\|\lim\limits_{\tau\rto 0}\xi_0(x,\tau)\|_{H^5(\mathbb{R}^3)}^2
+\|\lim\limits_{\tau\rto 0}\phi_0(x,\tau)\|_{H^4(\mathbb{R}^3)}^2
\leq \min\{\frac{\ve}{2\beta_1},\frac{\ve}{4C_7\beta_1}\}, \\[6pt]

\|\lim\limits_{\tau\rto 0}\phi_0(x,\tau)\|_{H^4(\mathbb{R}^3)}^2 \leq \min\{
\frac{\ve}{4C_6}\exp\{-\frac{\beta_2\beta_3T\ve}{2\beta_1}\},
\frac{(1-2C_7)\ve}{4C_7}\exp\{-\frac{\beta_2\beta_3T\ve}{2\beta_1}\}
\}.
\end{array}\right.
\end{equation}
Now, we can check a priori assumption $\mathcal{F}[\xi,v,\phi,\zeta](t)\leq \ve$ is satisfied, then the validity of the former a priori estimates is verified.

By Lemma $\ref{Sect5_Decay_Lemma}$, we have
\begin{equation}\label{Sect4_Aubin_Lions_1}
\begin{array}{ll}
\int\limits_{0}^{T} \mathcal{E}[\nabla\cdot v](s)\,\mathrm{d}s \leq \beta_2(\|\xi_0\|_{H^5(\mathbb{R}^3)}+\|\phi_0\|_{H^4(\mathbb{R}^3)}),
\end{array}
\end{equation}
which implies that for any given time $T\in (0,+\infty)$,
\begin{equation}\label{Sect4_Aubin_Lions_2}
\left\{\begin{array}{ll}
\|\nabla\cdot v\|_{L^2([0,T],H^4(\mathbb{R}^3))}^2 \leq \beta_2(\|\xi_0\|_{H^5(\mathbb{R}^3)}+\|\phi_0\|_{H^4(\mathbb{R}^3)}), \\[6pt]
\|\nabla\cdot v_t\|_{L^2([0,T],H^2(\mathbb{R}^3))}^2 \leq \beta_2(\|\xi_0\|_{H^5(\mathbb{R}^3)}+\|\phi_0\|_{H^4(\mathbb{R}^3)}).
\end{array}\right.
\end{equation}

By Aubin-Lions' Lemma, we obtain
\begin{equation}\label{Sect4_Aubin_Lions_3}
\begin{array}{ll}
\|\nabla\cdot v\|_{C([0,T],H^3(\mathbb{R}^3))}^2 \leq \beta_2(\|\xi_0\|_{H^5(\mathbb{R}^3)}+\|\phi_0\|_{H^4(\mathbb{R}^3)}),
\end{array}
\end{equation}
which implies that $\nabla\cdot v\in C(\mathbb{R}^3\times [0,T])$ for any $T\in (0,+\infty)$, then
\begin{equation}\label{Sect4_Aubin_Lions_4}
\begin{array}{ll}
\triangle\xi = -a k_1\varrho\nabla\cdot v - a k_1 v\cdot\nabla\varrho \in C(\mathbb{R}^3\times [0,T]).
\end{array}
\end{equation}

Due to the a priori estimates for $(\xi,v,\phi,\zeta)$ and Lemma $\ref{Sect4_LocalExistence}$ on the local existence result, the solution of $(\ref{Sect2_Relaxed_Eq})$ can be extended from $[0,T_{\ast})$ to any finite time interval $[0,T]$. Thus, Theorem $\ref{Sect4_GlobalExistence_Thm}$ on the existence in $[0,T]$ of classical solutions to Cauchy problem $(\ref{Sect2_Relaxed_Eq})$ is proved.
\end{proof}

\section{Uniform A Priori Estimates for the Relaxing Equations}
In this section, we derive unform a priori estimates for the relaxing Cauchy problem $(\ref{Sect2_Relaxing_Eq})$.
In order to discuss the initial layer and strong convergence of the velocity, we need to estimate the higher order time derivatives. Note that the notation of small quantity $\e$ used in the following is different from $\ve$ used before.

The following lemma states that $\mathcal{E}[\xi](t)$ and $\mathcal{E}_1[\xi](t)$ are equivalent, $\mathcal{E}[v](t)$ and $\mathcal{E}_1[v](t)$ are equivalent.
\begin{lemma}\label{Sect5_Epsilon0_Lemma}
For any given $T\in (0,+\infty),\tau\in(0,1]$, there exists $\e_1>0$ which is independent of $(\xi_0,\tau v_0,\phi_0)$, such that if $\sup\limits_{0\leq t\leq T} \mathcal{E}[\xi,\tau v,\phi,\zeta](t) \leq\e_1$, then
$|\xi|_{\infty}\leq \frac{\bar{p}}{3},|\zeta|_{\infty}\leq \frac{\bar{\varrho}}{2}$ and there exist $c_4>0,c_5>0$ such that
\begin{equation}\label{Sect5_Energy_Equivalence}
\begin{array}{ll}
c_4 \mathcal{E}[\xi](t) \leq \mathcal{E}_1[\xi](t) \leq c_5 \mathcal{E}[\xi](t),\quad
c_4 \mathcal{E}[v](t) \leq \mathcal{E}_1[v](t) \leq c_5 \mathcal{E}[v](t).
\end{array}
\end{equation}
\end{lemma}

The following lemma gives uniform a priori estimate for $\mathcal{E}_1[\xi](t)+\tau^2 \mathcal{E}_1[v](t)$, which is equivalent to $\mathcal{E}[\xi](t)+\tau^2 \mathcal{E}[v](t)$.
\begin{lemma}\label{Sect5_Energy_Estimate_Lemma1}
For any given $T\in (0,+\infty),\tau\in(0,1]$, if
\begin{equation*}
\sup\limits_{0\leq t\leq T} \mathcal{E}[\xi,\tau v,\phi,\zeta](t) \leq\e,
\end{equation*}
where $0<\e\ll 1$, then for $\forall t\in [0,T]$,
\begin{equation}\label{Sect5_Estimate1_toProve}
\frac{\mathrm{d}}{\mathrm{d}t}\mathcal{E}_1[\xi](t) + \tau^2 \frac{\mathrm{d}}{\mathrm{d}t}\mathcal{E}_1[v](t)
+ 2 \mathcal{E}_1[v](t) \leq C\sqrt{\e}(\mathcal{E}_X[\xi](t) + \mathcal{E}[v](t)).
\end{equation}
\end{lemma}

\begin{proof}
Let $(\ref{Sect2_Relaxing_Eq})\cdot(\xi, v)$, we get
\begin{equation}\label{Sect5_ZeroOrder_1}
\begin{array}{ll}
(|\xi|^2+\tau^2|v|^2)_t + 2k_2 \xi\nabla\cdot v + 2k_2 v\cdot\nabla\xi + 2 |v|^2 \\[6pt]
= -2\gamma k_1 \xi^2\nabla\cdot v - 2k_1 \xi v\cdot\nabla \xi - 2 k_1 \tau^2 v\cdot\nabla v\cdot v
+ \frac{2}{k_1} (\frac{1}{\bar{\varrho}}- \frac{1}{\varrho})\nabla\xi\cdot v.
\end{array}
\end{equation}

Integrate $(\ref{Sect5_ZeroOrder_1})$ in $\mathbb{R}^3$ and note that
$\int\limits_{\mathbb{R}^3}\nabla\cdot(\xi v) \,\mathrm{d}x = 0$, we get
\begin{equation}\label{Sect5_ZeroOrder_3}
\begin{array}{ll}
\frac{\mathrm{d}}{\mathrm{d} t}
\int\limits_{\mathbb{R}^3}|\xi|^2+\tau^2|v|^2 \,\mathrm{d}x + 2\int\limits_{\mathbb{R}^3}|v|^2 \,\mathrm{d}x \\[9pt]
= \int\limits_{\mathbb{R}^3} 2\gamma k_1 v\cdot\nabla(\xi^2)
- 2k_1 \xi v\cdot\nabla \xi - 2 k_1 \tau^2 v\cdot\nabla v\cdot v
+ \frac{2}{k_1} (\frac{1}{\bar{\varrho}}- \frac{1}{\varrho})\nabla\xi\cdot v \,\mathrm{d}x \\[9pt]
\lem \sqrt{\e}\|\nabla\xi\|_{L^2(\mathbb{R}^3)}\|v\|_{L^2(\mathbb{R}^3)}
+\sqrt{\e}\|v\|_{L^2(\mathbb{R}^3)}^2 \\[9pt]
\lem \sqrt{\e}(\mathcal{E}_X[\xi](t) + \mathcal{E}[v](t)).
\end{array}
\end{equation}

Apply $\partial_t^{\ell}\mathcal{D}^{\alpha}$ to $(\ref{Sect2_Relaxing_Eq})$, where $0\leq \ell\leq 2, 1\leq\ell+|\alpha|\leq 4$, then eliminate $\tau$ from both sides of the equations, we get
\begin{equation}\label{Sect5_Estimate2_1}
\left\{\begin{array}{ll}
(\partial_t^{\ell}\mathcal{D}^{\alpha}\xi)_t + k_2\nabla\cdot(\partial_t^{\ell}\mathcal{D}^{\alpha} v)
= -\gamma k_1\partial_t^{\ell}\mathcal{D}^{\alpha}(\xi\nabla\cdot v)
- k_1 \partial_t^{\ell}\mathcal{D}^{\alpha}(v\cdot\nabla \xi), \\[8pt]
\tau^2(\partial_t^{\ell}\mathcal{D}^{\alpha} v)_t + k_2 \nabla(\partial_t^{\ell}\mathcal{D}^{\alpha}\xi)
+ \partial_t^{\ell}\mathcal{D}^{\alpha} v \\[6pt]\hspace{4.05cm}
= - k_1 \tau^2\partial_t^{\ell}\mathcal{D}^{\alpha}(v\cdot\nabla v)
+ \frac{1}{k_1} \partial_t^{\ell}\mathcal{D}^{\alpha}[(\frac{1}{\bar{\varrho}}- \frac{1}{\varrho})\nabla\xi].
\end{array}\right.
\end{equation}

Let $(\ref{Sect5_Estimate2_1})\cdot(\partial_t^{\ell}\mathcal{D}^{\alpha}\xi, \partial_t^{\ell}\mathcal{D}^{\alpha} v)$, we get
\begin{equation}\label{Sect5_Estimate2_2}
\begin{array}{ll}
(|\partial_t^{\ell}\mathcal{D}^{\alpha}\xi|^2+\tau^2|\partial_t^{\ell}\mathcal{D}^{\alpha} v|^2)_t
+ 2k_2 \partial_t^{\ell}\mathcal{D}^{\alpha}\xi\nabla\cdot(\partial_t^{\ell}\mathcal{D}^{\alpha} v) \\[6pt]\quad
+ 2k_2 \partial_t^{\ell}\mathcal{D}^{\alpha} v\cdot\nabla(\partial_t^{\ell}\mathcal{D}^{\alpha}\xi)
+ 2 |\partial_t^{\ell}\mathcal{D}^{\alpha} v|^2 \\[6pt]
= -2\gamma k_1 (\partial_t^{\ell}\mathcal{D}^{\alpha}\xi)\partial_t^{\ell}\mathcal{D}^{\alpha}(\xi\nabla\cdot v)
- 2k_1 (\partial_t^{\ell}\mathcal{D}^{\alpha}\xi)\partial_t^{\ell}\mathcal{D}^{\alpha}(v\cdot\nabla \xi) \\[6pt]\quad
- 2 k_1 \tau^2 (\partial_t^{\ell}\mathcal{D}^{\alpha} v)\cdot\partial_t^{\ell}\mathcal{D}^{\alpha}(v\cdot\nabla v)
+ \frac{2}{k_1} (\partial_t^{\ell}\mathcal{D}^{\alpha} v)\cdot
\partial_t^{\ell}\mathcal{D}^{\alpha}[(\frac{1}{\bar{\varrho}}- \frac{1}{\varrho})\nabla\xi].
\end{array}
\end{equation}

Integrate $(\ref{Sect5_Estimate2_2})$ in $\mathbb{R}^3$ and note that
$\int\limits_{\mathbb{R}^3}\nabla\cdot(\partial_t^{\ell}\mathcal{D}^{\alpha}\xi\partial_t^{\ell}\mathcal{D}^{\alpha} v) \,\mathrm{d}x = 0$, we get
\begin{equation}\label{Sect5_Estimate2_5}
\begin{array}{ll}
\frac{\mathrm{d}}{\mathrm{d} t}
\int\limits_{\mathbb{R}^3}|\partial_t^{\ell}\mathcal{D}^{\alpha}\xi|^2+\tau^2|\partial_t^{\ell}\mathcal{D}^{\alpha} v|^2 \,\mathrm{d}x
+ 2\int\limits_{\mathbb{R}^3}|\partial_t^{\ell}\mathcal{D}^{\alpha} v|^2 \,\mathrm{d}x \\[10pt]
= \int\limits_{\mathbb{R}^3} -2\gamma k_1 (\partial_t^{\ell}\mathcal{D}^{\alpha}\xi)\partial_t^{\ell}\mathcal{D}^{\alpha}(\xi\nabla\cdot v)
- 2k_1 (\partial_t^{\ell}\mathcal{D}^{\alpha}\xi)\partial_t^{\ell}\mathcal{D}^{\alpha}(v\cdot\nabla \xi) \\[10pt]\quad
- 2 k_1 \tau^2 (\partial_t^{\ell}\mathcal{D}^{\alpha} v)\cdot\partial_t^{\ell}\mathcal{D}^{\alpha}(v\cdot\nabla v)
+ \frac{2}{k_1} (\partial_t^{\ell}\mathcal{D}^{\alpha} v)\cdot
\partial_t^{\ell}\mathcal{D}^{\alpha}[(\frac{1}{\bar{\varrho}}- \frac{1}{\varrho})\nabla\xi] \,\mathrm{d}x
:= I_7.
\end{array}
\end{equation}

\vspace{0.3cm}
When $1\leq\ell+|\alpha|\leq 3$, it is easy to check that
$I_7\lem \sqrt{\e}(\mathcal{E}_X[\xi](t)+\mathcal{E}[v](t))$. \\[6pt]
\indent
When $\ell+|\alpha|=4$, we estimate the quantity
$I_7 - \frac{\mathrm{d}}{\mathrm{d}t}\int\limits_{\mathbb{R}^3}\frac{\xi}{p}(\partial_t^{\ell}\mathcal{D}^{\alpha}\xi)^2 \,\mathrm{d}x
+ \tau^2\frac{\mathrm{d}}{\mathrm{d}t}\int\limits_{\mathbb{R}^3}(\frac{\varrho}{\bar{\varrho}}-1)
|\partial_t^{\ell}\mathcal{D}^{\alpha} v|^2 \,\mathrm{d}x$, then
\begin{equation}\label{Sect5_Estimate2_6}
\begin{array}{ll}
I_7 - \frac{\mathrm{d}}{\mathrm{d}t}\int\limits_{\mathbb{R}^3}\frac{\xi}{p}(\partial_t^{\ell}\mathcal{D}^{\alpha}\xi)^2 \,\mathrm{d}x
+ \tau^2\frac{\mathrm{d}}{\mathrm{d}t}\int\limits_{\mathbb{R}^3}(\frac{\varrho}{\bar{\varrho}}-1)
|\partial_t^{\ell}\mathcal{D}^{\alpha} v|^2 \,\mathrm{d}x \\[8pt]

= -2\gamma k_1 \int\limits_{\mathbb{R}^3}(\partial_t^{\ell}\mathcal{D}^{\alpha}\xi)\xi\nabla\cdot (\partial_t^{\ell}\mathcal{D}^{\alpha} v)\,\mathrm{d}x
- 2k_1 \int\limits_{\mathbb{R}^3}(\partial_t^{\ell}\mathcal{D}^{\alpha}\xi) v\cdot\nabla (\partial_t^{\ell}\mathcal{D}^{\alpha}\xi)\,\mathrm{d}x \\[8pt]\quad
- 2 k_1 \tau^2\int\limits_{\mathbb{R}^3}v\cdot\nabla (\partial_t^{\ell}\mathcal{D}^{\alpha} v)\cdot (\partial_t^{\ell}\mathcal{D}^{\alpha} v)\,\mathrm{d}x
+ \frac{2}{k_1} \int\limits_{\mathbb{R}^3}(\frac{1}{\bar{\varrho}}
- \frac{1}{\varrho})(\partial_t^{\ell}\mathcal{D}^{\alpha} v)\cdot\nabla(\partial_t^{\ell}\mathcal{D}^{\alpha}\xi)\,\mathrm{d}x \\[8pt]\quad
+ 2\tau^2\int\limits_{\mathbb{R}^3}(\frac{\varrho}{\bar{\varrho}}-1)(\partial_t^{\ell}\mathcal{D}^{\alpha} v)\cdot (\partial_t^{\ell}\mathcal{D}^{\alpha} v_t) \,\mathrm{d}x
+ \tau^2\int\limits_{\mathbb{R}^3}\frac{\zeta_t}{\bar{\varrho}}|\partial_t^{\ell}\mathcal{D}^{\alpha} v|^2 \,\mathrm{d}x
\\[8pt]\quad
- 2\int\limits_{\mathbb{R}^3}\frac{\xi}{p}(\partial_t^{\ell}\mathcal{D}^{\alpha}\xi)
(\partial_t^{\ell}\mathcal{D}^{\alpha}\xi_t) \,\mathrm{d}x
- \int\limits_{\mathbb{R}^3}\partial_t(\frac{\xi}{p})(\partial_t^{\ell}\mathcal{D}^{\alpha}\xi)^2 \,\mathrm{d}x
 \\[8pt]

\lem \sqrt{\e}(\mathcal{E}_X[\xi](t)+\mathcal{E}[v](t))
+ k_1 \int\limits_{\mathbb{R}^3}(|\partial_t^{\ell}\mathcal{D}^{\alpha}\xi|^2
+\tau^2|\partial_t^{\ell}\mathcal{D}^{\alpha} v|^2) \nabla\cdot v\,\mathrm{d}x \\[8pt]\quad
-2\gamma k_1 \int\limits_{\mathbb{R}^3}\xi(\partial_t^{\ell}\mathcal{D}^{\alpha}\xi)\nabla\cdot (\partial_t^{\ell}\mathcal{D}^{\alpha} v)\,\mathrm{d}x

+ \frac{2}{k_1} \int\limits_{\mathbb{R}^3}(\frac{1}{\bar{\varrho}}- \frac{1}{\varrho})(\partial_t^{\ell}\mathcal{D}^{\alpha} v)\cdot\nabla(\partial_t^{\ell}\mathcal{D}^{\alpha}\xi)\,\mathrm{d}x
\\[8pt]\quad
- 2\int\limits_{\mathbb{R}^3}\frac{\xi}{p}
(\partial_t^{\ell}\mathcal{D}^{\alpha}\xi)(\partial_t^{\ell}\mathcal{D}^{\alpha}\xi_t) \,\mathrm{d}x
+ 2\tau^2\int\limits_{\mathbb{R}^3}(\frac{\varrho}{\bar{\varrho}}-1)(\partial_t^{\ell}\mathcal{D}^{\alpha} v)\cdot (\partial_t^{\ell}\mathcal{D}^{\alpha} v_t) \,\mathrm{d}x \\[10pt]

\lem -2\int\limits_{\mathbb{R}^3}\frac{\xi}{p}(\partial_t^{\ell}\mathcal{D}^{\alpha}\xi)
[\partial_t^{\ell}\mathcal{D}^{\alpha}\xi_t +k_1\gamma p\nabla\cdot (\partial_t^{\ell}\mathcal{D}^{\alpha} v)]\,\mathrm{d}x
\\[8pt]\quad
+ \frac{2}{k_1} \int\limits_{\mathbb{R}^3}(\frac{1}{\bar{\varrho}}- \frac{1}{\varrho})(\partial_t^{\ell}\mathcal{D}^{\alpha} v)\cdot[\nabla(\partial_t^{\ell}\mathcal{D}^{\alpha}\xi)
+k_1\tau^2\varrho (\partial_t^{\ell}\mathcal{D}^{\alpha} v_t)]\,\mathrm{d}x \\[8pt]\quad
+ \sqrt{\e}(\mathcal{E}_X[\xi](t)+\mathcal{E}[v](t)).
\end{array}
\end{equation}

Apply $\partial_t^{\ell}\mathcal{D}^{\alpha}$ to $(\ref{Sect2_Relaxing_Eq})_1$, where $0\leq \ell\leq 2,
\ell+|\alpha|=4$, we get
\begin{equation}\label{Sect5_Estimate2_7}
\begin{array}{ll}
\partial_t^{\ell}\mathcal{D}^{\alpha}\xi_t + k_1\gamma p\nabla\cdot (\partial_t^{\ell}\mathcal{D}^{\alpha} v) = -k_1\partial_t^{\ell}\mathcal{D}^{\alpha}(v\cdot\nabla\xi)  \\[8pt]\hspace{4.5cm}
- k_1\gamma\sum\limits_{\ell_1+|\alpha_1|>0}\partial_t^{\ell_1}\mathcal{D}^{\alpha_1}\xi \nabla\cdot (\partial_t^{\ell_2}\mathcal{D}^{\alpha_2} v).
\end{array}
\end{equation}

Plug $(\ref{Sect5_Estimate2_7})$ into the following integral, we get
\begin{equation}\label{Sect5_Estimate2_8}
\begin{array}{ll}
\int\limits_{\mathbb{R}^3}\frac{\xi}{p}(\partial_t^{\ell}\mathcal{D}^{\alpha}\xi)
(\partial_t^{\ell}\mathcal{D}^{\alpha}\xi_t+k_1\gamma p\nabla\cdot (\partial_t^{\ell}\mathcal{D}^{\alpha} v))\,\mathrm{d}x
= \int\limits_{\mathbb{R}^3}\frac{\xi}{p}(\partial_t^{\ell}\mathcal{D}^{\alpha}\xi)[R.H.S.\ of\ (\ref{Sect5_Estimate2_7})]\,\mathrm{d}x \\[9pt]

\lem \sqrt{\e}(\mathcal{E}_X[\xi](t) + \mathcal{E}[v](t))
+ \frac{k_1}{2}\int\limits_{\mathbb{R}^3}|\partial_t^{\ell}\mathcal{D}^{\alpha}\xi|^2 \nabla\cdot(\frac{\xi}{p}v) \,\mathrm{d}x
\\[6pt]
\lem \sqrt{\e}(\mathcal{E}_X[\xi](t) + \mathcal{E}[v](t)).
\end{array}
\end{equation}

Apply $\partial_t^{\ell}\mathcal{D}^{\alpha}$ to
$k_1\tau^2 v_t + k_1^2 \tau^2 v\cdot\nabla v + k_1 v + \frac{1}{\varrho}\nabla\xi =0$, where $0\leq \ell\leq 2,
\ell+|\alpha|=4$, we get
\begin{equation}\label{Sect5_Estimate2_9}
\begin{array}{ll}
\frac{1}{\varrho}\nabla(\partial_t^{\ell}\mathcal{D}^{\alpha}\xi) + k_1\tau^2 (\partial_t^{\ell}\mathcal{D}^{\alpha} v_t) \\[6pt]
= -k_1 \partial_t^{\ell}\mathcal{D}^{\alpha} v -k_1^2 \tau^2 \partial_t^{\ell}\mathcal{D}^{\alpha}(v\cdot\nabla v)
-\sum\limits_{\ell_1+|\alpha_1|>0}\partial_t^{\ell_1}\mathcal{D}^{\alpha_1}(\frac{1}{\varrho}) \partial_t^{\ell_2}\mathcal{D}^{\alpha_2}\nabla\xi.
\end{array}
\end{equation}

Plug $(\ref{Sect5_Estimate2_9})$ into the following integral, we get
\begin{equation}\label{Sect5_Estimate2_10}
\begin{array}{ll}
\int\limits_{\mathbb{R}^3}(\frac{1}{\bar{\varrho}}- \frac{1}{\varrho})(\partial_t^{\ell}\mathcal{D}^{\alpha} v)\cdot[\nabla(\partial_t^{\ell}\mathcal{D}^{\alpha}\xi)
+k_1\tau^2\varrho \partial_t^{\ell}\mathcal{D}^{\alpha} v_t]\,\mathrm{d}x \\[6pt]
= \int\limits_{\mathbb{R}^3}(\frac{1}{\bar{\varrho}}- \frac{1}{\varrho})(\partial_t^{\ell}\mathcal{D}^{\alpha} v)\cdot \varrho[R.H.S.\ of\ (\ref{Sect5_Estimate2_9})]\,\mathrm{d}x \\[6pt]

\lem \frac{k_1^2 \tau}{2}\int\limits_{\mathbb{R}^3}|\partial_t^{\ell}\mathcal{D}^{\alpha} v|^2 \nabla\cdot[(\frac{\varrho}{\bar{\varrho}}-1) \tau v] \,\mathrm{d}x
+ \sqrt{\e}(\mathcal{E}_X[\xi](t) + \mathcal{E}[v](t)) \\[6pt]

\lem \sqrt{\e}(\mathcal{E}_X[\xi](t) + \mathcal{E}[v](t)).
\end{array}
\end{equation}

Plug $(\ref{Sect5_Estimate2_8})$ and $(\ref{Sect5_Estimate2_10})$ into $(\ref{Sect5_Estimate2_6})$, we get
\begin{equation}\label{Sect5_Estimate2_11}
\begin{array}{ll}
I_7 - \frac{\mathrm{d}}{\mathrm{d}t}\int\limits_{\mathbb{R}^3}\frac{\xi}{p}(\partial_t^{\ell}\mathcal{D}^{\alpha}\xi)^2 \,\mathrm{d}x
+ \tau^2\frac{\mathrm{d}}{\mathrm{d}t}\int\limits_{\mathbb{R}^3}(\frac{\varrho}{\bar{\varrho}}-1)|\partial_t^{\ell}\mathcal{D}^{\alpha} v|^2 \,\mathrm{d}x
\lem \sqrt{\e}(\mathcal{E}_X[\xi](t) + \mathcal{E}[v](t)).
\end{array}
\end{equation}

After Summing $\ell$ and $\alpha$, we have
\begin{equation}\label{Sect5_Estimate2_12}
\begin{array}{ll}
\frac{\mathrm{d}}{\mathrm{d} t}\left(\sum\limits_{0\leq\ell\leq 2,0\leq \ell+|\alpha|\leq 4}\
\int\limits_{\mathbb{R}^3}|\partial_t^{\ell}\mathcal{D}^{\alpha}\xi|^2+\tau^2|\partial_t^{\ell}\mathcal{D}^{\alpha} v|^2 \,\mathrm{d}x \right. \\[10pt]\quad
\left. -\sum\limits_{0\leq\ell\leq 2,\ell+|\alpha|=4}\
\int\limits_{\mathbb{R}^3}\frac{\xi}{p} (\partial_t^{\ell}\mathcal{D}^{\alpha}\xi)^2\,\mathrm{d}x
+\tau^2 \sum\limits_{0\leq\ell\leq 2,\ell+|\alpha|=4}\
\int\limits_{\mathbb{R}^3}(\frac{\varrho}{\bar{\varrho}}- 1)|\partial_t^{\ell}\mathcal{D}^{\alpha} v|^2 \,\mathrm{d}x \right) \\[15pt]\quad
+ 2 \sum\limits_{0\leq\ell\leq 2,0\leq \ell+|\alpha|\leq 4}\
\int\limits_{\mathbb{R}^3}|\partial_t^{\ell}\mathcal{D}^{\alpha} v|^2 \,\mathrm{d}x \\[14pt]
\lem \sqrt{\e}(\mathcal{E}_X[\xi](t) + \mathcal{E}[v](t)).
\end{array}
\end{equation}

Then
\begin{equation}\label{Sect5_Estimate2_13}
\begin{array}{ll}
\frac{\mathrm{d}}{\mathrm{d} t}\left(\sum\limits_{0\leq\ell\leq 2,0\leq \ell+|\alpha|\leq 4}\
\int\limits_{\mathbb{R}^3}|\partial_t^{\ell}\mathcal{D}^{\alpha}\xi|^2 \,\mathrm{d}x
-\sum\limits_{0\leq\ell\leq 2,\ell+|\alpha|=4}\
\int\limits_{\mathbb{R}^3}\frac{\xi}{p} (\partial_t^{\ell}\mathcal{D}^{\alpha}\xi)^2\,\mathrm{d}x\right) \\[13pt]

+ \tau^2\frac{\mathrm{d}}{\mathrm{d} t}\left(
\sum\limits_{0\leq\ell\leq 2,0\leq\ell+|\alpha|\leq 4}\
\int\limits_{\mathbb{R}^3}|\partial_t^{\ell}\mathcal{D}^{\alpha} v|^2 \,\mathrm{d}x
+\sum\limits_{0\leq\ell\leq 2,\ell+|\alpha|=4}\
\int\limits_{\mathbb{R}^3}(\frac{\varrho}{\bar{\varrho}}- 1)|\partial_t \mathcal{D}^{\alpha}v|^2 \,\mathrm{d}x \right) \\[13pt]

+ 2 \left(\sum\limits_{0\leq\ell\leq 2,0\leq \ell+|\alpha|\leq 4}\
\int\limits_{\mathbb{R}^3}|\partial_t^{\ell}\mathcal{D}^{\alpha} v|^2 \,\mathrm{d}x
+\sum\limits_{0\leq\ell\leq 2,\ell+|\alpha|=4}\
\int\limits_{\mathbb{R}^3}(\frac{\varrho}{\bar{\varrho}}- 1)|\partial_t^{\ell}\mathcal{D}^{\alpha}v|^2 \,\mathrm{d}x \right)
\\[15pt]

\lem \sqrt{\e}(\mathcal{E}_X[\xi](t) + \mathcal{E}[v](t)).
\end{array}
\end{equation}

By $(\ref{Sect5_ZeroOrder_3})+(\ref{Sect5_Estimate2_13})$, we obtain
\begin{equation}\label{Sect5_Estimate2_14}
\frac{\mathrm{d}}{\mathrm{d}t}\mathcal{E}_1[\xi](t) + \tau^2 \frac{\mathrm{d}}{\mathrm{d}t}\mathcal{E}_1[v](t)
+ 2 \mathcal{E}_1[v](t) \leq C\sqrt{\e}(\mathcal{E}_X[\xi](t) + \mathcal{E}[v](t)).
\end{equation}
Thus, Lemma $\ref{Sect5_Energy_Estimate_Lemma1}$ is proved.
\end{proof}

The structure of the equations $(\ref{Sect2_Relaxing_Eq})$ implies
$\mathcal{E}_X[\xi](t)$ can be estimates by $\mathcal{E}[v](t)$, as the following lemma stated:
\begin{lemma}\label{Sect5_Energy_Ineq}
For any given $T\in (0,+\infty),\tau\in(0,1]$, if
\begin{equation*}
\sup\limits_{0\leq t\leq T} \mathcal{E}[\xi,\tau v,\phi,\zeta](t) \leq\e,
\end{equation*}
there exists $\e_2>0$, which is independent of $(\xi_0,v_0,\phi_0)$, such that if $0<\e\ll \min\{1,\e_2\}$, then for $\forall t\in [0,T]$,
\begin{equation}\label{Sect5_Energy_Ineq_toProve}
\mathcal{E}_X[\xi](t) \leq c_6\mathcal{E}[v](t),
\end{equation}
\end{lemma}

\begin{proof}
Apply $\mathcal{D}^{\alpha}$ to $\nabla\xi = - k_1\tau^2\varrho v_t - k_1^2\tau^2 v\cdot\nabla v - k_1\varrho v$,
where $0\leq |\alpha|\leq 3$, we get
\begin{equation}\label{Sect5_Energy_Ineq_1}
\begin{array}{ll}
\mathcal{D}^{\alpha}\nabla\xi = - k_1\tau^2\mathcal{D}^{\alpha}(\varrho v_t)
- k_1^2\tau^2 \mathcal{D}^{\alpha}(\varrho v\cdot\nabla v) - k_1 \mathcal{D}^{\alpha}(\varrho v), \\[6pt]

\|\mathcal{D}^{\alpha}\nabla\xi\|_{L^2(\mathbb{R}^3)}^2 \lem
\tau^4\|\mathcal{D}^{\alpha}(\varrho v_t)\|_{L^2(\mathbb{R}^3)}^2
+ \tau^4\|\mathcal{D}^{\alpha}(\varrho v\cdot\nabla v)\|_{L^2(\mathbb{R}^3)}^2
+ \|\mathcal{D}^{\alpha}(\varrho v)\|_{L^2(\mathbb{R}^3)}^2 \\[6pt]\hspace{2.3cm}

\lem \|\mathcal{D}^{\alpha}v_t\|_{L^2(\mathbb{R}^3)}^2 + \|\mathcal{D}^{\alpha}v\|_{L^2(\mathbb{R}^3)}^2
+ \mathcal{E}[\zeta]\mathcal{E}[v](t) + \mathcal{E}[\tau v](t)\mathcal{E}[v](t) \\[6pt]\hspace{2.7cm}
+ \mathcal{E}[\tau v](t)\mathcal{E}[\zeta]\mathcal{E}[v](t).
\end{array}
\end{equation}

Apply $\mathcal{D}^{\alpha}$ to $\xi_t = - k_1v\cdot\nabla\xi - k_1\gamma p\nabla\cdot v$,
where $0\leq|\alpha|\leq 3$, we get
\begin{equation}\label{Sect5_Energy_Ineq_2}
\begin{array}{ll}
\mathcal{D}^{\alpha}\xi_t = - k_1\mathcal{D}^{\alpha}(v\cdot\nabla\xi)
- k_1\gamma \mathcal{D}^{\alpha}(p\nabla\cdot v), \\[6pt]

\|\mathcal{D}^{\alpha}\xi_t\|_{L^2(\mathbb{R}^3)}^2 \lem
\|\mathcal{D}^{\alpha}(v\cdot\nabla\xi)\|_{L^2(\mathbb{R}^3)}^2
+ \|\mathcal{D}^{\alpha}(p \nabla\cdot v)\|_{L^2(\mathbb{R}^3)}^2 \\[6pt]\hspace{2.05cm}

\lem \|\mathcal{D}^{\alpha}\nabla\cdot v\|_{L^2(\mathbb{R}^3)}^2
+ \mathcal{E}_X[\xi]\mathcal{E}[v](t).
\end{array}
\end{equation}

Apply $\mathcal{D}^{\alpha}$ to $\xi_{tt} = - k_1v_t\cdot\nabla\xi - k_1v\cdot\nabla\xi_t 
- k_1\gamma p\nabla\cdot v_t - k_1\gamma \xi_t\nabla\cdot v$,
where $0\leq|\alpha|\leq 2$, we get
\begin{equation}\label{Sect5_Energy_Ineq_3}
\begin{array}{ll}
\mathcal{D}^{\alpha}\xi_{tt} = - k_1\mathcal{D}^{\alpha}(v_t\cdot\nabla\xi)- k_1\mathcal{D}^{\alpha}(v\cdot\nabla\xi_t)
- k_1\gamma \mathcal{D}^{\alpha}(p\nabla\cdot v_t)- k_1\gamma \mathcal{D}^{\alpha}(\xi_t\nabla\cdot v), \\[6pt]

\|\mathcal{D}^{\alpha}\xi_{tt}\|_{L^2(\mathbb{R}^3)}^2 
\lem \|\mathcal{D}^{\alpha}\nabla\cdot v_t\|_{L^2(\mathbb{R}^3)}^2
+ \mathcal{E}_X[\xi]\mathcal{E}[v](t).
\end{array}
\end{equation}

By $(\ref{Sect5_Energy_Ineq_1})+(\ref{Sect5_Energy_Ineq_2})+(\ref{Sect5_Energy_Ineq_3})$, we have
\begin{equation}\label{Sect5_Energy_Ineq_4}
\begin{array}{ll}
\mathcal{E}_X[\xi](t) \leq C_8\mathcal{E}[v](t) + C_8\mathcal{E}[\zeta]\mathcal{E}[v](t)
+ C_8\mathcal{E}[\tau v](t)\mathcal{E}[v](t) \\[6pt]\hspace{1.7cm}
+ C_8\mathcal{E}[\tau v](t)\mathcal{E}[\zeta]\mathcal{E}[v](t) + C_8\mathcal{E}_X[\xi]\mathcal{E}[v](t) \\[5pt]\hspace{1.4cm}
\leq C_8\mathcal{E}_X[\xi]\mathcal{E}[v](t)
+ C_8(1+ 2\e+\e^2)\mathcal{E}[v](t).
\end{array}
\end{equation}
for some $C_8>0$

Let $\e_2\leq \frac{1}{2C_8}$, $c_6=2C_8(1+2\e_2 +\e_2^2)\leq 2C_8+2+\frac{1}{2C_8}$.
Then $\mathcal{E}_X[\xi](t) \leq c_6\mathcal{E}[v](t)$.
Thus, Lemma $\ref{Sect5_Energy_Ineq}$ is proved.
\end{proof}

Based on the above a priori estimates, we prove not only the uniform $L^{\infty}$ bound of $\mathcal{E}[\xi,\tau v](t)$, but also the uniform bound of $\int\limits_0^{\infty}\mathcal{E}[v](s)\,\mathrm{d}s$.
\begin{lemma}\label{Sect5_Decay_Lemma}
For any given $T\in (0,+\infty),\tau\in(0,1]$, if
\begin{equation*}
\sup\limits_{0\leq t\leq T} \mathcal{E}[\xi,\tau v,\phi,\zeta](t) \leq\e,
\end{equation*}
there exists $\e_3>0$, which is independent of $(\xi_0,v_0,\phi_0)$, such that if $0<\e\ll \min\{1,\e_1,\e_2,\e_3\}$, then for $\forall t\in [0,T]$,
\begin{equation}\label{Sect5_Decay_Lemma_Eq}
\begin{array}{ll}
\mathcal{E}[\xi,\tau v](t) \leq \beta_4,\\[8pt]
\int\limits_0^{T}\mathcal{E}[v](s)\,\mathrm{d}s \leq \beta_4,
\end{array}
\end{equation}
for some $\beta_4>0$ which is independent of $\tau$.
\end{lemma}

\begin{proof}
In view of Lemmas $\ref{Sect5_Energy_Estimate_Lemma1}$ and $\ref{Sect5_Energy_Ineq}$,
we have a priori estimates as follows:
\begin{equation}\label{Sect5_Obtained_Estimates}
\left\{\begin{array}{ll}
\frac{\mathrm{d}}{\mathrm{d}t}\mathcal{E}_1[\xi](t) + \tau^2 \frac{\mathrm{d}}{\mathrm{d}t}\mathcal{E}_1[v](t)
+ 2 \mathcal{E}_1[v](t) \leq C_9\sqrt{\e}(\mathcal{E}_X[\xi](t) + \mathcal{E}[v](t)), \\[6pt]
\mathcal{E}_X[\xi](t)\leq c_6\mathcal{E}[v](t).
\end{array}\right.
\end{equation}

Plug $(\ref{Sect5_Obtained_Estimates})_2$ into $(\ref{Sect5_Obtained_Estimates})_1$, we get
\begin{equation}\label{Sect5_Decay1}
\begin{array}{ll}
\frac{\mathrm{d}}{\mathrm{d}t}\mathcal{E}_1[\xi](t) + \tau^2 \frac{\mathrm{d}}{\mathrm{d}t}\mathcal{E}_1[v](t)
+ 2\mathcal{E}_1[v](t) \leq C_9(1+c_6)\sqrt{\e}\mathcal{E}[v](t) \\[6pt]\hspace{5.5cm}

\leq \frac{C_9(1+c_6)}{c_4}\sqrt{\e}\mathcal{E}_1[v](t).
\end{array}
\end{equation}

Take $\e_3=\frac{c_4^2}{C_9^2(1+c_6)^2}$, when $\e\leq \e_3$,
\begin{equation}\label{Sect5_Decay2}
\begin{array}{ll}
\frac{\mathrm{d}}{\mathrm{d}t}\mathcal{E}_1[\xi](t) + \tau^2 \frac{\mathrm{d}}{\mathrm{d}t}\mathcal{E}_1[v](t)
+ \mathcal{E}_1[v](t) \leq 0.
\end{array}
\end{equation}

Integrate $(\ref{Sect5_Decay2})$ from $0$ to $t$, we get
\begin{equation}\label{Sect5_Decay3}
\begin{array}{ll}
\mathcal{E}_1[\xi](t) + \tau^2\mathcal{E}_1[v](t) + \int\limits_{0}^{t} \mathcal{E}_1[v](s)\,\mathrm{d}s
\leq \mathcal{E}_1[\xi_0](t) + \mathcal{E}_1[\tau v_0](t) \\[6pt]\hspace{5.4cm}

\leq C_{10}\|(\xi_0,u_0)\|_{H^4(\mathbb{R}^3)}^2,
\end{array}
\end{equation}
then
\begin{equation}\label{Sect5_Decay4}
\begin{array}{ll}
c_4\mathcal{E}[\xi](t) + c_4\mathcal{E}[\tau v](t) + c_4\int\limits_{0}^{t} \mathcal{E}[v](s)\,\mathrm{d}s
\leq C_{10}\|(\xi_0,u_0)\|_{H^4(\mathbb{R}^3)}^2.
\end{array}
\end{equation}

Thus,
\begin{equation}\label{Sect5_Decay5}
\begin{array}{ll}
\mathcal{E}[\xi,\tau v](t) + \int\limits_{0}^{t} \mathcal{E}[v](s)\,\mathrm{d}s
\leq \frac{C_{10}}{c_4}\|(\xi_0,u_0)\|_{H^4(\mathbb{R}^3)}^2 \leq \beta_4,
\end{array}
\end{equation}
where $\beta_4>0$ is independent of $\tau$, since the initial data are sufficiently small.
Thus, Lemma $\ref{Sect5_Decay_Lemma}$ is proved.
\end{proof}

The following lemma concerns the uniform bound of $\mathcal{E}[\phi](t)$. Here, the finiteness of $T$ plays a key role in the proof.
\begin{lemma}\label{Sect5_Entropy_Lemma}
For any given $T\in (0,+\infty),\tau\in(0,1]$, if
\begin{equation*}
\sup\limits_{0\leq t\leq T} \mathcal{E}[\xi,\tau v,\phi,\zeta](t) \leq\e,
\end{equation*}
where $0<\e\ll \min\{1,\e_1,\e_2,\e_3\}$, then for $\forall t\in [0,T]$,
\begin{equation}\label{Sect5_Entropy_toProve_1}
\begin{array}{ll}
\mathcal{E}[\phi](t) \leq \beta_5,
\end{array}
\end{equation}
for some $\beta_5>0$ which is independent of $\tau$.
\end{lemma}

\begin{proof}
Let $\partial_t^{\ell}\mathcal{D}^{\alpha}\phi\cdot\partial_t^{\ell}\mathcal{D}^{\alpha}(\ref{Sect2_Relaxing_Eq})_3$, where $0\leq \ell\leq 2, 0\leq \ell+|\alpha|\leq 4$, we get
\begin{equation}\label{Sect5_Entropy_1}
\begin{array}{ll}
(|\partial_t^{\ell}\mathcal{D}^{\alpha}\phi|^2)_t \\[8pt]
= -2\sum\limits_{\ell_1+|\alpha_1|>0}\partial_t^{\ell}\mathcal{D}^{\alpha}\phi
\partial_t^{\ell_1}\mathcal{D}^{\alpha_1}v\cdot\nabla(\partial_t^{\ell_2}\mathcal{D}^{\alpha_2}\phi)
-v\cdot\nabla|\partial_t^{\ell}\mathcal{D}^{\alpha}\phi|^2.\\[6pt]
\end{array}
\end{equation}

Integrate $(\ref{Sect5_Entropy_1})$ in $\mathbb{R}^3$, we have
\begin{equation}\label{Sect5_Entropy_2}
\begin{array}{ll}
\frac{\mathrm{d}}{\mathrm{d}t}\int\limits_{\mathbb{R}^3}
|\partial_t^{\ell}\mathcal{D}^{\alpha}\phi|^2 \,\mathrm{d}x
= -2\sum\limits_{\ell_1+|\alpha_1|>0}\int\limits_{\mathbb{R}^3}\partial_t^{\ell}\mathcal{D}^{\alpha}\phi
\partial_t^{\ell_1}\mathcal{D}^{\alpha_1}v\cdot\nabla(\partial_t^{\ell_2}\mathcal{D}^{\alpha_2}\phi) \,\mathrm{d}x
\\[15pt]\hspace{3.1cm}
+ \int\limits_{\mathbb{R}^3}|\partial_t^{\ell}\mathcal{D}^{\alpha}\phi|^2 \nabla\cdot v \,\mathrm{d}x
:= I_8.\\[6pt]
\end{array}
\end{equation}

When $\ell+|\alpha|\leq 4$, it is easy to check $I_8 \lem \mathcal{E}[v](t)^{\frac{1}{2}}\mathcal{E}[\phi](t)$.
Sum $\ell,\alpha$, we have
\begin{equation}\label{Sect5_Entropy_Prove_3}
\begin{array}{ll}
\frac{\mathrm{d}}{\mathrm{d}t}\mathcal{E}[\phi](t)
\leq C_{11}\mathcal{E}[v](t)^{\frac{1}{2}}\mathcal{E}[\phi](t)
\end{array}
\end{equation}

Integrate $(\ref{Sect5_Entropy_Prove_3})$ in $(0,t)$, where $t\in [0,T]$, we obtain the uniform a priori estimate for $\mathcal{E}[\phi](t)$ which is independent of $\tau$. Here, $\tau>0$ is variant, even $\tau$ approaches $0+$.
\begin{equation}\label{Sect5_Entropy_Prove_4}
\begin{array}{ll}
\mathcal{E}[\phi](t)\leq C_{12}\|\phi_0\|_{H^4(\mathbb{R}^3)}^2\exp\{\int\limits_0^{t}C_{11}\mathcal{E}[v](s)^{\frac{1}{2}}\,\mathrm{d}s\} \\[9pt]\hspace{1.1cm}
\leq C_{12}\|\phi_0\|_{H^4(\mathbb{R}^3)}^2\exp\{C_{11}T
\int\limits_0^{T}\mathcal{E}[v](s)\,\mathrm{d}s\} \\[9pt]\hspace{1.1cm}

\leq C_{12}\|\phi_0\|_{H^4(\mathbb{R}^3)}^2\exp\{C_{11}T\beta_4\}
\leq \beta_5,
\end{array}
\end{equation}
where $\beta_5>0$ is independent of $\tau$, since the initial data are sufficiently small.
Thus, Lemma $\ref{Sect5_Entropy_Lemma}$ is proved.
\end{proof}

Due to $\zeta=\varrho(\xi+\bar{p},\phi+\bar{S})-\bar{\varrho}$, we can estimate $\mathcal{E}[\zeta](t)$ in the following lemma:
\begin{lemma}\label{Sect5_Varrho_Bound_Lemma}
For any given $T\in (0,+\infty),\tau\in(0,1]$, if
\begin{equation*}
\sup\limits_{0\leq t\leq T} \mathcal{E}[\xi,\tau v,\phi,\zeta](t) \leq\e,
\end{equation*}
where $0<\e\ll \min\{1,\e_1,\e_2,\e_3\}$, then for $\forall t\in [0,T]$,
\begin{equation}\label{Sect5_Density_toProve_1}
\begin{array}{ll}
\mathcal{E}[\zeta](t)\leq \beta_6,
\end{array}
\end{equation}
for some $\beta_6>0$ which is independent of $\tau$.
\end{lemma}

\begin{proof}
\begin{equation}\label{Sect5_Density_Prove}
\begin{array}{ll}
\mathcal{E}[\zeta](t) \leq C\mathcal{E}[\xi](t) + C\mathcal{E}[\phi](t) \\[8pt]
\leq CC_{12}\|\phi_0\|_{H^4(\mathbb{R}^3)}^2 \exp\{C_{11}\beta_4 T\|(\xi_0,v_0)\|_{H^4(\mathbb{R}^3)}^2\}
+ C\beta_4\|(\xi_0,v_0)\|_{H^4(\mathbb{R}^3)}^2 \leq\beta_6,
\end{array}
\end{equation}
where $\beta_6>0$ is independent of $\tau$, since the initial data are sufficiently small.
\end{proof}

By now, we have obtained the $L^{\infty}$ bound of $\mathcal{E}[\xi,\tau v,\phi,\zeta](t)$, namely,
\begin{equation*}
\begin{array}{ll}
\mathcal{E}[\xi](t)\leq \beta_4,\ \mathcal{E}[\tau v](t)\leq \beta_4,\
\mathcal{E}[\phi](t)\leq \beta_5,\ \mathcal{E}[\zeta](t)\leq \beta_6.
\end{array}
\end{equation*}
These bounds are independent of $\tau$. Thus, as long as the initial data are sufficiently small, the a priori assumption $\sup\limits_{0\leq t\leq T}\mathcal{E}[\xi,\tau v,\phi,\zeta](t)\leq \e$ is valid.

\section{Initial Layer and Relaxation Limit of the Relaxing Equations}
This section concerns the relaxation limit of the relaxing Cauchy problem $(\ref{Sect2_Relaxing_Eq})$ and the initial layer for the velocity.

In finite time interval $[0,T]$, the bounds of $\mathcal{E}[\xi](t),\mathcal{E}[\tau v](t),
\mathcal{E}[\phi](t),\mathcal{E}[\zeta](t)$ and $\int\limits_0^T \mathcal{E}[v](s)\,\mathrm{d}s$ are uniform with respect to $\tau$, thus we can pass to the limit. The following theorem states the relaxation limit of the relaxing Cauchy problem $(\ref{Sect2_Relaxing_Eq})$.
\begin{theorem}\label{Sect6_Relaxation_Limit_Thm}
Suppose that the initial data for the relaxing Cauchy problem $(\ref{Sect2_Relaxing_Eq})$ satisfy $(p_0(x,\tau)-\bar{p},\frac{1}{k_1\tau}u_0(x,\tau),S_0(x,\tau)-\bar{S})\in H^4(\mathbb{R}^3)$,
$\inf\limits_{x\in\mathbb{R}^3} p_0(x,\tau)>0$, $\inf\limits_{x\in\mathbb{R}^3}\lim\limits_{\tau\rto 0}p_0(x,\tau)>0$, $\|(p_0(x,\tau)-\bar{p},\frac{1}{k_1\tau}u_0(x,\tau),S_0(x,\tau)-\bar{S})\|_{H^4(\mathbb{R}^3)}$ is sufficiently small. Then for any finite $T>0$, the problem $(\ref{Sect2_Relaxing_Eq})$ admits a unique solution $(\xi,v,\phi,\zeta)$ in $[0,T]$ satisfying
\begin{equation}\label{Sect6_Solution_Regularity}
\begin{array}{ll}
\partial_t^{\ell}\xi,\ \tau \partial_t^{\ell}v,\ \partial_t^{\ell}\phi,\ \partial_t^{\ell}\zeta \in L^{\infty}([0,T],H^{4-\ell} (\mathbb{R}^3)), 0\leq\ell\leq 2, \\[6pt]

\xi,\ v,\ \phi,\ \zeta \in \underset{0\leq \ell\leq 2}{\cap}H^\ell([0,T],H^{4-\ell}(\mathbb{R}^3)),
\end{array}
\end{equation}
such that as $\tau\rto 0$,
\begin{equation}\label{Sect6_Convergence_Result}
\begin{array}{ll}
\xi \rto \tilde{\xi}\ in\ C([0,T],C^{2+\mu_1}(\mathcal{K})\cap W^{3,\mu_2}(\mathcal{K})),
\mu_1\in (0,\frac{1}{2}),2\leq \mu_2<6, \\[6pt]
\phi \rto \tilde{\xi}\ in\ C([0,T],C^{2+\mu_1}(\mathcal{K})\cap W^{3,\mu_2}(\mathcal{K})),
\mu_1\in (0,\frac{1}{2}),2\leq \mu_2<6, \\[6pt]
\zeta \rto \tilde{\zeta}\ in\ C([0,T],C^{2+\mu_1}(\mathcal{K})\cap W^{3,\mu_2}(\mathcal{K})),
\mu_1\in (0,\frac{1}{2}),2\leq \mu_2<6, \\[6pt]
v \rightharpoonup \tilde{v} \ in\ \underset{0\leq \ell\leq 2}{\cap}H^\ell([0,T],H^{4-\ell}(\mathbb{R}^3)),
\end{array}
\end{equation}
for some function $(\tilde{\xi},\tilde{v},\tilde{\phi},\tilde{\zeta})$ which is a weak solution to the relaxed equations $(\ref{Sect2_Relaxed_Eq})$, where $\mathcal{K}$ is arbitrary compact subset of $\mathbb{R}^3$. $(\tilde{\xi},\tilde{v},\tilde{\phi},\tilde{\zeta})$ is the classical solution to $(\ref{Sect2_Relaxed_Eq})$, if $(\lim\limits_{\tau\rto 0}\xi_0(x,\tau),\lim\limits_{\tau\rto 0}\phi_0(x,\tau))\in H^5(\mathbb{R}^3)\times H^4(\mathbb{R}^3)$ is satisfied..
\end{theorem}

\begin{proof}
In view of Lemmas $\ref{Sect5_Decay_Lemma},\ref{Sect5_Entropy_Lemma},\ref{Sect5_Varrho_Bound_Lemma}$, we have uniform regularities $(\ref{Sect6_Solution_Regularity})$ under small data assumption, then the solution to $(\ref{Sect1_NonIsentropic_EulerEq})$, i.e., $(\hat{p},\hat{u},\hat{S},\hat{\varrho})(x,t^{\prime})=(\xi+\bar{p},k_1\tau v,\phi+\bar{S},\zeta+\bar{\varrho})(x,t)$ has uniform bounds with respect to $\tau$:
\begin{equation}\label{Sect6_Uniform_Bound}
\begin{array}{ll}
\sup\limits_{0\leq t^{\prime}\leq T/{\tau}}
\mathcal{E}[\hat{p}-\bar{p},\hat{u},\hat{S}-\bar{S},\hat{\varrho}-\bar{\varrho}](t^{\prime})
+\frac{1}{\tau}\int\limits_0^{T/{\tau}} \mathcal{E}[\hat{u}](t^{\prime})\,\mathrm{d}t^{\prime} \\[9pt]

=\sup\limits_{0\leq t\leq T}\mathcal{E}[\xi,k_1\tau v,\phi,\zeta](t)
+ k_1^2\int\limits_0^{T} \mathcal{E}[v](t)\,\mathrm{d}t \\[15pt]

\leq (2+k_1^2)\beta_4 + \beta_5+\beta_6.
\end{array}
\end{equation}

By a priori estimate $(\ref{Sect6_Uniform_Bound})$ and the local existence of $(\ref{Sect1_NonIsentropic_EulerEq})$ which is standard (see \cite{Majda_1984}), we have the existence of classical solutions to $(\ref{Sect1_NonIsentropic_EulerEq})$ in the time interval $[0,T/{\tau}]$, then the solutions of the relaxing equations $(\ref{Sect1_Relaxing_Eq})$ in $[0,T]$ can be constructed via the rescaling of variables $(\ref{Sect1_Rescaling_Variables})$.

The existence of the relaxing equations $(\ref{Sect1_Relaxing_Eq})$ in $[0,T]$ is uniform with respect to $\tau$, due to the uniform regularities $(\ref{Sect1_Uniform_Regularity})$ and the fact that $T$ is independent of $\tau$.
Thus, the uniform existence in $[0,T]$ and uniform regularities $(\ref{Sect1_Uniform_Regularity})$ for the relaxing Cauchy problem $(\ref{Sect1_Relaxing_Eq})$ are proved.

\vspace{0.3cm}
By Aubin's Lemma (see \cite{Simon_1987}):\\
$$L^{\infty}([0,T],H^4(\mathbb{R}^3))\cap H^1([0,T],H^3(\mathbb{R}^3))\hookrightarrow
C([0,T],C^{2+\mu_1}(\mathcal{K})\cap W^{3,\mu_2}(\mathcal{K})),$$
where $\mathcal{K}$ is any compact subset of $\mathbb{R}^3$, then
we have the convergence results as $(\ref{Sect6_Convergence_Result})$ stated.

For any $\varphi_1 \in H_0^1(\mathcal{K}\times[0,T))$, it follows from $(\ref{Sect2_Relaxing_Eq})_1$ that
\begin{equation}\label{Sect6_Relaxation_Xi_1}
\begin{array}{ll}
\int\limits_0^t \int\limits_{\mathbb{R}^3}
\partial_t\varphi_1\xi + \gamma k_1 v\cdot(\xi\nabla\varphi_1 +\varphi_1\nabla\xi)
+\gamma k_1 \bar{p} v\cdot\nabla\varphi_1
- k_1 \varphi_1 v\cdot\nabla \xi \,\mathrm{d}x\mathrm{d}s \\[6pt]
= -\int\limits_{\mathbb{R}^3}\xi_0(x,\tau) \varphi_1(\cdot,0)\,\mathrm{d}x.
\end{array}
\end{equation}

In view of convergence results $(\ref{Sect6_Convergence_Result})$, let $\tau\rto 0$, we have $\xi\rto\tilde{\xi}$, $\nabla\cdot\xi\rto\nabla\cdot\tilde{\xi}$, $v\rightharpoonup \tilde{v}$, then
\begin{equation}\label{Sect6_Relaxation_Xi_2}
\begin{array}{ll}
\int\limits_0^t \int\limits_{\mathbb{R}^3}
\partial_t\varphi_1\tilde{\xi} + \gamma k_1 \tilde{v}\cdot(\tilde{\xi}\nabla\varphi_1 +\varphi_1\nabla\tilde{\xi})
+\gamma k_1 \bar{p} \tilde{v}\cdot\nabla\varphi_1 - k_1 \varphi_1 \tilde{v}\cdot\nabla \tilde{\xi} \,\mathrm{d}x\mathrm{d}s
\\[6pt]
= -\int\limits_{\mathbb{R}^3}\lim\limits_{\tau\rto 0} \xi_0(x,\tau) \varphi_1(\cdot,0)\,\mathrm{d}x.
\end{array}
\end{equation}

For any $\vec{\varphi}_2 \in H_0^1(\mathcal{K}\times[0,T))^3$, it follows from $(\ref{Sect2_Relaxing_Eq})_2$ that
\begin{equation}\label{Sect6_Relaxation_V_1}
\begin{array}{ll}
\int\limits_0^t \int\limits_{\mathbb{R}^3}
- v \cdot\vec{\varphi}_2 - \frac{1}{k_1\varrho}\nabla\xi\cdot\vec{\varphi}_2
\,\mathrm{d}x\mathrm{d}s

= \tau^2\int\limits_0^t \int\limits_{\mathbb{R}^3}
v_t \cdot \vec{\varphi}_2 + k_1 v\cdot\nabla v \cdot\vec{\varphi}_2
\,\mathrm{d}x\mathrm{d}s.
\end{array}
\end{equation}

While as $\tau\rto 0$,
\begin{equation}\label{Sect6_Relaxation_V_2}
\begin{array}{ll}
[R.H.S.\ of\ (\ref{Sect6_Relaxation_V_1})]
\lem \tau^2 \|v_t\|_{L^2(\mathbb{R}^3\times[0,T))}\|\vec{\varphi}_2\|_{L^2(\mathbb{R}^3\times[0,T))} \\[6pt]\hspace{3cm}
+ \tau |\tau v|_{\infty}\|\nabla v\|_{L^2(\mathbb{R}^3\times[0,T))}
\|\vec{\varphi}_2\|_{L^2(\mathbb{R}^3\times[0,T))}
\rto 0.
\end{array}
\end{equation}

Let $\tau\rto 0$ in $(\ref{Sect6_Relaxation_V_1})$, we have
\begin{equation}\label{Sect6_Relaxation_V_3}
\begin{array}{ll}
\int\limits_0^t \int\limits_{\mathbb{R}^3}
- \tilde{v} \cdot\vec{\varphi}_2 - \frac{1}{k_1\tilde{\varrho}}\nabla\tilde{\xi}\cdot\vec{\varphi}_2
\,\mathrm{d}x\mathrm{d}s = 0.
\end{array}
\end{equation}

For any $\varphi_3 \in H_0^1(\mathcal{K}\times[0,T))$, it follows from $(\ref{Sect2_Relaxing_Eq})_3$ that
\begin{equation}\label{Sect6_Relaxation_S_1}
\begin{array}{ll}
\int\limits_0^t \int\limits_{\mathbb{R}^3}
\phi\partial_t\varphi_3 - k_1 v\cdot\nabla\phi \varphi_3
\,\mathrm{d}x\mathrm{d}s = -\int\limits_{\mathbb{R}^3} \phi_0(x,\tau)\varphi_3(\cdot,0)\,\mathrm{d}x.
\end{array}
\end{equation}

Let $\tau\rto 0$ in $(\ref{Sect6_Relaxation_S_1})$, we have
\begin{equation}\label{Sect6_Relaxation_S_2}
\begin{array}{ll}
\int\limits_0^t \int\limits_{\mathbb{R}^3}
\tilde{\phi}\partial_t\varphi_3 - k_1 \tilde{v}\cdot\nabla\tilde{\phi} \varphi_3
\,\mathrm{d}x\mathrm{d}s = -\int\limits_{\mathbb{R}^3} \lim\limits_{\tau\rto 0} \phi_0(x,\tau)\varphi_3(\cdot,0)\,\mathrm{d}x.
\end{array}
\end{equation}

While in arbitrary $\mathcal{K}\subset\mathbb{R}^3$,
\begin{equation*}
\left\{\begin{array}{ll}
\bar{p}+\xi\rto \bar{p}+\tilde{\xi}, \\[6pt]
A(\zeta+\bar{\varrho})^{\gamma}\exp\{\bar{S}+\phi\}
\rto A(\tilde{\zeta}+\bar{\varrho})^{\gamma}\exp\{\bar{S}+\tilde{\phi}\}, \\[6pt]
\bar{p}+\xi = A(\zeta+\bar{\varrho})^{\gamma}\exp\{\bar{S}+\phi\},
\end{array}\right.
\end{equation*}
imply 
\begin{equation}\label{Sect6_Relaxation_Pressure}
\begin{array}{ll}
\tilde{p}=\bar{p}+\tilde{\xi} = A(\tilde{\zeta}+\bar{\varrho})^{\gamma}\exp\{\bar{S}+\tilde{\phi}\}
= A \tilde{\varrho}^{\gamma}e^{\tilde{S}},\ in\ \mathcal{K}.
\end{array}
\end{equation}

Since $\mathcal{K}$ and test functions $\varphi_1,\vec{\varphi}_2,\varphi_3$ are arbitrary,
$(\ref{Sect6_Relaxation_Xi_2})$, $(\ref{Sect6_Relaxation_V_3})$, $(\ref{Sect6_Relaxation_S_2})$
and $(\ref{Sect6_Relaxation_Pressure})$ imply $(\tilde{\xi},\tilde{v},\tilde{\phi},\tilde{\varrho})$ is the weak solution of the relaxed equations $(\ref{Sect2_Relaxed_Eq})$ with initial data $(\lim\limits_{\tau\rto 0}\xi_0,\lim\limits_{\tau\rto 0}\phi_0)$.

As proved in Theorem $\ref{Sect4_GlobalExistence_Thm}$, Cauchy problem for the relaxed equations $(\ref{Sect2_Relaxed_Eq})$ with small data $(\lim\limits_{\tau\rto 0}\xi_0,\lim\limits_{\tau\rto 0}\phi_0)\in H^5(\mathbb{R}^3)\times H^4(\mathbb{R}^3)$ admits a unique classical solution.
By the uniqueness of the weak solutions to $(\ref{Sect2_Relaxed_Eq})$,
$(\tilde{\xi},\tilde{v},\tilde{\phi},\tilde{\varrho})$
is classical solution of the relaxed equations $(\ref{Sect2_Relaxed_Eq})$ if the data satisfy $(\lim\limits_{\tau\rto 0}\xi_0,\lim\limits_{\tau\rto 0}\phi_0)\in H^5(\mathbb{R}^3)\times H^4(\mathbb{R}^3)$. Thus, Theorem \ref{Sect6_Relaxation_Limit_Thm} is proved.
\end{proof}

\vspace{0.3cm}
Besides the weak convergence result of the velocity, we have the strong convergence of the velocity outside the initial layer for the ill-prepared data and strong convergence of the velocity near $t=0$ for the well-prepared data, as the following theorem stated.
\begin{theorem}\label{Sect6_Strong_Congvergence_Thm}
Let $(\xi,v,\phi,\zeta)$ and $(\tilde{\xi},\tilde{v},\tilde{\phi},\tilde{\zeta})$ be the solutions obtained in Theorem  $\ref{Sect6_Relaxation_Limit_Thm}$, and $\mathcal{K}$ is any compact subset of $\mathbb{R}^3$.
For the ill-prepared data, i.e., $\lim\limits_{\tau\rto 0}\left|v_0(x,\tau) + \frac{1}{k_1\varrho_0(x,\tau)}\nabla\xi_0(x,\tau)\right|_{\infty} \neq 0$,
there exists an initial layer $[0, t^{\ast}]$ with $0<t^{\ast}< C\tau^{2-\delta}$ for the velocity $u$, where $C>0$, $\delta>0$, such that as $\tau\rto 0$,
\begin{equation*}
\begin{array}{ll}
v\rto \tilde{v}\ in\ C((0,T],C^{0+\mu_1}(\mathcal{K})\cap W^{1,\mu_2}(\mathcal{K})),\mu_1\in (0,\frac{1}{2}),2\leq \mu_2<6;
\end{array}
\end{equation*}
for the well-prepared data, i.e., $\lim\limits_{\tau\rto 0}\left\|v_0(x,\tau) + \frac{1}{k_1\varrho_0(x,\tau)}\nabla\xi_0(x,\tau)\right\|_{H^2(\mathbb{R}^3)} =0$, as $\tau\rto 0$,
\begin{equation*}
\begin{array}{ll}
v\rto \tilde{v}\ in\ C([0,T],C^{0+\mu_1}(\mathcal{K})\cap W^{1,\mu_2}(\mathcal{K})),\mu_1\in (0,\frac{1}{2}),2\leq \mu_2<6.
\end{array}
\end{equation*}
\end{theorem}

\begin{proof}
Apply $\mathcal{D}^{\alpha}$ to $(\ref{Sect2_Wave_Eq})$, where $|\alpha|\leq 2$,
we have
\begin{equation}\label{Sect6_Wave_1}
\begin{array}{ll}
(\mathcal{D}^{\alpha}\eta)_t +\frac{1}{\tau^2}\mathcal{D}^{\alpha}\eta
= \mathcal{D}^{\alpha}[ v\cdot\nabla(\frac{1}{\varrho})\nabla\xi
- \frac{1}{\varrho}(\nabla v) \nabla\xi - \frac{\gamma}{\varrho}\nabla \xi \nabla\cdot v \\[10pt]\quad
- \frac{\gamma}{\varrho}p\nabla(\nabla\cdot v)
+\frac{1}{\varrho^2}(v\cdot\nabla\zeta + \varrho\nabla\cdot v)\nabla\xi]
-k_1 \sum\mathcal{D}^{\alpha_1}v\cdot\nabla\mathcal{D}^{\alpha_2}\eta.
\end{array}
\end{equation}

Let $(\ref{Sect6_Wave_1})\cdot \mathcal{D}^{\alpha}\eta$, we get
\begin{equation}\label{Sect6_Wave_2}
\begin{array}{ll}
(|\mathcal{D}^{\alpha}\eta|^2)_t +\frac{2}{\tau^2}|\mathcal{D}^{\alpha}\eta|^2
= 2\mathcal{D}^{\alpha}\eta\cdot\mathcal{D}^{\alpha}[ v\cdot\nabla(\frac{1}{\varrho})\nabla\xi
- \frac{1}{\varrho}(\nabla v) \nabla\xi - \frac{\gamma}{\varrho}\nabla \xi \nabla\cdot v \\[10pt]\quad
- \frac{\gamma}{\varrho}p\nabla(\nabla\cdot v)
+\frac{1}{\varrho^2}(v\cdot\nabla\zeta + \varrho\nabla\cdot v)\nabla\xi]
-2k_1 \sum\mathcal{D}^{\alpha_1}v\cdot\nabla\mathcal{D}^{\alpha_2}\eta \cdot \mathcal{D}^{\alpha}\eta.
\end{array}
\end{equation}

After integrating $(\ref{Sect6_Wave_2})$ in $\mathbb{R}^3$, we have
\begin{equation}\label{Sect6_Wave_3}
\begin{array}{ll}
\frac{\mathrm{d}}{\mathrm{d}t}\int\limits_{\mathbb{R}^3}|\mathcal{D}^{\alpha}\eta|^2 \,\mathrm{d}x +\frac{2}{\tau^2}\int\limits_{\mathbb{R}^3}|\mathcal{D}^{\alpha}\eta|^2 \,\mathrm{d}x \\[8pt]
= 2 \int\limits_{\mathbb{R}^3}\mathcal{D}^{\alpha}\eta\cdot\mathcal{D}^{\alpha}[ v\cdot\nabla(\frac{1}{\varrho})\nabla\xi
- \frac{1}{\varrho}(\nabla v) \nabla\xi
- \frac{\gamma}{\varrho}\nabla \xi \nabla\cdot v - \frac{\gamma}{\varrho}p\nabla(\nabla\cdot v)
\\[8pt]\quad
+\frac{1}{\varrho^2}(v\cdot\nabla\zeta + \varrho\nabla\cdot v)\nabla\xi]\,\mathrm{d}x
-k_1 \int\limits_{\mathbb{R}^3}v\cdot\nabla|\mathcal{D}^{\alpha}\eta|^2 \,\mathrm{d}x \\[8pt]\quad
-2k_1 \sum\limits_{\alpha_1>0}\ \int\limits_{\mathbb{R}^3}\mathcal{D}^{\alpha_1}v\cdot\nabla\mathcal{D}^{\alpha_2}\eta \cdot \mathcal{D}^{\alpha}\eta\,\mathrm{d}x \\[12pt]

\leq  \frac{1}{2\tau^2} \int\limits_{\mathbb{R}^3}|\mathcal{D}^{\alpha}\eta|^2 \,\mathrm{d}x
+ 2\tau^2\|\mathcal{D}^{\alpha}[ v\cdot\nabla(\frac{1}{\varrho})\nabla\xi
- \frac{1}{\varrho}(\nabla v) \nabla\xi - \frac{\gamma}{\varrho}\nabla \xi \nabla\cdot v \\[8pt]\quad
- \frac{\gamma}{\varrho}p\nabla(\nabla\cdot v)
+\frac{1}{\varrho^2}(v\cdot\nabla\zeta + \varrho\nabla\cdot v)\nabla\xi]\|_{L^2(\mathbb{R}^3)}^2
\\[8pt]\quad
+k_1 \int\limits_{\mathbb{R}^3}\nabla\cdot v|\mathcal{D}^{\alpha}\eta|^2 \,\mathrm{d}x
-2k_1 \sum\limits_{\alpha_1>0}\ \int\limits_{\mathbb{R}^3}\mathcal{D}^{\alpha_1}v\cdot\nabla\mathcal{D}^{\alpha_2}\eta \cdot \mathcal{D}^{\alpha}\eta\,\mathrm{d}x \\[12pt]

\leq  \frac{1}{2\tau^2} \int\limits_{\mathbb{R}^3}|\mathcal{D}^{\alpha}\eta|^2 \,\mathrm{d}x
+ C\|\mathcal{D}^{\alpha}\nabla\xi\|_{L^2(\mathbb{R}^3)}^2 + C\|\mathcal{D}^{\alpha}\nabla\phi\|_{L^2(\mathbb{R}^3)}^2
\\[8pt]\quad
+ C\|\mathcal{D}^{\alpha}\nabla\zeta\|_{L^2(\mathbb{R}^3)}^2
+ \frac{2}{3}[\|\tau\mathcal{D}^{\alpha}v\|_{L^6(\mathbb{R}^3)}^6 + \|\mathcal{D}^{\alpha}\nabla(\frac{1}{\varrho})\|_{L^6(\mathbb{R}^3)}^6
 \\[8pt]\quad
+ (2+\gamma)\|\nabla\nabla(\frac{1}{\varrho})\|_{L^6(\mathbb{R}^3)}^6
+ \|\tau\mathcal{D}^{\alpha}(\nabla v)\|_{L^6(\mathbb{R}^3)}^6
+ (3+\gamma)\|\mathcal{D}^{\alpha}\nabla\xi\|_{L^6(\mathbb{R}^3)}^6 \\[8pt]\quad
+ (1+\gamma)\|\tau\mathcal{D}^{\alpha}(\nabla\cdot v)\|_{L^6(\mathbb{R}^3)}^6
+\|\mathcal{D}^{\alpha}(\frac{\tau v}{\varrho^2})\|_{L^6(\mathbb{R}^3)}^6 + \|\mathcal{D}^{\alpha}\nabla\zeta\|_{L^6(\mathbb{R}^3)}^6]
\\[8pt]\quad
+C|\tau\nabla(\nabla\cdot v)|_{\infty}\|\nabla\nabla(\frac{p}{\varrho})\|_{L^2(\mathbb{R}^3)}^2
+C|\frac{p}{\varrho}|_{\infty}\|\tau\mathcal{D}^{\alpha}\nabla(\nabla\cdot v)\|_{L^2(\mathbb{R}^3)}^2 \\[8pt]\quad

+\frac{1}{4\tau^2}\int\limits_{\mathbb{R}^3}|\mathcal{D}^{\alpha}\eta|^2 \,\mathrm{d}x \cdot
4k_1 \tau|\tau\nabla\cdot v|_{\infty}
-2k_1 \sum\limits_{\alpha_1>0}4\tau|\tau\mathcal{D}^{\alpha_1}v|_{\infty}
\frac{1}{4\tau^2}\int\limits_{\mathbb{R}^3}|\nabla\mathcal{D}^{\alpha_2}\eta|^2 \,\mathrm{d}x
\\[8pt]

\leq \frac{1}{\tau^2}\int\limits_{\mathbb{R}^3}|\mathcal{D}^{\alpha}\eta|^2\,\mathrm{d}x
+ C_{13}\mathcal{E}[\xi,\tau v,\phi,\zeta](t),
\end{array}
\end{equation}
where $0<\tau< \min\{1,\tau_0\}$ is small enough such that
$$4k_1\tau|\tau \nabla\cdot v|_{\infty}+ 8k_1\sum\limits_{\alpha_1>0}\tau|\tau \mathcal{D}^{\alpha_1}v|_{\infty}\leq
C\tau \mathcal{E}[\tau v](t)^{\frac{1}{2}} \leq C\tau_0 \sqrt{\e}\leq 1.$$

Thus,
\begin{equation}\label{Sect6_Wave_4}
\begin{array}{ll}
\frac{\mathrm{d}}{\mathrm{d}t}\int\limits_{\mathbb{R}^3}|\mathcal{D}^{\alpha}\eta|^2 \,\mathrm{d}x +\frac{1}{\tau^2}\int\limits_{\mathbb{R}^3}|\mathcal{D}^{\alpha}\eta|^2 \,\mathrm{d}x
\leq C_{13}\mathcal{E}[\xi,\tau v,\phi,\zeta](t), \\[6pt]

\frac{\mathrm{d}}{\mathrm{d}t}
\left(\exp\{\frac{t}{\tau^2}\}\int\limits_{\mathbb{R}^3}|\mathcal{D}^{\alpha}\eta|^2 \,\mathrm{d}x \right)
\leq C_{13}\exp\{\frac{t}{\tau^2}\}\mathcal{E}[\xi,\tau v,\phi,\zeta](t).
\end{array}
\end{equation}

After integrating from $0$ to $t$, we get
\begin{equation}\label{Sect6_Wave_5}
\begin{array}{ll}
\int\limits_{\mathbb{R}^3}|\mathcal{D}^{\alpha}\eta|^2 \,\mathrm{d}x
\leq \exp\{-\frac{t}{\tau^2}\} \int\limits_{\mathbb{R}^3}|\mathcal{D}^{\alpha}\eta_0|^2 \,\mathrm{d}x
+ C_{13}\int\limits_0^t \exp\{-\frac{t-s}{\tau^2}\}
\mathcal{E}[\xi,\tau v,\phi,\zeta](s)\,\mathrm{d}s \\[6pt]\hspace{1.9cm}

\leq \exp\{-\frac{t}{\tau^2}\} \int\limits_{\mathbb{R}^3}|\mathcal{D}^{\alpha}\eta_0|^2 \,\mathrm{d}x
+ C_{13}\e\int\limits_0^t \exp\{-\frac{t-s}{\tau^2}\}\,\mathrm{d}s \\[6pt]\hspace{1.9cm}

\leq \exp\{-\frac{t}{\tau^2}\} \int\limits_{\mathbb{R}^3}|\mathcal{D}^{\alpha}\eta_0|^2 \,\mathrm{d}x
+ C_{13}\e \tau^2.
\end{array}
\end{equation}

For any fixed $t_{\ast}\in (0,T)$, when $t\geq t_{\ast}$, we have
\begin{equation}\label{Sect6_Wave_6}
\begin{array}{ll}
\int\limits_{\mathbb{R}^3}|\mathcal{D}^{\alpha}\eta|^2 \,\mathrm{d}x
\leq \exp\{-\frac{t_{\ast}}{\tau^2}\} \int\limits_{\mathbb{R}^3}|\mathcal{D}^{\alpha}\eta_0|^2 \,\mathrm{d}x
+ C_{13}\e \tau^2.
\end{array}
\end{equation}

While $\exp\{-\frac{t_{\ast}}{\tau^2}\}\rto 0$, as $\tau\rto 0$.
Sum $\alpha$ and let $\tau\rto 0$, we have
\begin{equation}\label{Sect6_Wave_7}
\begin{array}{ll}
\|\eta\|_{L^{\infty}([t_{\ast},T],H^2(\mathbb{R}^3))}^2 \leq C\sum\limits_{|\alpha|\leq 2}\
\sup\limits_{t\in[t_{\ast},T]}\ \int\limits_{\mathbb{R}^3}|\mathcal{D}^{\alpha}\eta|^2 \,\mathrm{d}x \rto 0.
\end{array}
\end{equation}

It follows from $(\ref{Sect2_Eta_Velocity})$ that
\begin{equation}\label{Sect6_Wave_8}
\begin{array}{ll}
\int\limits_0^T \|\eta_t\|_{H^1(\mathbb{R}^3)}^2 \,\mathrm{d}s
\lem \tau^4\int\limits_0^T \|v_{tt}\|_{H^1(\mathbb{R}^3)}^2 \,\mathrm{d}s
+\tau^2|\tau v_t|_{\infty}^2\int\limits_0^T \|\nabla v\|_{H^1(\mathbb{R}^3)}^2 \,\mathrm{d}s \\[6pt]\hspace{2.7cm}
+\tau^2|\tau v|_{\infty}^2\int\limits_0^T \|\nabla v_t\|_{H^1(\mathbb{R}^3)}^2 \,\mathrm{d}s
\\[6pt]\hspace{2.35cm}
\lem \tau^4\int\limits_0^T \|v_{tt}\|_{H^1(\mathbb{R}^3)}^2 \,\mathrm{d}s
+\tau^2\|\tau v_t\|_{H^2(\mathbb{R}^3)}^2\int\limits_0^T \|\nabla v\|_{H^1(\mathbb{R}^3)}^2 \,\mathrm{d}s \\[6pt]\hspace{2.7cm}
+\tau^2\|\tau v\|_{H^2(\mathbb{R}^3)}^2\int\limits_0^T \|\nabla v_t\|_{H^1(\mathbb{R}^3)}^2 \,\mathrm{d}s
\\[6pt]\hspace{2.35cm}
\lem C\tau^2.
\end{array}
\end{equation}

Then by Aubin's Lemma (see \cite{Simon_1987}):\\
$$L^{\infty}([t_{\ast},T],H^2(\mathbb{R}^3))\cap H^1([t_{\ast},T],H^1(\mathbb{R}^3))\hookrightarrow
C([t_{\ast},T],C^{0+\mu_1}(\mathcal{K})\cap W^{1,\mu_2}(\mathcal{K})),$$
we have
\begin{equation}\label{Sect6_Wave_9}
\begin{array}{ll}
\|\eta\|_{C([t_{\ast},T],C^{0+\mu_1}(\mathcal{K})\cap W^{1,\mu_2}(\mathcal{K}))}\lem C\tau^2.
\end{array}
\end{equation}

Since
\begin{equation}\label{Sect6_Wave_10}
\begin{array}{ll}
v\rto -\frac{1}{k_1\varrho}\nabla\xi\ in\ C([t_{\ast},T],C^{0+\mu_1}(\mathcal{K})\cap W^{1,\mu_2}(\mathcal{K})), \\[6pt]
-\frac{1}{k_1\varrho}\nabla\xi \rto -\frac{1}{k_1\tilde{\varrho}}\nabla\tilde{\xi} = \tilde{v}
\ in\ C([t_{\ast},T],C^{1+\mu_1}(\mathcal{K})\cap W^{2,\mu_2}(\mathcal{K})),
\end{array}
\end{equation}
we have $v\rto \tilde{v}\ in\ C([t_{\ast},T],C^{0+\mu_1}(\mathcal{K})\cap W^{1,\mu_2}(\mathcal{K}))$. Due to the arbitrariness of $t_{\ast}$, we obtain the following convergence for the ill-prepared data:
\begin{equation}\label{Sect6_Convergence1}
\begin{array}{ll}
v\rto \tilde{v}\ in\ C((0,T],C^{0+\mu_1}(\mathcal{K})\cap W^{1,\mu_2}(\mathcal{K})).
\end{array}
\end{equation}

More precisely, for any $t>0$, in order to have $\lim\limits_{\tau\rto 0}\exp\{-\frac{t}{\tau^2}\}=0$ in $(\ref{Sect6_Wave_5})$, it requires $t\geq C\tau^{2-\delta}$, where $C>0,\delta>0$. Thus, the width of the initial layer $[0,t^{\ast}]$ has a upper bound $0<t^{\ast}<C\tau^{2-\delta}$.

While, for the initial data are well-prepared, i.e., $\eta(x,0)=O(\tau)$, it follows from $(\ref{Sect6_Wave_5})$ that
\begin{equation}\label{Sect6_Wave_11}
\begin{array}{ll}
\int\limits_{\mathbb{R}^3}|\mathcal{D}^{\alpha}\eta|^2 \,\mathrm{d}x

\leq \exp\{-\frac{t}{\tau^2}\} \int\limits_{\mathbb{R}^3}|\mathcal{D}^{\alpha}\eta_0|^2 \,\mathrm{d}x
+ C_{13}\e \tau^2 \\[5pt]\hspace{1.9cm}

\leq \int\limits_{\mathbb{R}^3}|\mathcal{D}^{\alpha}\eta_0|^2 \,\mathrm{d}x
+ C_{13}\e \tau^2 \leq C\tau^2.
\end{array}
\end{equation}

Sum $\alpha$, we have that in $[0,T]$, as $\tau\rto 0$,
\begin{equation}\label{Sect6_Wave_12}
\begin{array}{ll}
\|\eta\|_{L^{\infty}([0,T],H^2(\mathbb{R}^3))}^2
\leq C\sum\limits_{|\alpha|\leq 2}
\int\limits_{\mathbb{R}^3}\ |\mathcal{D}^{\alpha}\eta|^2 \,\mathrm{d}x \rto 0.
\end{array}
\end{equation}

By $(\ref{Sect6_Wave_8})$ and $(\ref{Sect6_Wave_12})$, we have
\begin{equation}\label{Sect6_Wave_13}
\begin{array}{ll}
\|\eta\|_{C([0,T],C^{0+\mu_1}(\mathcal{K})\cap W^{1,\mu_2}(\mathcal{K}))}\lem C\tau^2.
\end{array}
\end{equation}

Similarly, we obtain the following convergence for the well-prepared data:
\begin{equation}\label{Sect6_Convergence2}
\begin{array}{ll}
v\rto \tilde{v}\ in\ C([0,T],C^{0+\mu_1}(\mathcal{K})\cap W^{1,\mu_2}(\mathcal{K})).
\end{array}
\end{equation}

Thus, Theorem $\ref{Sect6_Strong_Congvergence_Thm}$ is proved.
\end{proof}

\begin{remark}\label{Sect6_Higher_Regularity_Remark}
Assume $k\geq 5$ is an integer, $\mu_1\in (0,\frac{1}{2}),2\leq \mu_2<6$, the initial data for the relaxing equations $(\ref{Sect2_Relaxing_Eq})$ satisfy $(p_0(x,\tau),u_0(x,\tau),S_0(x,\tau))\in H^k(\mathbb{R}^3)$, similar to Theorem
$\ref{Sect6_Relaxation_Limit_Thm},\ref{Sect6_Strong_Congvergence_Thm}$, we have that for any finite $T>0$, the problem $(\ref{Sect2_Relaxing_Eq})$ admits a unique solution $(\xi,v,\phi,\zeta)$ in $[0,T]$ such that
as $\tau\rto 0$,
\begin{equation*}
\begin{array}{ll}
(\xi,\phi,\zeta) \rto (\tilde{\xi},\tilde{\xi},\tilde{\zeta})\ in\ C([0,T],C^{k-2+\mu_1}(\mathcal{K})\cap W^{k-1,\mu_2}(\mathcal{K})), \\[6pt]
v \rightharpoonup \tilde{v} \ in\ \underset{0\leq \ell\leq 2}{\cap}H^\ell([0,T],H^{k-\ell}(\mathbb{R}^3)),
\end{array}
\end{equation*}

For the ill-prepared data, as $\tau\rto 0$,
\begin{equation*}
\begin{array}{ll}
v\rto \tilde{v}\ in\ C((0,T],C^{k-4+\mu_1}(\mathcal{K})\cap W^{k-3,\mu_2}(\mathcal{K}));
\end{array}
\end{equation*}
for the well-prepared data, as $\tau\rto 0$,
\begin{equation*}
\begin{array}{ll}
v\rto \tilde{v}\ in\ C([0,T],C^{k-4+\mu_1}(\mathcal{K})\cap W^{k-3,\mu_2}(\mathcal{K})).
\end{array}
\end{equation*}
\end{remark}

\section{Extensions}
In this section, we extend the results for the relaxing Cauchy problem $(\ref{Sect1_Relaxing_Eq})$ to
the relaxing equations in periodic domains.

Gagliado-Nirenberg type inequalities $(\ref{Sect2_Sobolev_Ineq})_2$ are used in estimates in $\mathbb{R}^3$, while
the bounded size of $\mathbb{T}^3$ allows the inequality $\|\cdot\|_{L^4(\mathbb{T}^3)}\lem \|\cdot\|_{H^1(\mathbb{T}^3)}$, thus the regularity index of the relaxed equations $(\ref{Sect1_Relaxed_Eq})$ in $\mathbb{T}^3$ can be one order lower than that of Cauchy problem. The existence in $[0,T]$ of classical solutions to $(\ref{Sect1_Relaxed_Eq})$ in periodic domains is stated as follows:
\begin{theorem}\label{Sect7_Relaxed_GlobalExistence_Thm}
$(Existence\ in\ [0,T])$\\[6pt]
Suppose $(\lim\limits_{\tau\rto 0}\xi_0(x,\tau),\lim\limits_{\tau\rto 0}\phi_0(x,\tau))\in H^4(\mathbb{T}^3)\times H^3(\mathbb{T}^3)$, $\inf\limits_{x\in\mathbb{T}^3}\lim\limits_{\tau\rto 0}p_0(x,\tau)>0$.
There exists a sufficiently small number $\delta_2>0$, such that if $\|\lim\limits_{\tau\rto 0}\xi_0(x,\tau)\|_{H^4(\mathbb{T}^3)}+\|\lim\limits_{\tau\rto 0}\phi_0(x,\tau)\|_{H^3(\mathbb{T}^3)}\leq \delta_2$, then the relaxed equations $(\ref{Sect2_Relaxed_Eq})$ in $\mathbb{T}^3$ admits a unique classical solution $(\xi,v,\phi,\zeta)$ satisfying
\begin{equation}\label{Sect7_Relaxed_Global_Regularity}
\begin{array}{ll}
\xi\in \underset{0\leq \ell\leq 1}{\cap}C^{\ell}([0,T],H^{4-\ell}(\mathbb{T}^3)), \\[8pt]
(v,\phi,\zeta)\in \underset{0\leq \ell\leq 1}{\cap}C^{\ell}([0,T],H^{3-\ell}(\mathbb{T}^3)^3), \\[8pt]
\nabla\cdot v,\ \triangle\xi \in C(\mathbb{T}^3\times[0,T]).
\end{array}
\end{equation}
\end{theorem}

The energy quantities are defined in $(\ref{Sect2_Energy_Define})$ after replacing $\mathbb{R}^3$ with $\mathbb{T}^3$. Due to the convenience of periodic boundary conditions, it is easy to extend the a priori estimates for $\mathbb{R}^3$ to periodic domains $\mathbb{T}^3$. Thus, we have the following lemma:
\begin{lemma}\label{Sect7_Bounds_Lemma}
For any given $T\in (0,+\infty),\tau\in(0,1]$, if
\begin{equation*}
\sup\limits_{0\leq t\leq T} \mathcal{E}[\xi,\tau v,\phi,\zeta](t) \leq\e,
\end{equation*}
where $\e>0$ is sufficiently small, then for $\forall t\in [0,T]$,
\begin{equation}\label{Sect5_Bounds_toProve_1}
\begin{array}{ll}
\mathcal{E}[\xi,\tau v,\phi,\zeta](t) \leq \beta_7, \\[9pt]
\int\limits_{0}^{T} \mathcal{E}[v](s)\,\mathrm{d}s \leq \beta_7,
\end{array}
\end{equation}
for some $\beta_7>0$ which is independent of $\tau$.
\end{lemma}

In finite time interval $[0,T]$, all these bounds are uniform with respect to $\tau$, thus we can pass to the limit. The following theorem states the relaxation limit of the relaxing equations $(\ref{Sect2_Relaxing_Eq})$ in periodic domains.
\begin{theorem}\label{Sect7_Relaxation_Limit_Thm}
Suppose that the initial data for the relaxing equations $(\ref{Sect2_Relaxing_Eq})$ satisfy $(p_0(x,\tau)-\bar{p},\frac{1}{k_1\tau}u_0(x,\tau),S_0(x,\tau)-\bar{S})\in H^4(\mathbb{T}^3)$.
Moreover, assume that $\inf\limits_{x\in\mathbb{T}^3} p_0(x,\tau)>0$, $\inf\limits_{x\in\mathbb{T}^3}\lim\limits_{\tau\rto 0}p_0(x,\tau)>0$, $\|(p_0(x,\tau)-\bar{p},\frac{1}{k_1}u_0(x,\tau),S_0(x,\tau)-\bar{S})\|_{H^4(\mathbb{T}^3)}$ is sufficiently small for some constants $\bar{p}>0$ and $\bar{S}$. Then for any finite $T>0$, the problem $(\ref{Sect2_Relaxing_Eq})$ admits a unique solution $(\xi,v,\phi,\zeta)$ in $[0,T]$ satisfying
\begin{equation}\label{Sect7_Solution_Regularity}
\begin{array}{ll}
\partial_t^{\ell}\xi,\ \tau \partial_t^{\ell}v,\ \partial_t^{\ell}\phi,\ \partial_t^{\ell}\zeta \in L^{\infty}([0,T],H^{4-\ell} (\mathbb{T}^3)), 0\leq\ell\leq 2, \\[6pt]

\xi,\ v,\ \phi,\ \zeta \in \underset{0\leq \ell\leq 2}{\cap}H^\ell([0,T],H^{4-\ell}(\mathbb{T}^3)),
\end{array}
\end{equation}
such that as $\tau\rto 0$,
\begin{equation}\label{Sect7_Convergence_Result}
\begin{array}{ll}
\xi \rto \tilde{\xi}\ in\ C([0,T],C^{2+\mu_1}(\mathbb{T}^3)\cap W^{3,\mu_2}(\mathbb{T}^3)),
\mu_1\in (0,\frac{1}{2}),2\leq \mu_2<6, \\[6pt]
\phi \rto \tilde{\xi}\ in\ C([0,T],C^{2+\mu_1}(\mathbb{T}^3)\cap W^{3,\mu_2}(\mathbb{T}^3)),
\mu_1\in (0,\frac{1}{2}),2\leq \mu_2<6, \\[6pt]
\zeta \rto \tilde{\zeta}\ in\ C([0,T],C^{2+\mu_1}(\mathbb{T}^3)\cap W^{3,\mu_2}(\mathbb{T}^3)),
\mu_1\in (0,\frac{1}{2}),2\leq \mu_2<6, \\[6pt]
v \rightharpoonup \tilde{v} \ in\ \underset{0\leq \ell\leq 2}{\cap}H^\ell([0,T],H^{4-\ell}(\mathbb{T}^3)),
\end{array}
\end{equation}
for some function $(\tilde{\xi},\tilde{v},\tilde{\phi},\tilde{\zeta})$ which is a weak solution to $(\ref{Sect2_Relaxed_Eq})$ with the data $(\lim\limits_{\tau\rto 0}\xi_0(x,\tau),
\lim\limits_{\tau\rto 0}\phi_0(x,\tau))$. $(\tilde{\xi},\tilde{v},\tilde{\phi},\tilde{\zeta})$ is the classical solution to $(\ref{Sect2_Relaxed_Eq})$, if $(\lim\limits_{\tau\rto 0}\xi_0(x,\tau),
\lim\limits_{\tau\rto 0}\phi_0(x,\tau))\in H^4(\mathbb{T}^3)\times H^3(\mathbb{T}^3)$ is satisfied. 
\end{theorem}

\begin{proof}
In view of Lemmas $\ref{Sect7_Bounds_Lemma}$, we have the uniform regularities $(\ref{Sect7_Solution_Regularity})$.
Similar to the argument of Theorem $\ref{Sect6_Relaxation_Limit_Thm}$, we have
the uniform existence of $(\ref{Sect1_Relaxing_Eq})$ in $[0,T]$.

In view of the regularities in $(\ref{Sect7_Solution_Regularity})$ and Aubin's Lemma (see \cite{Simon_1987}):\\
$$L^{\infty}([0,T],H^4(\mathbb{T}^3))\cap H^1([0,T],H^3(\mathbb{T}^3))\hookrightarrow
C([0,T],C^{2+\mu_1}(\mathbb{T}^3)\cap W^{3,\mu_2}(\mathbb{T}^3)),$$
so we can obtained the convergence results as $(\ref{Sect7_Convergence_Result})$ stated.

For any $\varphi_1,\varphi_3 \in H^1(\mathbb{T}^3\times[0,T)),\vec{\varphi}_2 \in H^1(\mathbb{T}^3\times[0,T))^3$,
similar to the argument of Theorem $\ref{Sect6_Relaxation_Limit_Thm}$, we have
$\nabla\cdot\xi\rto\nabla\cdot\tilde{\xi}$, $v\rightharpoonup \tilde{v}$, then
\begin{equation}\label{Sect7_Relaxation_Xi}
\begin{array}{ll}
\int\limits_0^t \int\limits_{\mathbb{T}^3}
\partial_t\varphi_1\tilde{\xi}+ \gamma k_1 \tilde{v}\cdot(\tilde{\xi}\nabla\varphi_1 +\varphi_1\nabla\tilde{\xi})
+ \gamma k_1 \tilde{v}\cdot\nabla\varphi_1 - k_1 \varphi_1 \tilde{v}\cdot\nabla \tilde{\xi} \,\mathrm{d}x\mathrm{d}s
\\[10pt]
= -\int\limits_{\mathbb{T}^3}\lim\limits_{\tau\rto 0} \xi_0(x,\tau) \varphi_1(\cdot,0)\,\mathrm{d}x, \\[9pt]

\int\limits_0^t \int\limits_{\mathbb{T}^3}
- \tilde{v} \cdot\vec{\varphi}_2 - \frac{1}{k_1\tilde{\varrho}}\nabla\tilde{\xi}\cdot\vec{\varphi}_2
\,\mathrm{d}x\mathrm{d}s = 0, \\[9pt]

\int\limits_0^t \int\limits_{\mathbb{T}^3}
\tilde{\phi}\partial_t\varphi_3 - k_1 \tilde{v}\cdot\nabla\tilde{\phi} \varphi_3
\,\mathrm{d}x\mathrm{d}s = -\int\limits_{\mathbb{T}^3} \lim\limits_{\tau\rto 0} \phi_0(x,\tau)\varphi_3(\cdot,0)\,\mathrm{d}x, \\[9pt]

\tilde{p}= A \tilde{\varrho}^{\gamma}e^{\tilde{S}}.
\end{array}
\end{equation}

$(\ref{Sect7_Relaxation_Xi})$ implies $(\tilde{\xi},\tilde{v},\tilde{\phi},\tilde{\varrho})$ is the weak solution of the relaxed equations $(\ref{Sect2_Relaxed_Eq})$ in $\mathbb{T}^3$ with initial data $(\lim\limits_{\tau\rto 0}\xi_0,\lim\limits_{\tau\rto 0}\phi_0)$.

As Theorem $\ref{Sect7_Relaxed_GlobalExistence_Thm}$ stated, the problem $(\ref{Sect2_Relaxed_Eq})$ in $\mathbb{T}^3$ with data $(\lim\limits_{\tau\rto 0}\xi_0,\lim\limits_{\tau\rto 0}\phi_0) \in H^4(\mathbb{T}^3)\times H^3(\mathbb{T}^3)$ admits a unique classical solution.
By the uniqueness of the weak solutions to $(\ref{Sect2_Relaxed_Eq})$,
$(\tilde{\xi},\tilde{v},\tilde{\phi},\tilde{\varrho})$
is classical solution of the relaxed equations $(\ref{Sect2_Relaxed_Eq})$ if the data satisfy 
$(\lim\limits_{\tau\rto 0}\xi_0,\lim\limits_{\tau\rto 0}\phi_0) \in H^4(\mathbb{T}^3)\times H^3(\mathbb{T}^3)$. Thus, Theorem \ref{Sect7_Relaxation_Limit_Thm} is proved.
\end{proof}

Besides the weak convergence result of the velocity, we have the strong convergence result of the velocity outside the initial layer for the ill-prepared data and the strong convergence result of the velocity near $t=0$ for the well-prepared data, as the following theorem stated.
\begin{theorem}\label{Sect7_Strong_Congvergence_Thm}
Let $(\xi,v,\phi,\zeta)$ and $(\tilde{\xi},\tilde{v},\tilde{\phi},\tilde{\zeta})$ be the solutions obtained in Theorem  $\ref{Sect7_Relaxation_Limit_Thm}$.
For the ill-prepared data, i.e., $\lim\limits_{\tau\rto 0}\left|v_0(x,\tau) + \frac{1}{k_1\varrho_0(x,\tau)}\nabla\xi_0(x,\tau)\right|_{\infty} \neq 0$,
there exists an initial layer $[0, t^{\ast}]$ with $0<t^{\ast}< C\tau^{2-\delta}$ for the velocity $u$, where $C>0$, $\delta>0$, such that as $\tau\rto 0$,
\begin{equation*}
\begin{array}{ll}
v\rto \tilde{v}\ in\ C((0,T],C^{0+\mu_1}(\mathbb{T}^3)\cap W^{1,\mu_2}(\mathbb{T}^3)),\mu_1\in (0,\frac{1}{2}),2\leq \mu_2<6;
\end{array}
\end{equation*}
for the well-prepared data, i.e., $\lim\limits_{\tau\rto 0}\left\|v_0(x,\tau) + \frac{1}{k_1\varrho_0(x,\tau)}\nabla\xi_0(x,\tau)\right\|_{H^2(\mathbb{T}^3)} =0$, as $\tau\rto 0$,
\begin{equation*}
\begin{array}{ll}
v\rto \tilde{v}\ in\ C([0,T],C^{0+\mu_1}(\mathbb{T}^3)\cap W^{1,\mu_2}(\mathbb{T}^3)),\mu_1\in (0,\frac{1}{2}),2\leq \mu_2<6.
\end{array}
\end{equation*}
\end{theorem}

\begin{proof}
After integrating $(\ref{Sect6_Wave_2})$ in $\mathbb{T}^3$, similar to the estimate $(\ref{Sect6_Wave_3})$, we have
\begin{equation}\label{Sect7_Wave_3}
\begin{array}{ll}
\frac{\mathrm{d}}{\mathrm{d}t}\int\limits_{\mathbb{T}^3}|\mathcal{D}^{\alpha}\eta|^2 \,\mathrm{d}x +\frac{1}{\tau^2}\int\limits_{\mathbb{T}^3}|\mathcal{D}^{\alpha}\eta|^2 \,\mathrm{d}x
\leq C_{14}\mathcal{E}[\xi,\tau v,\phi,\zeta](t),
\end{array}
\end{equation}
where $0<\tau< \min\{1,\tau_0\}$ is sufficiently small.

Thus, similar to the estimate $(\ref{Sect6_Wave_5})$, we have
\begin{equation}\label{Sect7_Wave_5}
\begin{array}{ll}
\int\limits_{\mathbb{T}^3}|\mathcal{D}^{\alpha}\eta|^2 \,\mathrm{d}x
\leq \exp\{-\frac{t}{\tau^2}\} \int\limits_{\mathbb{T}^3}|\mathcal{D}^{\alpha}\eta_0|^2 \,\mathrm{d}x
+ C_{14}\e \tau^2.
\end{array}
\end{equation}

For any fixed $t_{\ast}\in (0,T)$, when $t\geq t_{\ast}$, we have
\begin{equation}\label{Sect7_Wave_6}
\begin{array}{ll}
\int\limits_{\mathbb{T}^3}|\mathcal{D}^{\alpha}\eta|^2 \,\mathrm{d}x
\leq \exp\{-\frac{t_{\ast}}{\tau^2}\} \int\limits_{\mathbb{T}^3}|\mathcal{D}^{\alpha}\eta_0|^2 \,\mathrm{d}x
+ C_{14}\e \tau^2.
\end{array}
\end{equation}

While $\exp\{-\frac{t_{\ast}}{\tau^2}\}\rto 0$, as $\tau\rto 0$.
Sum $\alpha$ and let $\tau\rto 0$, we have
\begin{equation}\label{Sect7_Wave_7}
\begin{array}{ll}
\|\eta\|_{L^{\infty}([t_{\ast},T],H^2(\mathbb{T}^3))}^2 \leq C\sum\limits_{|\alpha|\leq 2}\
\sup\limits_{t\in[t_{\ast},T]}\ \int\limits_{\mathbb{T}^3}|\mathcal{D}^{\alpha}\eta|^2 \,\mathrm{d}x \rto 0.
\end{array}
\end{equation}

It follows from $(\ref{Sect2_Eta_Velocity})$ that
\begin{equation}\label{Sect7_Wave_8}
\begin{array}{ll}
\int\limits_0^T \|\eta_t\|_{H^1(\mathbb{T}^3)}^2 \,\mathrm{d}s
\lem \tau^4\int\limits_0^T \|v_{tt}\|_{H^1(\mathbb{T}^3)}^2 \,\mathrm{d}s
+\tau^2\|\tau v_t\|_{H^3(\mathbb{T}^3)}^2\int\limits_0^T \|\nabla v\|_{H^1(\mathbb{T}^3)}^2 \,\mathrm{d}s \\[6pt]\hspace{2.7cm}
+\tau^2\|\tau v\|_{H^3(\mathbb{T}^3)}^2\int\limits_0^T \|\nabla v_t\|_{H^1(\mathbb{T}^3)}^2 \,\mathrm{d}s
\lem C\tau^2.
\end{array}
\end{equation}

Then by Aubin's Lemma (see \cite{Simon_1987}):\\
$$L^{\infty}([t_{\ast},T],H^2(\mathbb{T}^3))\cap H^1([t_{\ast},T],H^1(\mathbb{T}^3))\hookrightarrow
C([t_{\ast},T],C^{0+\mu_1}(\mathbb{T}^3)\cap W^{1,\mu_2}(\mathbb{T}^3)),$$
we have
\begin{equation}\label{Sect7_Wave_9}
\begin{array}{ll}
\|\eta\|_{C([t_{\ast},T],C^{0+\mu_1}(\mathbb{T}^3)\cap W^{1,\mu_2}(\mathbb{T}^3))}\lem C\tau^2.
\end{array}
\end{equation}

Therefore, $v\rto \tilde{v}\ in\ C((0,T],C^{0+\mu_1}(\mathbb{T}^3)\cap W^{1,\mu_2}(\mathbb{T}^3))$ due to the arbitrariness of $t_{\ast}$ and the following convergence results:
\begin{equation}\label{Sect7_Wave_10}
\begin{array}{ll}
v\rto -\frac{1}{k_1\varrho}\nabla\xi\ in\ C([t_{\ast},T],C^{0+\mu_1}(\mathbb{T}^3)\cap W^{1,\mu_2}(\mathbb{T}^3)), \\[6pt]
-\frac{1}{k_1\varrho}\nabla\xi \rto -\frac{1}{k_1\tilde{\varrho}}\nabla\tilde{\xi} = \tilde{v}
\ in\ C([t_{\ast},T],C^{1+\mu_1}(\mathbb{T}^3)\cap W^{2,\mu_2}(\mathbb{T}^3)).
\end{array}
\end{equation}

More precisely, for any $t>0$, in order to have $\lim\limits_{\tau\rto 0}\exp\{-\frac{t}{\tau^2}\}=0$ in $(\ref{Sect7_Wave_5})$, it requires $t\geq C\tau^{2-\delta}$, where $C>0,\delta>0$. Thus, the width of the initial layer $[0,t^{\ast}]$ has a upper bound $0<t^{\ast}<C\tau^{2-\delta}$.

While, for the initial data are well-prepared, i.e., $\eta(x,0)=O(\tau)$, it follows from $(\ref{Sect7_Wave_5})$ that
\begin{equation}\label{Sect7_Wave_11}
\begin{array}{ll}
\int\limits_{\mathbb{T}^3}|\mathcal{D}^{\alpha}\eta|^2 \,\mathrm{d}x

\leq \exp\{-\frac{t}{\tau^2}\} \int\limits_{\mathbb{T}^3}|\mathcal{D}^{\alpha}\eta_0|^2 \,\mathrm{d}x
+ C_{14}\e \tau^2 \\[5pt]\hspace{1.9cm}

\leq \int\limits_{\mathbb{T}^3}|\mathcal{D}^{\alpha}\eta_0|^2 \,\mathrm{d}x
+ C_{14}\e \tau^2 \leq C\tau^2.
\end{array}
\end{equation}

Sum $\alpha$, we have that in $[0,T]$, as $\tau\rto 0$,
\begin{equation}\label{Sect7_Wave_12}
\begin{array}{ll}
\|\eta\|_{L^{\infty}([0,T],H^2(\mathbb{T}^3))}^2
\leq C\sum\limits_{|\alpha|\leq 2}
\int\limits_{\mathbb{T}^3}|\mathcal{D}^{\alpha}\eta|^2 \,\mathrm{d}x \rto 0.
\end{array}
\end{equation}

By $(\ref{Sect7_Wave_8})$ and $(\ref{Sect7_Wave_12})$, we have
\begin{equation}\label{Sect7_Wave_13}
\begin{array}{ll}
\|\eta\|_{C([0,T],C^{0+\mu_1}(\mathbb{T}^3)\cap W^{1,\mu_2}(\mathbb{T}^3))}\lem C\tau^2.
\end{array}
\end{equation}

Similarly, $v\rto \tilde{v}\ in\ C([0,T],C^{0+\mu_1}(\mathbb{T}^3)\cap W^{1,\mu_2}(\mathbb{T}^3))$.
Thus, Theorem $\ref{Sect7_Strong_Congvergence_Thm}$ is proved.
\end{proof}

\bibliographystyle{siam}
\addcontentsline{toc}{section}{References}
\bibliography{FuzhouWu_EulerEq_Relaxation}

\end{document}